\documentclass[a4paper,11pt]{article}
\usepackage[utf8]{inputenc}

\usepackage{amsthm,amsmath,amssymb,amsfonts,amscd,tikz, extarrows}
\usepackage{color}
\usepackage[pdftex, bookmarksnumbered=true]{hyperref}
\usepackage{graphicx}
\usepackage{wrapfig}
\usepackage{pgfplots}
\usepackage{enumerate}
\usepackage{mathtools}

\newtheorem{thm}{Theorem}[section]
\newtheorem{prop}{Proposition}[section]
\newtheorem{lemma}[prop]{Lemma}
\newtheorem{conj}{Conjecture}[section]
\newtheorem{corol}[prop]{Corollary}

\theoremstyle{definition}
\newtheorem{defin}[prop]{Definition}
\newtheorem{Example}[prop]{Example}

\newtheorem{Remark}[prop]{Remark}

\numberwithin{equation}{section}
\usetikzlibrary{arrows,positioning,cd,calc,patterns}
\newcommand{\bc}{\mathbf{S}}
\newcommand{\DI}{U_{q_1,q_2}(\ddot{\mathfrak{gl}}_1)}
\newcommand{\DIp}{U_{q_1,q_2}(\ddot{\mathfrak{gl}}_1)^+}
\newcommand{\DIm}{U_{q_1,q_2}(\ddot{\mathfrak{gl}}_1)^-}
\newcommand{\siw}[1]{  \Lambda_{v, #1}^{^{\! \infty/2}} \,   }
\newcommand{\sip}{  \Lambda_{v}^{^{\! \infty/2}} \,   }
\newcommand{\SHp}{S\mathcal{H}_N^{[+]}}
\newcommand{\SHm}{S\mathcal{H}_N^{[-]}}
\newcommand{\SHpm}{S\mathcal{H}_N^{\pm}}
\newcommand{\st}{\omega}
\newcommand{\hphi}{\hat{\Phi}}
\newcommand{\hpsi}{\hat{\Psi}}
\newcommand{\Tphi}{\tilde{\Phi}}
\newcommand{\tpsi}{\tilde{\Psi}}
\newcommand{\AB}{\mathbf{A}}
\newcommand{\BB}{\mathbf{B}}
\newcommand{\Ch}{\mathbf{C}}
\newcommand{\lt}{\mathtt{l}^{(\lambda)}}
\newcommand{\ltt}{\mathtt{l}^{(\tilde{\lambda})}}
\newcommand{\ct}{\mathtt{c}}
\newcommand{\dt}{\mathtt{d}}
\newcommand{\mt}{\mathtt{m}}
\newcommand{\pt}{p}
\newcommand{\mb}{h}

\DeclareMathOperator{\Image}{Im}
\DeclareMathOperator{\End}{End}
\DeclareMathOperator{\Ind}{Ind}

\newcommand{\qE}{\mathtt{E}}
\newcommand{\qF}{\mathtt{F}}
\newcommand{\qK}{\mathtt{K}}
\newcommand{\fq}{v}
\newcommand{\qt}{\tilde{q}}
\newcommand{\kt}{\mathtt{k}}
\newcommand{\vsl}{U_{v}(\widehat{\mathfrak{sl}}_n)}
\newcommand{\qgl}{U_{\fq}(\widehat{\mathfrak{gl}}_n)}
\newcommand{\qheis}{U_{\fq}(\mathsf{Heis})}
\newcommand{\rt}{\mathtt{r}}
\newcommand{\Fsl}{F^{sl}}
\newcommand{\Fheis}{F^{\mathsf{H}}}
\newcommand{\qbin}[2]{\begin{bmatrix}{#1}\\ {#2}\end{bmatrix}_{\fq}}

\title{Twisted Fock module of toroidal algebra \\ via  DAHA and vertex operators}
\author{Mikhail Bershtein \and Roman Gonin}
\date{}

\begin{document}

\maketitle 

\begin{abstract}
	We construct the twisted Fock module of quantum toroidal \(\mathfrak{gl}_1\) algebra with a slope \(n'/n\) using vertex operators of quantum affine \(\mathfrak{gl}_n\). The proof is based on the \(q\)-wedge construction of an integrable level-one \(U_q(\widehat{\mathfrak{gl}}_n)\)-module and the representation theory of double affine Hecke algebra. The results are consistent with Gorsky-Negu\c{t} conjecture (Kononov-Smirnov theorem) on stable envelopes for Hilbert schemes of points in the plane and can be viewed as a manifestation of \((\mathfrak{gl}_1,\mathfrak{gl}_n)\)-duality.
\end{abstract}


\section{Introduction}

\paragraph{Toroidal algebra}
In this article, we study the representation theory of quantum toroidal $\mathfrak{gl}_1$, denoted by $\DI$. The subject has been actively developed in recent years. It has numerous applications, including geometric representation theory, enumerative geometry, topological strings, integrable systems, knot theory,  and combinatorics. Due to this, the study of the algebra is particularly interesting.

One of the definitions of \(\DI\) is a \(N \rightarrow \infty\) limit of the spherical double affine Hecke algebra. The double affine Hecke algebra for \(\mathfrak{gl}_N\) will be called DAHA for short and denoted by \(\mathcal{H}_N\), it can be presented by generators 
\begin{equation}
	T_1,\dots, T_{N{-}1},Y_1^{\pm 1},\dots, Y_N^{\pm 1}, X_1^{\pm 1},\dots, X_N^{\pm 1},
\end{equation} 
and relations which we recall in  Section \ref{sec: Double Affine Hecke Algebra}. The algebra \(\mathcal{H}_N\) depends on a parameter \(v\) which appears in Hecke relations and an additional parameter \(q\). The group \(\widetilde{SL}(2,\mathbb{Z})\) (central extension of \(SL(2,\mathbb{Z})\) by \(\mathbb{Z}\)) acts on \(\mathcal{H}_N\) via automorphisms. The algebra $\mathcal{H}_N$ has two natural gradings $\deg_X$ and $\deg_Y$. 
 
Also, the algebra \(\mathcal{H}_N\) has the standard Cherednik representation, denoted by $\Ch_u$. As a vector space, $\Ch_u$ is the space of Laurent polynomials \(\mathbb{C}[Y_1^{\pm 1},\dots,Y_N^{\pm 1}]\). The generators \(Y_i\) act by multiplication and \(X_i\) act by certain \(q\)-difference operators. The common eigenvectors of \(X_1\dots, X_N\) in the space of Laurent polynomials \(\mathbb{C}[Y_1^{\pm 1},\dots,Y_N^{\pm 1}]\) are nonsymmetric Macdonald polynomials \cite{C05}.

Let \(\bc_+ \in \mathcal{H}_N\) denotes symmetrizer in Hecke algebra, then the spherical DAHA $\SHp$ is $\bc_+ \mathcal{H}_N \bc_+$. 
It  has homogeneous generators \(P_{a,b}\), \((a,b)\in \mathbb{Z}^2\backslash \{(0,0)\}\) such that 
\begin{equation}
	P_{0,k}=\bc_+ \sum Y_i^k \bc_+, \; k>0, \quad \deg_X P_{a,b}=a,\; \deg_Y P_{a,b}=b,\quad \sigma P_{a,b}=P_{\sigma(a,b)},\; \sigma\in SL(2,\mathbb{Z}).
\end{equation}
It appears that relations between \(P_{a,b}\) stabilize in the limit \(N\rightarrow \infty\). The quantum toroidal algebra \(\DI\) is a central extension of a limit (in a certain sense) of \( \SHp \) \cite{SV11}. The algebra $\DI$ is generated by \(P_{a,b}\) and central elements \(c^{\pm 1},(c')^{\pm 1}\). The algebra depends on parameters $q_1$ and $q_2$. For any \(N\) and $q_1=q$, $q_2=v^2$, there is a surjective homomorphism \(\DI \twoheadrightarrow \SHp\). This gives an action of \(\DI\) on the space of symmetric Laurent polynomials \(\bc_+ \mathbb{C}[Y_1^{\pm 1},\dots,Y_N^{\pm 1}]\). Moreover, there is a limit action of \(\DI\) on the space of symmetric functions \(\Lambda\). The eigenvectors of the operators \(P_{k,0}\) in \(\Lambda\) are Macdonald symmetric functions. The \(SL(2,\mathbb{Z})\)-action on \(\SHp\) can be lifted to an \(\widetilde{SL}(2,\mathbb{Z})\)-action on \(\DI\).

The space \(\Lambda=\mathbb{C}[p_1,p_2, \dots]\) is the space of polynomials in infinite number of variables (power sums). This space has a natural action of Heisenberg algebra with generators \(a_k\), \(k \in \mathbb{Z}_{\neq 0}\) and relations \([a_k,a_l]=k\delta_{k+l,0}\). Namely, operators \(a_{-k}\) act as multiplication by \(p_k\) and \(a_k\) act as \(k \partial_{p_k}\), for \(k>0\). The action of \(\DI\) on \(\Lambda\) can be written via Heisenberg generators \(a_k\), this construction is called \emph{bosonization}. The formulas are written for the generators \(P_{1,b}\), \(P_{0,k}\), \(P_{-1,b}\) for $b\in \mathbb{Z}$ and $k \in \mathbb{Z}_{\neq 0}$, these generators are called \emph{Chevalley generators} (contrary to all \(P_{a,b}\) for \((a,b)\in \mathbb{Z}^2\backslash \{(0,0)\}\), which could be called \emph{PBW generators}). The formulas for the bosonization roughly look like 
\begin{equation} \label{eq: bosonization of Fock in introduction}
	P_{0,k}=\#\, a_k, \;\;\;  \sum P_{1,b} z^{-b}= u^{-1} \# \exp \Big(\sum_{
	} \#  a_{-k} z^k\Big),  \;\;\;  \sum P_{-1,b} z^{-b}= u\# \exp \Big(\sum 
	\#  a_{-k} z^k\Big),
\end{equation}
see Section \ref{subsec: example n=1} for the precise formulas. The representation is called \emph{Fock module}. It depends on the parameter \(u\), we denote the representation by \(\mathcal{F}_u\). The formula \eqref{eq: bosonization of Fock in introduction} is equivalent to the orignal consturction in \cite{FHHSY}. 

The algebra $\DI$ can be considered as a quantum group, it has topological coproduct as in Drinfeld current realization \cite{DI97}. Using this coproduct we can construct \(\DI\)-representations of the form \(\mathcal{F}_{u_1}\otimes \dots \otimes  \mathcal{F}_{u_n}\). It was shown in \cite{FHSSY}, \cite{N16} that the image of (completed) toroidal algebra  $\DI$ in the endomorphisms of this tensor product is the \(q\)-$W$-algebra for $\mathfrak{gl}_n$. Moreover, the algebra \(\DI\) can be viewed as a central extension of the limit of such \(W\)-algebras\footnote{The algebra $\DI$ is a central extension of the limit of $\SHp$ for $N \rightarrow \infty$. The completed algebra $\DI$ is a central extension of the limit of $W$-algebras for $\mathfrak{gl}_n$ for $n \rightarrow \infty$. These statements are independent. E.g. non-completed $\DI$ with $c=v^{-1}$, $c'=1$ acts faithfully on $\mathcal{F}_u$, though the image of the completion is isomorphic to (completed) Heisenberg algebra due to formula \eqref{eq: bosonization of Fock in introduction}.}, this was one of the first motivations to study $\DI$, see \cite{M07}. 

\paragraph{Stable bases and duality}
The algebra \(\DI\) plays an important role in geometric representation theory. Let \(\text{Hilb}_k\) denotes Hilbert scheme of \(k\) points in \(\mathbb{C}^2\). There is a natural action of the torus $T = \mathbb{C}^*_{q_1} \times \mathbb{C}^*_{q_2}$ on \(\text{Hilb}_k\). Let  $K_{T} (\text{Hilb}_k)_{\text{loc}}$ denotes localized equivariant $K$-theory of $\text{Hilb}_k$. There is an action of \(\DI\) on the sum \(\mathbf{K} = \bigoplus_{k=0}^{\infty}  K_{T} (\text{Hilb}_k)_{\text{loc}}\) via correspondences.  The obtained representation is isomorphic to the Fock module $\mathcal{F}_{u}$ \cite{FT11},\cite{SV13}.

The torus fixed points of \(\text{Hilb}_k\) are labeled by the partitions \(\lambda\) with \(|\lambda|=k\). Denote the corresponding point by \(p_\lambda\in \text{Hilb}_k\). Due to localization theorem, there is a basis \(\{[p_\lambda]\}\) in \(\mathbf{K}\). This basis corresponds to the Macdonald symmetric functions under the identification \(\mathbf{K} \cong \mathcal{F}_u \cong \Lambda\).

There is another remarkable basis in the space \(\mathbf{K}\), namely, $K$-theoretic stable envelope basis \(\{ s_\lambda^{\st} \}\), see \cite{O:2017}. The stable envelope \(s_\lambda^{\st}\) depends on the slope parameter \(\st \in \mathbb{R}\setminus\mathbb{Q}\). For fixed $k$, there is a chamber structure, namely, \(s_\lambda^{\st} \in K_{T} (\text{Hilb}_k)_{\text{loc}}\) is a piecewise constant function on \(\st\) with discontinuities at rational numbers \(m/n\) with \(n\leq k\). So for any rational number \(m/n\) we can define elements \(s_\lambda^{m/n+\epsilon}\) and \(s_\lambda^{m/n-\epsilon}\) for sufficiently small \(\epsilon\).

In the paper \cite{GN15}, Gorsky and Negu\c{t} conjectured that for coprime integer \(m,n\) there is an identification between \(\mathcal{F}_u\) and integrable level-one representation of quantum affine algebra \(\qgl\) such that bases \(\{ s_\lambda^{m/n-\epsilon} \}\) and \(\{ s_\lambda^{m/n+\epsilon} \}\) correspond to the standard and costandard bases respectively. This conjecture was supported by the consistency of the action of operators \(P_{km,kn}\in \DI\) on stable basis \cite{Negut:2016} and diagonal \(v\)–Heisenberg subalgebra in \(\qgl\) on standard basis \cite{LT00}.

This conjecture was proven by Kononov and Smirnov in \cite{KS:2020} using geometric methods, in particular 3d-mirror symmetry.
Algebraically, this relation could be viewed as a manifestation of \((\mathfrak{gl}_1,\mathfrak{gl}_n)\)-duality in toroidal setting. The relation between 3d-mirror symmetry and \((\mathfrak{gl}_m,\mathfrak{gl}_n)\)-duality (in affine setting) appeared in \cite{RSVZ:2019}.

To state our main result, we need to reformulate the above “in the opposite way”. Let us start not from \(\DI\) but from the intergrable level-one representation of \(\qgl\). We construct an action of \(\DI\) on this space such that the obtained representation is isomorphic to the Fock module twisted by \(\sigma \in \widetilde{SL}(2,\mathbb{Z})\), here \(\sigma (m,n)=(0,1)\). The action of Chevalley generators \(P_{1,b},P_{-1,b}\) is given in terms of vertex operators for \(\qgl\), and \(P_{0,k}\) are proportional to the generators of the diagonal \(v\)–Heisenberg subalgebra in \(\qgl\).  

\paragraph{Our results} The integrable level-one representation of \(\qgl\) has an explicit construction in terms of semi-infinite \(v\)-wedges \cite{St:1995, KMS95}. We first consider the case of finite \(v\)-wedges and then take the limit. Let \(\mathbb{C}^n[Y^{\pm 1}]\) denotes the evaluation representation of \(\qgl\), then its \(N\)-th tensor power can be written as
\begin{equation}\label{eq:C^nN}
	 \left( \mathbb{C}^n[Y^{\pm 1}]\right)^{\otimes N} \simeq  \left(\mathbb{C}^n\right)^{\otimes N}[Y_1^{\pm1}, \dots, Y_N^{\pm 1}].
\end{equation}
By the quantum affine version of Schur-Weyl duality, there is an action of the affine Hecke algebra \(H^Y\) with the generators \(T_i,Y_j\) on this space, moreover this action commutes with \(\qgl\) (see \cite{GRV92}, \cite{CP94} and references therein). In Theorem \ref{thm: finite explicit construction} we extend this action to the whole DAHA \(\mathcal{H}_N\). The construction depends on an integer number \(n_{tw}\). If \(n_{tw}\) is coprime to \(n\) then the obtained representation is isomorphic to Cherednik polynomial representation twisted by certain \(\sigma \in \widetilde{SL}(2,\mathbb{Z})\), see Theorem~\ref{thm:isomorphism with twisted}. Below we will assume this coprimeness and denote \(n'=n_{tw}\).

Let \(\bc_-\) be antisymmetrizer in the (finite) Hecke algebra, \(\SHm = \bc_-\mathcal{H}_N\bc_-\) be the corresponding spherical DAHA. Above we discussed that the algebra \(\DI\) surjects onto $\SHp$. Analogously, $\DI$ surjects onto $\SHm$ for $q_1=q$, $q_2=v^{-2}$. Using this we obtain an action
\begin{equation}
	\DI \rightarrow \SHm \curvearrowright \; \bc_- \left(\left(\mathbb{C}^n\right)^{\otimes N}[Y_1^{\pm1}, \dots, Y_N^{\pm 1}]\right)= \Lambda_{v}^{N} \!\left(\mathbb{C}^n[Y^{\pm 1}] \right).
\end{equation}
The space $\mathbb{C}^n[Y^{\pm 1}]$ has a basis $e_{nj+a} =Y^j e_a $ for $a=0, \dots, n-1$ and $j \in \mathbb{Z}$. Then the space $\Lambda_{v}^{N} \!\left(\mathbb{C}^n[Y^{\pm 1}] \right)$ has a basis $e_{k_1} \wedge \dots \wedge e_{k_N}$ for $k_1 < \dots < k_N$. It appears that the action Chevalley generators \(P_{1,b},P_{-1,b}\) on this space can be written in terms of vertex operators \(\Phi_k, \Phi_k^*, \Psi_k, \Psi_k^*\).
The operators \(\Phi_k,\Psi_k\) are given by \(\Phi_k(w)=e_k\wedge w\), \(\Psi_k(w)=w\wedge e_k\). The operators \(\Phi_k^*, \Psi_k^*\) are dual. Then the action of Chevalley generators is given by the formulas 
\begin{equation} \label{eq: intro: cheval before the limit}
	P_{1,b} \mapsto \sum\nolimits_{k \in \mathbb{Z}} \# \Psi_{k+n'+nb} \Phi^*_{-k},\quad P_{-1,b} \mapsto \sum\nolimits_{k \in \mathbb{Z}} \# \Phi_{k-n'+nb} \Psi^*_{-k},
\end{equation}
see Section \ref{subsec: finite} for details. Let us remark that the operators 	\(\Phi_k, \Phi^*_k, \Psi_k, \Psi^*_k\) are components of natural intertwining operators 
\begin{align*}
	\Phi \colon& \mathbb{C}^n[Y^{\pm 1}]\!\otimes\! \Lambda_{v}^{N} \!\left(\mathbb{C}^n[Y^{\pm 1}] \right)  \rightarrow  \Lambda_{v}^{N{+}1} \!\left(\mathbb{C}^n[Y^{\pm 1}] \right), & \Phi^* \colon& \Lambda_{v}^{N} \!\left(\mathbb{C}^n[Y^{\pm 1}] \right)  \rightarrow \mathbb{C}^n[Y^{\pm 1}]\!\otimes\! \Lambda_{v}^{N{-}1} \!\left(\mathbb{C}^n[Y^{\pm 1}] \right), \\
	\Psi \colon&  \Lambda_{v}^{N} \!\left(\mathbb{C}^n[Y^{\pm 1}] \right) \!\otimes\! \mathbb{C}^n[Y^{\pm 1}] \rightarrow  \Lambda_{v}^{N{+}1} \!\left(\mathbb{C}^n[Y^{\pm 1}] \right), & \Psi^* \colon& \Lambda_{v}^{N} \!\left(\mathbb{C}^n[Y^{\pm 1}] \right)  \rightarrow \Lambda_{v}^{N{-}1} \!\left(\mathbb{C}^n[Y^{\pm 1}] \right)  \!\otimes\! \mathbb{C}^n[Y^{\pm 1}].	
\end{align*}
Vertex operators can be considered in a more general context as such intertwiners \cite{FR92,  JM95}. 

The space \(\sip\!\!\left(\mathbb{C}^n[Y^{\pm 1}] \right)\) is defined as an inductive limit 
\begin{equation} 
	\mathbb{C}^n[Y^{\pm 1}] \rightarrow \Lambda_{v}^{2} \!\left(\mathbb{C}^n[Y^{\pm 1}] \right)  \rightarrow \dots \rightarrow \Lambda_{v}^{N} \!\left(\mathbb{C}^n[Y^{\pm 1}] \right) \rightarrow \dots
\end{equation}
Since \(\qgl\) and \(\DI\) act on \(\Lambda_{v}^{N} \!\left(\mathbb{C}^n[Y^{\pm 1}] \right)\), it is natural to ask for the limit of these actions. It was shown \cite{St:1995, KMS95} that the action of \(\qgl\) has an (appropriately defined) limit, the obtained representation is an intergrable level-one module of $\qgl$.

We take the limit of the \(\DI\)-action in several steps. Introduce several subalgebras in \(\DI\): \(\DI^+,\DI^-,\DI^\swarrow,\DI^\downarrow\) which are generated by \(c,c'\), and \(P_{a,b}\) subject to the conditions \(a\geq 0,\) \(a\leq 0\), \(an'+bn \leq 0\), \(b \leq 0\) respectively. 
\begin{figure}[ht]\label{fig:subalgebras}
	\begin{center}
		\begin{tabular}{c c c c }
			\begin{tikzpicture}[scale=0.4]
				\draw[color=gray!30] (-3.5,-3.5) grid (3.5,3.5);
				\draw[-,thick] (-3.5,0)--(3.5,0);
				\draw[-,thick] (0,-3.5)--(0,3.5);
				\fill[pattern={north east lines}, pattern color=gray!40] 		(0,-3.5)--(0,3.5)-- (3.5,3.5) -- (3.5,-3.5);
			\end{tikzpicture} \qquad
			&
			\begin{tikzpicture}[scale=0.4]
				\draw[color=gray!30] (-3.5,-3.5) grid (3.5,3.5);
				\draw[-,thick] (-3.5,0)--(3.5,0);
				\draw[-,thick] (0,-3.5)--(0,3.5);
				\fill[pattern={north east lines}, pattern color=gray!40] 	(0,-3.5)--(0,3.5)-- (-3.5,3.5) -- (-3.5,-3.5);
			\end{tikzpicture} \qquad
			&
			\begin{tikzpicture}[scale=0.4]
				\draw[color=gray!30] (-3.5,-3.5) grid (3.5,3.5);
				\draw[-,thick] (-3.5,0)--(3.5,0);
				\draw[-,thick] (0,-3.5)--(0,3.5);
				\fill[pattern={north east lines}, pattern color=gray!40] 	(-3.5,-3.5)--(-3.5,7/3)-- (3.5,-7/3) -- (3.5,-3.5);
				\draw[color=gray!40] (-3.5,7/3)-- (3.5,-7/3);
			\end{tikzpicture} 
			\qquad
			&
			\begin{tikzpicture}[scale=0.4]
				\draw[color=gray!30] (-3.5,-3.5) grid (3.5,3.5);
				\draw[-,thick] (-3.5,0)--(3.5,0);
				\draw[-,thick] (0,-3.5)--(0,3.5);
				\fill[pattern={north east lines}, pattern color=gray!40] 	(-3.5,-3.5)--(-3.5,0)-- (3.5,0) -- (3.5,-3.5);
			\end{tikzpicture} 
			\\
			\(U_{q_1,q_2}(\ddot{\mathfrak{gl}}_1)^+\) & 	\(U_{q_1,q_2}(\ddot{\mathfrak{gl}}_1)^-\) &	\(U_{q_1,q_2}(\ddot{\mathfrak{gl}}_1)^{\swarrow}\) &	\(U_{q_1,q_2}(\ddot{\mathfrak{gl}}_1)^{\downarrow}\) 
		\end{tabular}
		\caption{Subalgebras of \(\DI\) used in Sections \ref{sec:Semi-inifinite 1}, \ref{sec: twisted Fock II}}
	\end{center}	
\end{figure}
These subalgebras are displayed on the Fig.~\ref{fig:subalgebras}. They  are isomorphic due to \(\widetilde{SL}(2,\mathbb{Z})\)-action on \(\DI\). Consider an element $\sigma \in \widetilde{SL}(2, \mathbb{Z})$ such that $\DI^{\swarrow} = \sigma \left(  \DI^{\downarrow}  \right)$. The element is particularly important since we will obtain the Fock module twisted by $\sigma$.

Theorem \ref{thm: positive half} states that the action of operators from \(\DI^+\) converges for \(|q^{-1}v^2|<1\). Similarly, Theorem \ref{thm: negative half} states that the action of \(\DI^-\) converges for \(|q^{-1}v^2|>1\). Therefore we can not obtain the action of whole $\DI$ by a straightforward limit argument. Though the obtained action of the generators $P_{\pm 1, b}$ is given by rational functions of \(q,v\), hence it can be analytically continued for generic \(q,v\). This gives actions $\DI^{\pm} \curvearrowright \sip\!\!\left(\mathbb{C}^n[Y^{\pm 1}] \right)$ for generic \(q,v\). The proofs use the expressions via vertex operators \eqref{eq: intro: cheval before the limit} and as a result we get analogous formulas after the limit, see formula \eqref{eq: div series} below.

On the other hand, we show in Section \ref{subsec: bottom half} that the action of \(\DI^\swarrow\) converges for any $q, v$. Also, we prove that the obtained representation is isomorphic to the restriction of twisted Fock module $\mathcal{F}_u^{\sigma}$ to \(\DI^\swarrow\). Using this we can glue the actions of \(\DI^+\) and \(\DI^-\) and obtain 



\begin{thm} \label{intro: main theorem}
	The following formulas determine an action of $\DI$ on \(\sip\!\!\left(\mathbb{C}^n[Y^{\pm 1}] \right)\)
	\begin{align}
			&c \mapsto v^{-n},\;  c' \mapsto v^{-n'},\quad 
			P_{0,j} \mapsto \# B_j,\\
			&P_{1,b} \mapsto \sum\nolimits_{k \in \mathbb{Z}} \# \hpsi_{k+n'+nb} \hphi^*_{-k},\quad P_{-1,b} \mapsto \sum\nolimits_{k \in \mathbb{Z}} \# \hphi_{k-n'+nb} \hpsi^*_{-k} \label{eq: div series},
	\end{align} 
	for $q_1=q$, $q_2 =v^{-2}$. The obtained representation is isomorphic to twisted Fock module $\mathcal{F}_{u}^{\sigma}$.	
\end{thm}
Here \(B_k\) are the generators of the diagonal \(v\)–Heisenberg subalgebra in \(\qgl\) and \(\hphi_k, \hphi^*_k, \hpsi_k, \hpsi^*_k\) are vertex operators acting on the semi-infinite wedge. We discuss the meaning of the series \eqref{eq: div series} in Section \ref{subsection: analytic continuation} as an explicitly written analytic continuation. For the precise statement of the theorem above see Theorem \ref{thm: explict construction of twisted Fock}. 


Recall that our original motivation was the relation between standard and stable bases. By the definition \cite{LT00}, the standard basis consists of semi-infinite wedges \(e_{-\lambda_1} \wedge e_{-\lambda_2+1} \wedge \dots\in \sip\!\!\left(\mathbb{C}^n[Y^{\pm 1}] \right)\). There is a combinatorial characterization of the stable basis given by certain three conditions \cite[Sect. 4.1]{Negut:2016}. We check two of these conditions, see Theorem \ref{thm:stable triangularity}. As above, the proof is based on the ``finitization'', namely, we first prove similar properties for the basis \(\{ e_{i_1}\otimes \dots\otimes e_{i_N} \}\) in the space \eqref{eq:C^nN}, see Corollary \ref{corol: eigenvectors twisted case} and Theorem \ref{thm:A lmabda basis}. The most interesting “window” condition will be considered elsewhere. A verification of the condition would give a new proof of Gorsky-Negu\c{t} conjecture (Kononov-Smirnov theorem).

\paragraph{Further Questions}
\begin{itemize}
	\item In the paper we proved Theorem \ref{thm:stable triangularity} using the corresponding properties of the basis  \(\{ e_{i_1}\otimes \dots\otimes e_{i_N} \}\) in the space \eqref{eq:C^nN}. It is natural to expect a finite non-symmetric analog of Gorsky-Negu\c{t} conjecture. This means that apart from the proven properties, there is a “window” condition given in terms of nonsymmetric Macdonald polynomials.	
	\item In the paper we constructed duality (action on the same space) of two algebras \(\DI\) and \(\qgl\). The action of quantum affine algebra could be promoted to the toroidal algebra \(U_{v, v'}(\ddot{\mathfrak{gl}}_n)\). This could be viewed as some twisted generalization of \cite{FJM:2019}, we expect that a new presentation of toroidal algebra suggested in \cite{N:2019} is relevant for this question.
	
	Usually, the most interesting part of \((\mathfrak{gl}_m,\mathfrak{gl}_n)\)-duality is a relation  between integrable systems, see e.g. \cite{FJM:2019}. We do not know the corresponding results in our setting.
	
	\item It is interesting to find an interpretation of our construction in the framework of geometric representation theory. In particular, to relate our results to the results of \cite{KS:2020}.

	\item Usually, the algebra \(\DI\) acts on representations through a quotient which is isomorphic to a \(q\)-\(W\) algebra. Our construction of the twisted Fock module should be related to realizations of twisted \(q\)-\(W\) algebras.  To the best of our knowledge, even the corresponding \(q\)-\(W\) algebras are not known. See \cite{BG:2021} for the \(n=2\) case and \cite{BG:2020} for the case of any $n$ and \(q_2=1\).
\end{itemize}
\paragraph{Plan of the paper}
In Section \ref{sec: Double Affine Hecke Algebra} we recall the definition of DAHA and the corresponding basic constructions.

In Section \ref{sec:Representation} we construct and study action of DAHA on the space \(\left(\mathbb{C}^n\right)^{\otimes N}[Y_1^{\pm1}, \dots, Y_N^{\pm 1}]\). We prove that the obtained representation is isomorphic to twisted Cherednik representation (Theorem~\ref{thm:isomorphism with twisted}). Also, we prove Theorems \ref{thm:Bi triang} and \ref{thm:A lmabda basis} which state that certain matrices are upper triangular, the theorems are used in the proof of Theorem \ref{thm:isomorphism with twisted} and in Sections  \ref{sec: twisted Fock II} and \ref{sec:standard basis}.

In Section \ref{sec:toroidal} we recall a presentation and some properties of the toroidal algebra \(\DI\). In particular, we describe its relation to spherical DAHA. We mainly follow \cite{SV11}.

In Section \ref{sec:wedge} we study the space of finite and semi-infinite wedges. We mainly follow \cite{KMS95, LT00}. We slightly extend \emph{loc. cit.} results since we need properties of all the types of the vertex operators \(\Phi, \Phi^*, \Psi, \Psi^*\), but only \(\Phi\) was considered in \emph{loc. cit.}. The main results are Propositions~\ref{prop: stable vertex}, \ref{prop: stable dual vertex} on \(N\rightarrow \infty\) stabilization of the vertex operators and Proposition~\ref{prop: heis and vertex operators} on commutation relations with the diagonal $v$-Heisenberg subalgebra of $\qgl$. 

In Section \ref{sec:Semi-inifinite 1} we study the limit for the action of Chevalley generators via the vertex operators formula. This gives the actions of $\DIp$ and $\DIm$ as a limit for $|q^{-1}v^2|<1$ and $|q^{-1} v^2|>1$ respectively. The formula allows us to obtain an analytic continuation of the actions. Also, we consider an example of the construction in the case of the (non-twisted) Fock module.

In Section \ref{sec: twisted Fock II} we prove that the action of \(\DI^\swarrow\) converges for any $q$ and $v$. Then we show that the action of the subalgebras $\DI^{\pm}$ and \(\DI^\swarrow\) are restrictions of an action of the whole $\DI$. The obtained representation is isomorphic to the twisted Fock module. The action of Chevalley generators of $\DI$ is given by the formulas via vertex operators (Theorem \ref{thm: explict construction of twisted Fock}). This is the main result of the whole paper.

In Section \ref{sec:standard basis} we study standard basis, the main result is Theorems \ref{thm:stable triangularity} mentioned above.

In Appendix \ref{app:A} we recall the action of \(\qgl\) on \(\sip\!\!\left(\mathbb{C}^n[Y^{\pm 1}] \right)\). Then we deduce intertwining properties for the operators \(\Phi^*_k\), here the non-trivial part is a precise form of commutation relations with the diagonal $v$-Heisenberg algebra (Theorem \ref{thm: Bj Phik star}). This is used in the main part of the text (the proof of Proposition \ref{prop: heis and vertex operators}).

\paragraph{Acknowledgments} We are grateful to B.~Feigin,  E.~Gorsky,  Y.~Kononov, I. Makedonskyi, A.~Negu\c{t},  A.~Okounkov, A. Smirnov for their interest in our work and discussions. We are grateful to the participants of the seminar `Macdonald polynomials and DAHA' for the inspiring atmosphere of learning. MB is grateful to S. Grishaev for his help with computations in Sage. The work was partially carried out in Skolkovo Institute of Science and Technology under financial support of Russian Science Foundation within grant 19-11-00275. RG was also supported by ``Young Russian Mathematics'' award.

\section{Double Affine Hecke Algebra} \label{sec: Double Affine Hecke Algebra}
In this section, we recall the definition and basic properties of double affine Hecke algebra (DAHA) \cite{C92, K97, C05}. This section consists no new results.
\begin{defin}
    Double affine Hecke algebra $\mathcal{H}_N$ is an algebra with generators $T_1,\dots,T_{N-1}$, $\pi^{\pm 1}$,$Y_1^{\pm1}$, $\dots$, $Y_N^{\pm1}$ and relations
    \begin{align}
        &(T_i-v)(T_i+v^{-1})=0, \quad T_i T_{i+1} T_i = T_{i+1} T_i T_{i+1},\quad T_i T_j= T_j T_i \quad \text{for $|i-j|>1$}, \label{eq: DAHA relation 1} \\
        &T_i Y_i T_i=Y_{i+1}, \quad T_{i}Y_j=Y_jT_i \quad \text{for $j\neq i,i+1$,}\\
        &\pi Y_i \pi^{-1}= q^{\delta_{i,N}}  Y_{i+1},\quad  Y_iY_j=Y_jY_i, \\ 
        &\pi T_i \pi^{-1}= T_{i+1},\quad \pi^N T_i =T_i \pi^N, \label{eq: DAHA relation 4}
    \end{align}
here and below we use the convention $Y_1=Y_{N+1}$.
\end{defin}

The operators $T_1,\dots,T_{N-1}$ generate finite Hecke algebra $H$. The operators $T_1,\dots,T_{N-1},Y_1$,$\dots$,$Y_N$ generate affine Hecke algebra $H^Y$. The operators $T_1,\dots,T_{N-1},\pi^{\pm 1}$ generate another affine Hecke algebra denoted by $H^X$. Here one can define 
\begin{equation}
    X_i=T_i\cdots T_{N-1}\pi^{-1}T_1^{-1}\cdots T_{i-1}^{-1}.    
\end{equation}
The relations on $X_i$ are 
\begin{align}
    T_i^{-1} X_i T_i^{-1}=X_{i+1},\quad
    X_2^{-1}Y_1X_2Y_1^{-1}=T_1^2,\quad X_1Y_2^{-1}X_1^{-1}Y_2=T_1^2.
\end{align}
Let $\widetilde{SL}(2,\mathbb{Z})$ be the braid group on 3 stands. More precisely, $\widetilde{SL}(2,\mathbb{Z})$ is generated by $\tau_+$ and $\tau_-$ with the relation $\tau_+ \tau_-^{-1} \tau_+ = \tau_-^{-1} \tau_+ \tau_-^{-1}$. We use this notation since $\widetilde{SL}(2,\mathbb{Z})$ is an extension of $SL(2,\mathbb{Z})$ by $\mathbb{Z}$, the projection is given by
\begin{align} \label{eq: projection SL2}
    \tau_+ \mapsto& \begin{pmatrix}
    1 & 1\\
    0 & 1    
    \end{pmatrix}, &
    \tau_- \mapsto& \begin{pmatrix}
    1 & 0\\
    1 & 1    
    \end{pmatrix}.
\end{align}
The kernel is generated by $(\tau_+ \tau_-^{-1} \tau_+)^4$.
\begin{prop}[\cite{C05}]
There is an action of $\widetilde{SL}(2,\mathbb{Z})$ on $\mathcal{H}_N$ determined by the following formulas 
\begin{align}
    &\tau_+\colon \quad T_i \mapsto  T_i,\;\; X_i\mapsto X_i,\;\; Y_i\mapsto Y_iX_i T_{i-1}\cdots T_1 T_1 \cdots T_{i-1},
    \\
    &\tau_-\colon \quad T_i \mapsto  T_i,\;\; X_i\mapsto X_iY_i T_{i-1}^{-1}\cdots T_1^{-1}T_1^{-1}\cdots T_{i-1}^{-1},\;\; Y_i\mapsto Y_i.
\end{align}
\end{prop}
The algebra $\mathcal{H}_N$ is bigraded with the gradings $\deg_X$ and $\deg_Y$ defined by
\begin{align}
    \deg_X \pi =&-1 &  \deg_X Y_i =&0 & \deg_X T_i =&0\\
    \deg_Y \pi =&0 &  \deg_Y Y_i =&1 & \deg_Y T_i =&0.
\end{align}
\begin{lemma}\label{lemma:B_i}
    Let $n,n' \in \mathbb{Z}_{\geq 0}$ be coprime integers. Let $m,m'\in \mathbb{Z}_{\geq 0}$ be unique pair of integers such that $m<n$, $m'<n'$ and $nm'-n'm=1$. Then there is an  $\widetilde{SL}(2,\mathbb{Z})$ transformation mapping $X_i^{-1}$ to $\BB_i$ and $Y_i$ to $\AB_i$ such that 
    \begin{equation}\label{eq:BBI ABi deg}
    	\deg_X(\BB_i)=-n,\, \, \deg_Y(\BB_i)=n',\quad  \quad \deg_X(\AB_i)=-m,\, \, \deg_Y(\AB_i)=m'.
    \end{equation}
    Moreover, $\BB_{i+1}=T_i\cdots T_1 \BB_1 T_1 \cdots T_i$, $\AB_{i+1}=T_i\cdots T_1 \AB_1 T_1 \cdots T_i$, and 
    \begin{equation}\label{eq:BBi ABi}
        \BB_1=Z_1\cdots  Z_{n+n'},\quad \AB_1=W_1 \cdots   W_{m+m'},
    \end{equation}
    where 
    \begin{align}
    	\label{eq:Bi in X Y}
    	&Z_j=Y_1\; \text{ if } \left\lfloor \frac{j n}{n+n'} \right\rfloor = \left\lfloor \frac{(j-1)n}{n+n'} \right\rfloor,& \quad &Z_j=X_1^{-1}  \text{ if }  \left\lfloor \frac{j n}{n+n'} \right\rfloor = \left\lfloor \frac{(j-1)n}{n+n'} \right\rfloor +1,\\
   	   	\label{eq:Ai in X Y}
   		&W_j=Y_1\; \text{ if } \left\lfloor \frac{j m}{m+m'} \right\rfloor = \left\lfloor \frac{(j-1)m}{m+m'} \right\rfloor,& \quad &W_j=X_1^{-1} \text{ if } \left\lfloor \frac{j m}{m+m'} \right\rfloor = \left\lfloor \frac{(j-1)m}{m+m'} \right\rfloor +1.
    \end{align}
\end{lemma}
\begin{proof}
    The formulas for $\BB_i$ and $\AB_i$,  in terms of $\BB_1$ and $\AB_1$ follows from the relations $X_{j+1}^{-1}= T_j X_j^{-1} T_j$, $Y_{j+1}= T_j Y_j T_j$. The formulas for $\BB_1$ and $\AB_1$ can be proven using Euclidean algorithm. For the step one uses the formulas $\tau_-^{-1}(X_1^{-1})=Y_1X_1^{-1}=\tau_+^{-1}(Y_1)$.
\end{proof}

Euclidean algorithm used in the proof can be viewed as decomposition of the matrix $\begin{pmatrix}
	n & -m \\ -n' & m'
\end{pmatrix}$ into a product of the matrices $\begin{pmatrix}
1 & -1\\
0 & 1    
\end{pmatrix}, \begin{pmatrix}
1 & 0\\
-1 & 1    
\end{pmatrix}$ which correspond to $\tau_+^{-1}$, $\tau_-^{-1}$.

\begin{wrapfigure}{r}{0pt}                \begin{tikzpicture}[scale=.5]
        \draw[help lines] (0,0) grid (3,5);
        \draw[thick] (0,0) -- (3,5);
        \draw[thick] (0,0) -- (1,0) -- (1,1) -- (2,1) -- (2,2) -- (2,3) -- (3,3) -- (3,4) -- (3,5); 
    \end{tikzpicture}  
\end{wrapfigure}
    The formula for $\BB_1$ has also the following geometric interpretation. Draw the segment from $(0,0)$ to $(n',n)$. Draw the closest north-east lattice path below the segment. Then for each east step of the path we write $Y_1$ and for each north step we write $X_1^{-1}$.

\begin{Example} \label{example:1}
Let us take $n=5$ and $n'=3$. Then 
\begin{align*} 
	\tau_+^{-1}\tau_-^{-1}\tau_+^{-1}\tau_-^{-1}(X_1^{-1})&=Y_1X_1^{-1}Y_1X_1^{-2}Y_1X_1^{-2},
	\\ 
	\tau_+^{-1}\tau_-^{-1}\tau_+^{-1}\tau_-^{-1}(Y_1)&=Y_1X_1^{-1}Y_1X_1^{-2}.
\end{align*}
The formulas of $\BB_1$ agrees with the form of the sequence $\frac{jn}{n+n'}$ namely  $0,\frac{5}{8},\frac{10}{8},\frac{15}{8},\frac{20}{8},\frac{25}{8},\frac{30}{8},\frac{35}{8},\frac{40}{8}$ as well as the geometric description.
\end{Example}

\subsection{Cherednik representation}

Algebra $H^X$ has one-dimensional representation $\mathbb{C}_u$
\[
	T_i \mapsto v,\quad \pi \mapsto u.
\]
Cherednik representation $\Ch_u$ of $\mathcal{H}_N$ is the induced representation $\Ind_{H^X}^{\mathcal{H}_N}\mathbb{C}_u$. It can be identified with the space of Laurent polynomials $\mathbb{C}[Y_1^{\pm 1},\dots, Y_N^{\pm 1}]$.  The action of the generators $T_i$ and $\pi$ can be written as
\[
	T_i = s_i^{Y} + (v-v^{-1}) \frac{s_i^Y-1}{Y_i/Y_{i+1}-1}, \quad  \quad
	\pi(Y_1^{\lambda_1} Y_2^{\lambda_2} \cdots Y_N^{\lambda_N}) =  u q^{\lambda_N}Y_1^{\lambda_N} Y_2^{\lambda_1} Y_3^{\lambda_2} \cdots Y_{N}^{\lambda_{N-1}},
\]
here $s_i^Y$ is the permutation of $Y_i$ and $Y_{i+1}$.

For any composition $\lambda=(\lambda_1,\dotsm\lambda_N)\in\mathbb{Z}^N$ we denote $Y^\lambda=Y_1^{\lambda_1}\cdots Y_N^{\lambda_N}$, such vectors form the standard monomial basis in Cherednik representation. We will write $\lambda \leq \mu$ if $\mu-\lambda\in \oplus \mathbb{Z}_{\geq 0} \alpha_i$, where $\alpha_i$ are positive simple roots of $\mathfrak{sl}_N$.
\begin{defin}\label{def:prec}
	Let $\lambda$, $\mu \in \mathbb{Z}^N$. We write $\lambda \prec \mu $ if
	\begin{enumerate}
		\item $\lambda^+ < \mu^+$ where $\lambda^+$ is the dominant coweight lying in the orbit of $\lambda$, and analogously for $\mu^+$
		\item $\lambda^+= \mu^+$ and $\lambda < \mu$
	\end{enumerate}
\end{defin}
For example, in case $N=2$ we have
\[
	(1,1) \prec (0,2) \prec (2,0).
\]
The action of the operators $X_1,\dots,X_N$ in the basis $Y^\lambda$ is triangular with respect to the order $\prec$. More explicitly
\begin{equation} \label{eq: uper triangular X_i}
    X_i Y^{\lambda} = u^{-1}  q^{-\lambda_i} v^{-\lt_i} Y^{\lambda} + \sum_{\mu \prec \lambda} x_{\lambda, \mu} Y^{\mu},
\end{equation}
here
\begin{equation}
  \lt_i= \big|\{j|  \lambda_j > \lambda_i  \} \big| + \big|\{j| j<i , \lambda_j = \lambda_i  \} \big| -  \big|\{j| \lambda_j < \lambda_i  \} \big| - \big|\{j| j>i, \lambda_j = \lambda_i  \} \big|.
  \end{equation}
The non-symmetric Macdonald polynomials $E_{\lambda}$ are defined as eigenvectors of $X_1,\dots,X_N$ with the leading term $Y^\lambda$. Note that \eqref{eq: uper triangular X_i} implies
	\begin{equation} \label{eq: Cherednik eigenvalues}
	X_i E_{\lambda} = u^{-1}  q^{-\lambda_i} v^{-\lt_i}  E_{\lambda}.
	\end{equation}

\section{Representation} \label{sec:Representation}
 In this section we introduce representation $\Ch^{(n,n_{tw})}_{u_0, \dots, u_{n-1}}$, generalize \eqref{eq: uper triangular X_i} and interpret the obtained representation as twisted Cherednik representation. Also, we study a basis $A_{\lambda}$ of the representation $\Ch^{(n,n_{tw})}_{u_0, \dots, u_{n-1}}$ in the case $\gcd(n, n_{tw})=1$.
\subsection{Explicit construction} \label{subsection: explicit construction}
\paragraph{Action of affine Hecke algebra}
Fix $n$ and let $\mathbb{C}^n$ be a vector space with the basis $e_0,\dots,e_{n-1}$. Define an $R$-matrix acting on $\mathbb{C}^n\otimes \mathbb{C}^n$
\begin{equation}
    R= v \sum_{a}  E_{aa}\otimes E_{aa}+\sum_{a<b} \Big( E_{ab}\otimes E_{ba}+E_{ba}\otimes E_{ab}+ (v-v^{-1})E_{aa}\otimes E_{bb}\Big),
\end{equation}
here $E_{ab}$ is a matrix unit. Define an action of $H$ on $(\mathbb{C}^n)^{\otimes N}$ by the formula $T_i \mapsto R_{i,i+1}$, here the indices encodes factors on which $R$-matrix acts. One can induce an action of $H^Y$ on $(\mathbb{C}^n)^{\otimes N}[Y_1^{\pm1}, \dots Y_N^{\pm 1}]$.

Below we will introduce notation to distinguish different actions of permutation group on the space $(\mathbb{C}^n)^{\otimes N}[Y_1^{\pm1}, \dots Y_N^{\pm 1}]$. 
Let $s_{ij}^Y$ be an operator acting on $(\mathbb{C}^n)^{\otimes N}[Y_1^{\pm1}, \dots, Y_N^{\pm 1}]$ which swaps $Y_i$ and $Y_{j}$. Let $s_{ij}^e$ be an operator acting on $(\mathbb{C}^n)^{\otimes N}[Y_1^{\pm1}, \dots, Y_N^{\pm 1}]$ which swaps tensor factors number $i$ and $j$ (and commutes with all $Y_k$). Finally, let $s_{ij} = s_{ij}^Y s_{ij}^e$. Also, we will denote $s_i^Y = s_{i, i+1}^Y$, $s_i^e = s_{i, i+1}^e$, and $s_i = s_{i, i+1}$.

The  action of $T_i$ is given by the following formula
\begin{equation}\label{eq:Ti in terms of si}
    T_i = s_i^{Y} R_{i, i+1} + (v-v^{-1}) \frac{s_i^Y-1}{Y_i/Y_{i+1}-1}
\end{equation}

The obtained representation of affine Hecke algebra $H^Y$ is well-know \cite{GRV92}, \cite{CP94}. It appears in the context of quantum affine Shcur-Weyl duality.
\paragraph{Action of DAHA} 
Below we will use the following identification 
\begin{align}
 \left( \mathbb{C}^n[Y^{\pm 1}]\right)^{\otimes N} &\rightarrow \left(\mathbb{C}^n\right)^{\otimes N}[Y_1^{\pm1}, \dots, Y_N^{\pm 1}] \\
 (Y^{j_1} e_{i_1}) \otimes \dots \otimes (Y^{j_N} e_{i_N}) &\mapsto Y_1^{j_1} \cdots Y_n^{j_n} e_{i_1} \otimes \dots \otimes e_{i_N}
\end{align}
Let us introduce $e_i\in \mathbb{C}^n[Y^{\pm 1}]$ for $i \in \mathbb{Z}$ by setting
\begin{equation}
e_{i} = Y^{-1} e_{i+n}.
\end{equation}
Introduce operators $\kappa$ and $D$ acting on $\mathbb{C}^n[Y^{\pm 1}]$ by $\kappa e_i=e_{i-1}$, $D (Y^j e_a)= u_a q^{j} Y^j e_a$ for $a=0, \dots, n{-}1$. Here $u_0, \dots, u_{n-1}$ are any non-zero numbers. By $\kappa_i$ and $D_i$ we denote the corresponding operators acting on $i$-th tensor factor. 
\begin{thm} \label{thm: finite explicit construction}
For any non-zero numbers $u_0, \dots, u_{n-1}$ and $n_{tw}\in \mathbb{Z}$, there is an action of algebra $\mathcal{H}_N$ on $(\mathbb{C}^n)^{\otimes N}[Y_1^{\pm1}, \dots, Y_N^{\pm1}]$ determined by the following conditions
\begin{itemize}
    \item subalgebra $H^Y$ acts as discribed above
    \item $\pi=\kappa_1^{n_{tw}}D_1 s_1\cdots s_{N-1}$
\end{itemize}
\end{thm}
Denote the obtained representation by $\Ch^{(n,n_{tw})}_{u_0, \dots, u_{n-1}}$.
\begin{proof}
    It is enough to check the relations involving $\pi$. The relations $\pi Y_i \pi^{-1} = q^{\delta_{i,N}} Y_{i+1}$ and $\pi T_i \pi^{-1}= T_{i+1}$ are easy to see. Let us check that $\pi^N$ commutes with $T_i$. Since  
    \[
        \pi^N(w_1\otimes \dots \otimes w_N)=  \left( \kappa^{n_{tw}} D w_1 \right) \otimes \left( \kappa^{n_{tw}} D w_2 \right) \otimes \dots \otimes \left( \kappa^{n_{tw}} D w_{N} \right),
    \]
    it is sufficient to consider only the $N=2$ case. In this case we denote $T=T_1$ for brevity. 
     Let $l = \mb + nk + s$ for $k \geq 0$ and $s =0, \dots n-1$. 
    For $s=0$
    \begin{align}
        T ( e_{\mb} \otimes e_l) &= v e_l \otimes e_{\mb} + (v-v^{-1})\sum_{j=1}^k  e_{l-nj} \otimes e_{\mb+nj}  \quad \quad \text{for $k \geq 0$},
        \label{eq: T1} \\
        T (e_l \otimes e_{\mb}) &= v^{-1} e_{\mb} \otimes e_l - (v-v^{-1})\sum_{j=1}^{k-1} e_{l-nj} \otimes e_{\mb+nj}  \quad \quad \text{for $k > 0$}. \label{eq: T2}
    \end{align}
For $s>0$
    \begin{align}        \label{eq: T3}
        T ( e_{\mb} \otimes e_l) &= e_l \otimes e_{\mb} + (v-v^{-1})\sum_{j=0}^{k} e_{\mb+nj} \otimes e_{l-nj},
        \\
        T (e_l \otimes e_{\mb}) &= e_{\mb} \otimes e_l - (v-v^{-1})\sum_{j=1}^{k} e_{l-nj} \otimes e_{\mb+nj}. \label{eq: T4}
    \end{align}
 Since the formulas \eqref{eq: T1}--\eqref{eq: T4} are invariant under the shift $l \mapsto l-n_{tw}$, $\mb \mapsto \mb- n_{tw}$, we see that $T \pi^2 = \pi^2 T$ for $N=2$.   
\end{proof}

\subsection{Triangularity of Macdonald operators} \label{ssec: Triangularity of Macdonald operators}
Introduce a grading on $(\mathbb{C}^n[Y^{\pm 1}])^{\otimes N}$ as follows 
    \begin{align}
    \deg{e_{a_1}\otimes \dots \otimes e_{a_N}}=&\sum a_i, & \deg Y_i=&n.
    \end{align}
    Then the operators $T_i$ preserve the grading and  $\deg \pi = - n_{tw}$. Hence $\deg X_i = n_{tw}$. 
    
    To simplify our notation, let us assume $n, n_{tw}>0$. Let $d=\gcd(n,n_{tw})$, we use notations $n''=n/d$, $n'=n_{tw}/d$. Consider integers $m,m'$ such that $n m'-n_{tw} m=d$, $0\leq m'<n'$, $0 \leq m<n''$. Hence there is $\sigma \in \widetilde{SL}(2,\mathbb{Z})$ such that
    \begin{align}
         \sigma(X_i)=&\BB_i^{-1},& \quad &\deg_X(\BB_i)=-n'',\, \, \deg_Y(\BB_i)=n',& \quad  \deg \BB_i =& 0, \label{def sigma 1} \\
         \sigma(Y_i) =& \AB_i,& \quad &\deg_X(\AB_i)=-m,\, \, \deg_Y(\AB_i)=m',& \quad   \deg \AB_i =& d. \label{def sigma 2}
    \end{align}
   The corresponding matrix in \(SL(2,\mathbb{Z})\) is \(\begin{pmatrix}
    n''& -m \\ -n' & m'
   \end{pmatrix}\). Lemma \ref{lemma:B_i} gives explicit formulas for $\BB_i,\AB_i$.


Consider operators
\begin{equation} \label{eq: formula for G_i}
    G_{ij} =  R_{ij} s_{ij}^e + (v-v^{-1})\frac{1-s_{ij}^Y}{Y_i/Y_{j}-1} s_{ij}^e.
\end{equation}
Denote $G_i = G_{i,i+1}$. It follows from \eqref{eq:Ti in terms of si} that $T_i = G_i s_i$. Using this, we can write  the following formula 
\begin{equation}
    X_1^{-1}=\pi T_{N-1}^{-1}\cdots T_1^{-1}=\kappa_1^{n_{tw}}D_1s_1\cdots s_{N-1}s_{N-1}G_{N-1,N}^{-1}\cdots s_{1}G_{1,2}^{-1}=\kappa_1^{n_{tw}}D_1G_{1,N}^{-1}\cdots G_{1,2}^{-1}.
\end{equation}

For any collection $\lambda=(\lambda_1,\dots,\lambda_N)\in \mathbb{Z}^N$, we can consider a vector $e_{\lambda} =e_{\lambda_1}\otimes\dots \otimes e_{\lambda_N}$. We are going to prove that the operators $\BB_i$ are triangular in the basis $e_\lambda$ with respect to order $\prec$. Basically, the proof is a computation, similar to \cite[Sect. 5]{K97}. Let us first explain the idea on the following example (cf. Example \ref{example:1}).
\begin{Example}
    Let us take $N=3$, $n=5$, $n_{tw}=3$. Note that $Y_i = \kappa_i^{-5}$. Then we have
    \begin{multline*}
        \BB_1=Y_1 X_1^{-1}Y_1 X_1^{-2} Y_1 X_1^{-2}=\kappa_1^{-5}(\kappa_1^{3}D_1G_{1,3}^{-1} G_{1,2}^{-1})\kappa_1^{-5}(\kappa_1^{3}D_1G_{1,3}^{-1} G_{1,2}^{-1})^2\kappa_1^{-5}(\kappa_1^{3}D_1G_{1,3}^{-1} G_{1,2}^{-1})^2
        \\
        =(\kappa_1^{-2}D_1G_{1,3}^{-1} G_{1,2}^{-1}\kappa_1^2)(\kappa_1^{-4}D_1G_{1,3}^{-1} G_{1,2}^{-1}\kappa_1^4)(\kappa_1^{-1}D_1G_{1,3}^{-1} G_{1,2}^{-1}\kappa_1^1)(\kappa_1^{-3}D_1G_{1,3}^{-1} G_{1,2}^{-1}\kappa_1^3)(D_1 G_{1,3}^{-1} G_{1,2}^{-1}),
    \end{multline*}
    \begin{multline*}
        \BB_2=T_1 \BB_1T_1= G_{1,2}s_1 \kappa_1^{-5}(\kappa_1^{3}D_1 G_{1,3}^{-1} G_{1,2}^{-1})\kappa_1^{-5}(\kappa_1^{3}D_1G_{1,3}^{-1} G_{1,2}^{-1})^2\kappa_1^{-5}(\kappa_1^{3}D_1G_{1,3}^{-1} G_{1,2}^{-1})^2 G_{1,2} s_1
        \\
        =G_{1,2} (\kappa_2^{-2}D_2G_{2,3}^{-1} G_{2,1}^{-1}\kappa_2^2)(\kappa_2^{-4}D_2G_{2,3}^{-1} G_{2,1}^{-1}\kappa_2^4)(\kappa_2^{-1}D_2 G_{2,3}^{-1} G_{2,1}^{-1}\kappa_2^1)(\kappa_2^{-3}D_2 G_{2,3}^{-1} G_{2,1}^{-1}\kappa_2^3)D_2
        G_{2,3}^{-1},
    \end{multline*}
    \begin{multline*}
        \BB_3=T_2T_1 \BB_1T_1T_2
        \\= G_{2,3}s_2G_{1,2}s_1 \kappa_1^{-5}(\kappa_1^{3}D_1G_{1,3}^{-1} G_{1,2}^{-1})\kappa_1^{-5}(\kappa_1^{3}D_1G_{1,3}^{-1} G_{1,2}^{-1})^2\kappa_1^{-5}(\kappa_1^{3}D_1G_{1,3}^{-1} G_{1,2}^{-1})^2 G_{1,2} s_1 G_{2,3}s_2
        \\
        =G_{2,3}G_{1,3} (\kappa_3^{-2}D_3G_{3,2}^{-1} G_{3,1}^{-1}\kappa_3^2)(\kappa_3^{-4}D_3G_{3,2}^{-1} G_{3,1}^{-1}\kappa_3^4)(\kappa_3^{-1}D_3G_{3,2}^{-1} G_{3,1}^{-1}\kappa_3^1)(\kappa_3^{-3}D_3G_{3,2}^{-1} G_{3,1}^{-1}\kappa_3^3)D_3.
    \end{multline*}
    Using Proposition \ref{prop:G kappa triangular} below, we see that all these operators are triangular.
\end{Example}
Now we proceed to the proof.

\begin{prop}\label{prop:G triangular}
The operators $G_i$ are triangular in the basis $e_\lambda$ with respect to order $\prec$.
\end{prop}
\begin{proof}

     It is sufficient to consider the case \(N=2\). In this case we will write simply $G$ omitting the index. The formulas below is just a reformulation of \eqref{eq: T1}--\eqref{eq: T4}. Recall that $l = \mb + nk + s$ for $k \geq 0$ and $s =0, \dots n-1$. 
    For $s=0$
    \begin{align}
        G ( e_l \otimes e_{\mb}) &= v e_l \otimes e_{\mb} + (v-v^{-1})\sum_{j=1}^k  e_{l-nj} \otimes e_{\mb+nj} \quad \quad \text{for $k \geq 0$,} \label{eq: G1}
        \\
        \label{eq: G2}
        G (e_{\mb} \otimes e_l) &= v^{-1} e_{\mb} \otimes e_l - (v-v^{-1})\sum_{j=1}^{k-1} e_{l-nj} \otimes e_{\mb+nj}  \quad \quad \text{for $k > 0$.}
    \end{align}
For $s>0$
    \begin{align}    
    	\label{eq: G3}    
        G ( e_l \otimes e_{\mb}) &= e_l \otimes e_{\mb} + (v-v^{-1})\sum_{j=0}^{k} e_{\mb+nj} \otimes e_{l-nj},
        \\
        G (e_{\mb} \otimes e_l) &= e_{\mb} \otimes e_l - (v-v^{-1})\sum_{j=1}^{k} e_{l-nj} \otimes e_{\mb+nj}. \label{eq: G4}
    \end{align}
The formulas \eqref{eq: G1}--\eqref{eq: G4} imply that $G$ is triangular.
\end{proof}

\begin{prop}\label{prop:G kappa triangular}
    For $i<j$, the operators $\kappa_i^{-\dt}G_{i,j}\kappa_i^{\dt}$, $0 \leq \dt <n$ and $\kappa_j^{-\dt}G_{j,i}\kappa_j^{\dt}$, $0 < \dt <n$ are triangular in the basis $e_\lambda$ with respect to order $\prec$. The operators $\kappa_i^{-\dt} D_i \kappa_i ^{\dt}$ are diagonal for any $\dt$.
\end{prop}
\begin{proof}
    It is sufficient to consider $N=2$ and the operators $\kappa_1^{-\dt}G\kappa_1^{\dt}$, $0 \leq \dt <n$ and $s_1 \kappa_1^{-\dt}G \kappa_1^{\dt} s_1$, $0 < \dt <n$. Everything follows from \eqref{eq: G1}--\eqref{eq: G4}.
\end{proof}

\begin{thm}\label{thm:Bi triang}
    The operators $\BB_1,\dots, \BB_N$ are triangular in the basis $e_\lambda$ with respect to order $\prec$.
\end{thm}
\begin{proof}
	Recall that $n''=n/d$, $n'=n_{tw}/d$. Using Lemma \ref{lemma:B_i} we can write 
\begin{equation}
    \BB_i = T_{i-1}\cdots T_1 Z_1\cdots Z_{n''+n'}T_1\cdots T_{i-1}.
\end{equation}
Now we substitute $Y_1=\kappa_1^{-n}$, $X_1^{-1}=\kappa_1^{n_{tw}}D_1G_{1,N}^{-1}\cdots G_{1,2}^{-1}$ and 
\begin{align*}
    T_{i-1}\cdots T_1&=G_{i-1,i}\cdots G_{1,i}s_{i-1}\cdots s_1,\\
    X_1^{-1}T_1\cdots T_{i-1}&=\kappa^{n_{tw}}_1 D_1  s_1 \cdots s_{i-1}  G_{i,N}^{-1}\cdots G_{i,i+1}^{-1}.
\end{align*}
Hence we get 
\begin{equation} \label{eq: BB via G}
    \BB_i= G_{i-1,i}\cdots G_{1,i} \bigg(\prod_{\substack{j < n'+n'' \\  Z_j=X_1^{-1}}} \kappa_i^{-\dt_j}D_i G_{i,N}^{-1}\cdots G_{i,1}^{-1} \kappa_i^{\dt_j}\bigg)D_i G_{i,N}^{-1}\cdots G_{i,i+1}^{-1},
\end{equation}
for certain integers $\dt_j$. Let $\{ x\}$ denote the fractional part of $x \in \mathbb{R}$. One can observe that
\begin{multline} \label{eq: formula for dj}
    \dt_j=n \big|\{s|s<j, Z_s=Y_1 \} \big|-n_{tw} \big|\{s|s\leq j, Z_s=X_1^{-1} \} \big|
    \\
    = n\left(j-\left\lfloor\frac{j n''}{n''+n'}\right\rfloor\right)-n_{tw}\left\lfloor \frac{j n''}{n''+n'}\right\rfloor   =(n+n_{tw})  \left\{ \frac{jn''}{n''+n'} \right\}.
\end{multline}
Here $j$ is such that $Z_j = X_1^{-1}$, hence $\left\{ \frac{jn''}{n''+n'} \right\}<\frac{n''}{n''+n'}$ by the condition \eqref{eq:Bi in X Y}. Hence $0 < \dt_j < n$. Therefore Proposition \ref{prop:G kappa triangular} implies, that the operator is triangular. 
\end{proof}

\begin{corol} \label{corol: eigenvectors twisted case}
If $n_{tw}=n'$ is coprime with $n$, then there are eigenvectors $\tilde{E}_{\lambda}= e_\lambda+\sum_{\mu \prec \lambda} \beta_{\lambda,\mu} e_\mu$ of $\BB_1,\dots,\BB_N$ in  $\Ch^{(n,n')}_{u_0, \dots, u_{n-1}}$ with eigenvalues given by 
\begin{equation} \label{eq: th eigen values of BBi on Etildelambda}
   \BB_i \tilde{E}_{\lambda} = u_0 \cdots u_{n-1} q^{1-n} v^{\lt_i } q^{\lambda_i} \tilde{E}_{\lambda}	
\end{equation}
\end{corol}
\begin{proof}
Theorem \ref{thm:Bi triang} is equivalent to the following formula
 \begin{equation} \label{eq: th eigen values of BBi on elambda}
        \BB_i e_{\lambda} = b_{\lambda, \lambda} e_{\lambda} + \sum_{\mu \prec \lambda} b_{\lambda, \mu} e_{\mu}
    \end{equation}
It remains to compute $b_{\lambda, \lambda}$ using \eqref{eq: BB via G}. It follows from \eqref{eq: formula for dj} that the numbers $\dt_j$ are distinct and form a set $\{1,\dots,n-1\}$. Hence it remains to compute the diagonal terms in the action of the operators 
\begin{multline} \label{eq: triangular terms}
	a)\; G_{j,i} \text{ for } j<i; \qquad b)\; \kappa_i^{-\dt} D_i \kappa_i^{\dt} \text{ for } 0 \leq \dt<n,\; \qquad \\ c) \kappa_i^{-\dt} G_{i,j}^{-1} \kappa_i^{\dt} \text{ for  $0<\dt<n$, $j\neq i$,} \qquad d)\; G_{i,j}^{-1} \text{ for } j>i.
\end{multline}
Using formulas \eqref{eq: G1}-\eqref{eq: G4}, we get 
\begin{multline*}
	b_{\lambda, \lambda} =v^{|\{j<i|\lambda_j\geq \lambda_i, \lambda_j \equiv \lambda_i\}|-|\{j<i|\lambda_j<\lambda_i, \lambda_j \equiv \lambda_i\}|}   (u_0\cdots u_{n-1}) q^{\lambda_i+1-n} 	 v^{|\{j|\lambda_j> \lambda_i, \lambda_j \not \equiv \lambda_i\}|-|\{j|\lambda_j < \lambda_i, \lambda_j \not \equiv \lambda_i\}|} \\ 
	  \times v^{|\{j>i|\lambda_j> \lambda_i, \lambda_j \equiv \lambda_i\}|-|\{j>i|\lambda_j\leq \lambda_i, \lambda_j \equiv \lambda_i\}|} 
	=  u_0 \cdots u_{n-1} q^{1-n} v^{\lt_i} q^{\lambda_i}
\end{multline*}
where $\equiv$ stands for $\equiv \pmod{n}$.
\end{proof}

\subsection{Monomial basis} \label{ssec: monomial basis}
Recall $d = \gcd(n, n_{tw})$.
\begin{lemma} \label{lemma: factorization for A}
	The operators $\AB_i$ can be presented in the form $\AB_i= \AB_i' \kappa_i^{-d}\AB_i''$, where the operators $\AB_i'$, $\AB_i''$ are compositions of operators of the form \eqref{eq: triangular terms}.	
\end{lemma}
\begin{Remark}
In particular, the operators $\AB_i'$, $\AB_i''$ are triangular and the diagonal matrix entries are invertible in $\mathbb{Z}[q^{\pm 1}, v^{\pm 1}]$ by Proposition \ref{prop:G kappa triangular}. For technical reasons, we will need that operators  the $\AB_i'$, $\AB_i''$ are compositions of operators of the form \eqref{eq: triangular terms} in the proof of Theorem \ref{thm:A lmabda basis}. It will be important that the operators of the form \eqref{eq: triangular terms} act at most on two tensor multiples.
\end{Remark}
The proof of the lemma is similar to the proof of Theorem \ref{thm:Bi triang}. 
\begin{proof}
	Using \eqref{eq:BBi ABi} we get
	\begin{multline}\label{eq:Ai in G}
		\AB_i=T_{i-1}\cdots T_1 W_1\cdots W_{m+m'}T_1\cdots T_{i-1}\\
		=G_{i-1,i}\cdots G_{1,i} \bigg(\prod_{\substack{j < m+m', \\ W_j=X_1^{-1}}} \kappa_i^{-\ct_j}D_i G_{i,N}^{-1}\cdots G_{i,1}^{-1} \kappa_i^{\ct_j}\bigg) \kappa_i^{-\ct_{m+m'}} D_i G_{i,N}^{-1}\cdots G_{i,i+1}^{-1},
	\end{multline}
	for certain integers $\ct_j$. It remains to compute the numbers $\ct_j$. We have 
	\begin{multline}
		\ct_j=n \big|\{s|s<j, W_s=Y_1 \} \big|-n_{tw} \big|\{s|s\leq j, W_s=X_1^{-1} \} \big|
		\\
		= n\left(j-\left\lfloor\frac{j m}{m+m'}\right\rfloor\right)-n_{tw}\left\lfloor \frac{j m}{m+m'}\right\rfloor   =d\left((n''+n')  \left\{ \frac{jm}{m+m'} \right\}+\frac{j}{m+m'}\right),
	\end{multline}
	Note that $\ct_{m+m'}= d$. Let $j<m+m'$. Since $W_j=X_1^{-1}$, we have $0 \leq \left\{ \frac{jm}{m+m'} \right\} \leq \frac{m-1}{m+m'}$. Then
	\[ 
		(n''+n')  \left\{ \frac{jm}{m+m'} \right\}\leq (n''+n')\frac{m-1}{m+m'}=\frac{n''m+n''m'-1-n''-n'}{m+m'}<n''-1.
	\]
	Hence $0<\ct_j<n$. 
\end{proof}

Let us now assume that $d=1$, i.e. $n=n''$, $n_{tw}=n'$. In this case, the operators $\AB_i$ increase the grading by $1$. For any $\lambda \in \mathbb{Z^N}$ denote 
\begin{equation}\label{eq:Abasis}
	A_\lambda= \AB_1^{\lambda_1}\cdots \AB_N^{\lambda_N} e_{0}\otimes \dots \otimes e_0
\end{equation}

\begin{thm} \label{thm:A lmabda basis}	
	The vectors $A_\lambda$ form a basis of $(\mathbb{C}^n[Y^{\pm 1}])^{\otimes N}$. The transition matrix from $e_\lambda$ basis to $A_\lambda$ basis is triangular. Moreover, we have 
\begin{equation}\label{eq:A in e basis}
		A_\lambda=\alpha_{\lambda,\lambda} e_\lambda+\sum_{\mu \prec \lambda} \alpha_{\lambda,\mu} e_\mu
\end{equation}
where the coefficients $\alpha_{\lambda,\mu} \in \mathbb{Z}[q^{\pm 1},v^{\pm 1}]$ and $\alpha_{\lambda,\lambda}$ is invertible in $\mathbb{Z}[q^{\pm 1},v^{\pm 1}]$.
\end{thm}
\begin{proof}
	Since $\AB_i$ is expressed via the operators $T_j^{\pm 1}, Y_j^{\pm 1}, \pi^{\pm 1}$, its matrix elements in $e_\lambda$ basis belong to $\mathbb{Z}[q^{\pm 1},v^{\pm 1}]$. Hence, the vectors $A_\lambda$ expand in $e_\lambda$ basis with coefficients in $\mathbb{Z}[q^{\pm 1},v^{\pm 1}]$. 
	
	The product $\AB_1\cdots \AB_N$ is a combination of products $Y_1\cdots Y_N=\kappa_1^{-n}\cdots \kappa_N^{-n}$ and $X_1^{-1}\cdots X_N^{-1}=\pi^{N}=\kappa_1^{n'}\cdots \kappa_N^{n'} D_1\cdots D_N$. Notice that
	\begin{align}
	    \AB_1\cdots \AB_N e_{(\lambda_1, \dots, \lambda_N)} =& c_{\lambda}e_{(\lambda_1+1, \dots, \lambda_N+1)} &
	    \AB_1\cdots \AB_N A_{(\lambda_1, \dots, \lambda_N)} =& A_{(\lambda_1+1, \dots, \lambda_N+1)}
	\end{align}
	for certain invertible $c_{\lambda} \in \mathbb{Z}[q^{\pm 1}, v^{\pm 1}]$. Therefore it is sufficient to prove the formula \eqref{eq:A in e basis} for the compositions such that $\lambda_1,\dots,\lambda_N \geq 0$.
	
	Let $l=\max \lambda_j$. We proceed by induction on $l$. For $l=0$ there is nothing to prove. The induction step is $l-1 \rightarrow l$. Let $1\leq i_1<\dots < i_k \leq N$ be a subset of indices such that $\lambda_{i_1}= \dots = \lambda_{i_k}=l$ and $\lambda_j<l$, for $j \not \in \{i_1,\dots,i_k\}$. 
	Let $\lambda^{(s)}$, $0\leq s \leq k$ be a composition such that \[\lambda_{j}^{(s)}=\left\{ \begin{aligned} 
		&l-1&\quad &\text{ for } j=i_1,\dots, i_s \\
		&l&\quad &\text{ for } j=i_{s+1},\dots, i_k \\
		&\lambda_{j}&\quad &\text{ for } j \not \in \{i_1,\dots, i_k\}.
	\end{aligned} \right.\]
	For example, $\lambda^{(0)}=\lambda$. By the induction hypothesis we know that $A_{\lambda^{(k)}}$ is a linear combination of $e_{\mu}$ with $\mu \preceq \lambda^{(k)}$ and the coefficient $\alpha_{\lambda^{(k)},\lambda^{(k)}}$ is invertible. Now we prove by induction on $s$ that 
	\[
		A_{\lambda^{(s)}} = \AB_{i_{s+1}}\cdots \AB_{i_{k}} A_{\lambda^{(k)}}
	\]
	satisfies condition \eqref{eq:A in e basis} with an additional constraint on $\mu$ appearing in the right-hand side
	\begin{equation}
    \text{$\mu_j<l$ for $j<i_{s+1}$}. 
	\end{equation}
	The induction base is $s=k$, the induction step is $s \mapsto s-1$. The step follows from Lemma \ref{lemma: factorization for A} for the operator $\AB_{i_s}$. Indeed, the triangular operators of the form \eqref{eq: triangular terms} have invertible elements on the diagonal and cannot make $\mu_j=l$ for $j<i_{s}$.
\end{proof}
\begin{corol} \label{color: e via A}
	The transition matrix from $A_\lambda$ basis to $e_\lambda$ basis is triangular. Moreover, we have $e_\lambda=\sum_{\mu \preceq \lambda} \tilde{\alpha}_{\lambda,\mu} A_\mu$ where the coefficients $\tilde{\alpha}_{\lambda,\mu} \in \mathbb{Z}[q^{\pm 1},v^{\pm 1}]$ and $\tilde{\alpha}_{\lambda,\lambda}$ is invertible in $\mathbb{Z}[q^{\pm 1},v^{\pm 1}]$.
\end{corol}
\begin{proof}
	Recall the transition matrix $\alpha$ defined by \eqref{eq:A in e basis}. It follows from Theorem \ref{thm:A lmabda basis} that $\tilde{\alpha}= \alpha^{-1}$ satisfies the properties in question.
\end{proof}
\subsection{Twisted Cherednik representation}
Below we assume $d=1$, i.e. $n_{tw}=n'$. In this case, the operators $\AB_i$ increase the grading by $1$. For any $\mathcal{H}_N$-module $M$ denote by $\rho_{M} \colon \mathcal{H}_N \rightarrow \End_{\mathbb{C}} (M)$ the corresponding homomorphism.
\begin{defin}
For any $\mathcal{H}_N$-module $M$ and $\tau \in \widetilde{SL}(2,\mathbb{Z})$, let us define the representation $M^\tau$ as follows. $M$ and $M^{\tau}$ are the same vector space with different actions, namely $\rho_{M^{\tau}} = \rho_M \circ \tau^{-1}$.
\end{defin}
	We will refer to $M^{\tau}$ as a \emph{twisted representation}.
\begin{thm}\label{thm:isomorphism with twisted}
	The representation $\Ch^{(n,n')}_{u_0, \dots, u_{n-1}}$ is isomorphic to twisted Cherednik representation $\Ch_{u}^{\sigma}$ for $\sigma$ as in \eqref{def sigma 1}, \eqref{def sigma 2} and $u =  u_0 \cdots u_{n-1} q^{1-n}$.
\end{thm}
\begin{proof}
	Let $H^{\BB}$ be a copy of affine Hecke algebra generated by $T_i$ and $\BB_i$. Twisted Cherednik representation $\Ch_{u}^{\sigma}$ can be interpreted as a $\mathcal{H}_N$-representation induced from one-dimensional representation of $H^{\BB}$. As a vector space, $\Ch_{u}^{\sigma}$ is isomorphic to the space of Laurent polynomials $\mathbb{C}[\AB_1^{\pm 1},\dots, \AB_n^{\pm 1}]$. 
	
	The vector $e_{(0)^N} \in \Ch^{(n,n')}_{u_0, \dots, u_{n-1}}$ is an eigenvector of $T_1,\dotsm T_{N-1}$. Moreover, due to Theorem \ref{thm:Bi triang}, the vector $e_{(0)^N}$ is an eigenvector of $\BB_1,\dots, \BB_N$. Corollary \ref{corol: eigenvectors twisted case} implies that the eigenvalues are given by
	\begin{equation} \label{eq: Bi eigevalues}
	\BB_i e_{(0)^N} = u_0 \cdots u_{n-1} q^{1-n} v^{2i-1-N } e_{(0)^N}.
	\end{equation}
	Comparing \eqref{eq: Cherednik eigenvalues} with \eqref{eq: Bi eigevalues}, we see that there is a homomorphism $\psi \colon \Ch_{u}^{\sigma} \rightarrow \Ch^{(n,n')}_{u_0, \dots, u_{n-1}}$ determined by $\psi(1) = e_{(0)^N}$.
	
	The twisted Cherednik representation $\Ch_{u}^{\sigma}$ has a basis $\AB_{\lambda_1} \cdots \AB_{\lambda_N} 1$. On the other hand, it follows from Theorem \ref{thm:A lmabda basis} that their images form a basis of $\Ch^{(n,n')}_{u_0, \dots, u_{n-1}}$. Hence the map $\psi$ is an isomorphism.	
\end{proof}
\begin{Remark}
	There is another way to finish the proof without using Theorem \ref{thm:A lmabda basis}. Namely, since Cherednik representation is irreducible, the map $\psi$ is injective. It remains to show that $e_{(0)^N}$ is a cyclic vector of $\Ch^{(n,n')}_{u_0, \dots, u_{n-1}}$ for the $\mathcal{H}_N$-action.
\end{Remark}

\begin{corol} \label{corol: depnds on product}
The isomorphism class of the representation $\Ch^{(n,n')}_{u_0, \dots, u_{n-1}}$ is determined by $n$, $n'$, and the product $u_0 \cdots u_{n-1}$. 
\end{corol}

\section{Toroidal algebra}\label{sec:toroidal}

In this section, we recall presentations and certain properties of  quantum toroidal $\mathfrak{gl}_1$ algebra denoted by $\DI$. In particular, we describe its connection with double affine Hecke algebra $\mathcal{H}_N$. The section contains no new results.

\paragraph{PBW presentation and $\widetilde{SL}(2, \mathbb{Z})$-action} The algebra $\DI$ is an algebra depending on parameters $q_1$ and $q_2$. Let us introduce a parameter $q_3$ such that $q_1 q_2 q_3 =1$. The algebra has a presentation via generators $P_{a,b}$ for $(a,b) \in \mathbb{Z}^2 \backslash  \{ (0,0) \}$ and central elements $c^{\pm1}$, $\left(c'\right)^{\pm 1}$. We will not need explicit form of the relations, see \cite[Def. 6.4]{Burban:2012} for a reference. \emph{Loc. cit.} generators $u_{a,b}$ correspond to \(P_{a,b}/(1-q_1^{\mathbf{d}})\) for \(\mathbf{d}=\gcd(a,b)\).

\begin{prop}[\cite{Burban:2012}]
    Group $\widetilde{SL}(2, \mathbb{Z})$ acts on $\DI$ via automorphisms.
\end{prop}
Let us consider an element $\tau \in \widetilde{SL}(2, \mathbb{Z})$ such that under the projection \eqref{eq: projection SL2} it is mapped
\begin{equation} \label{eq: tau matrix}
    \tau \mapsto \begin{pmatrix}
                    m'& m \\
                    n'& n
                    \end{pmatrix}.
\end{equation}
Let $\widetilde{SL}(2, \mathbb{R})$ be the universal covering of $SL(2, \mathbb{R})$. The group $\widetilde{SL}(2, \mathbb{Z})$ can be interpreted as the preimage of $SL(2, \mathbb{Z})$ in $\widetilde{SL}(2, \mathbb{R})$. Hence we can think about the element $\tau \in \widetilde{SL}(2, \mathbb{Z})$ as a path $\gamma$ in $SL(2, \mathbb{R})$ from the identity matrix to the matrix \eqref{eq: tau matrix}. The path $\gamma$ induces a path $\gamma(a,b)$ in $\mathbb{R}^2 \backslash \{ 0, 0 \}$ by action on $(a,b)$. The intersection number of $\gamma(a,b)$ and the line $a=0$ is called  \emph{winding number} $n_{\tau}(a,b)$.\footnote{if the path \(\gamma(a,b)\) goes counterclockwise then \(\tau(a,b)\) is included but \((a,b)\) is not included into the path \(\gamma(a,b)\). In the clockwise case, the convention is opposite. See a recent paper \cite[Sec 3.2.1]{BHMPS:2021} for a treatment of this group action, though the conventions in \emph{loc. cit}. differ from ours.}

Then the action of $\tau$ is given by the following formulas
\begin{align}
    &\tau(c) = c^n \left(c' \right)^m,   &   \tau(c') =& c^{n'} \left(c' \right)^{m'},  \label{eq:sigma central charge}\\
    &\tau \left( P_{a,b} \right) = \left( \left(c' \right)^{m'a+mb} c^{n'a+nb} \right)^{n_{\tau}(a,b)}P_{m'a+mb, n'a+nb}. \label{eq: sigma action on Pab} 
\end{align}
\begin{Remark}
    Another remarkable property of the generators $P_{a,b}$ is that an analog of PBW theorem holds with respect to the generators. This is the reason for the term \emph{PBW presentation}.  We will formulate an appropriate analog of PBW theorem (see Proposition \ref{prop: PBW for DI plus}) and use this in the proof of Theorem \ref{thm: explict construction of twisted Fock}. 
\end{Remark}
\paragraph{Chevalley presentation} The algebra has another presentation, the equivalence between them was shown in \cite{Schiffmann:2012}. The generators are $P_{1,b}$, $P_{-1,b}$ for $b \in \mathbb{Z}$,  $P_{0,k}$ for $k \in \mathbb{Z}_{\neq 0}$, and central elements $c^{\pm 1}$, $\left( c' \right)^{\pm 1}$. To describe the relations let us introduce currents (formal power series with coefficients in the algebra $\DI$)
\begin{align}
    &E(z) = \sum_{b\in \mathbb{Z}} P_{1,b} z^{-b}, & F(z) =& \sum_{b \in \mathbb{Z}} P_{-1,b} z^{-b}.
\end{align}
Define
\begin{equation}
    \sum_{k>0} \theta_{\pm k} z^{-k} = \exp \left( \sum_{k>0} \frac{(1-q_2^k)(1 - q_3^{k})}{k} P_{0, \pm k} z^{-k}\right).
\end{equation}
\begin{subequations}
The relations are the following. For $k, l \in \mathbb{Z}$
\begin{align} \label{eq: relation Heisenberg}
     [P_{0,k}, P_{0,l}] =&  k \frac{(1-q_1^{|k|})(c^{|k|} - c^{-|k|})}{(1-q_2^{|k|})(1  - q_3^{|k|})} \delta_{k+l,0}.
\end{align}
For $k \in \mathbb{Z}_{>0}$ and $b \in \mathbb{Z}$
\begin{align}
    [P_{0,k}, P_{1,b}] =& c^{-k}(q^k_1-1) P_{1, b+k}, &  [P_{0,-k}, P_{1,b}] =& (1-q_1^k) P_{1, b-k},  \label{eq: relation HeisenbergE}  \\
     [P_{0,k}, P_{-1,b}] =& (1-q_1^k) P_{-1, b+k}, & [P_{0,-k}, P_{-1,b}] =& (q_1^k-1) c^k P_{-1, b-k}.    \label{eq: relation HeisenbergF}
\end{align}
\begin{align}
(z-q_1 w)(z-q_2 w)(z-q_3 w) E(z) E(w) =& -(w-q_1 z)(w-q_2 z)(w-q_3 z) E(w) E(z), \label{bilinearE} \\
(z-q_1^{-1} w)(z-q_2^{-1} w)(z-q_3^{-1} w) F(z) F(w) =& -(w-q_1^{-1} z)(w-q_2^{-1} z)(w-q_3^{-1} z) F(w) F(z).  \label{bilinearF}
\end{align}
For $a+b>0$
\begin{align}
    [P_{1,a}, P_{-1,b}] =& \frac{(1-q_1)c^{a} c'}{(1-q_2)(1- q_3)}  \theta_{a+b}, &  [P_{1,-a}, P_{-1,-b}] =& -\frac{(1-q_1) c^{-b} (c')^{-1}}{(1-q_2)(1- q_3)}  \theta_{ -a-b}. \label{eq: commutator E and F 1}
\end{align}
For $a \in \mathbb{Z}$
\begin{align}
    &[P_{1,a}, P_{-1,-a}] = \frac{(1-q_1)\left(c^{a} c' - c^{-a} (c')^{-1}\right)}{(1-q_2)(1-q_3)},  \label{eq: relation central extention} \\
    &[P_{1,a}, [P_{1,a-1}, P_{1, a+1}]] =0,  \label{SerreE} \\
    &[P_{-1,a}, [P_{-1,a-1}, P_{-1, a+1}]]=0. \label{SerreF}
\end{align}
\end{subequations}

\begin{defin}
Algebra $\DI$ is an algebra generated by $P_{1,b}$, $P_{0,b}$, $P_{-1,b}$ for $b \in \mathbb{Z}$  and central elements $c^{\pm1}$, $\left(c'\right)^{\pm1}$ with the relations \eqref{eq: relation Heisenberg}--\eqref{SerreF}.
\end{defin}

\begin{defin}
Algebra $\DIp$ is an algebra generated by $P_{1,b}$, $P_{0,b}$ for $b \in \mathbb{Z}$ and central elements $c^{\pm1}$, $\left(c'\right)^{\pm1}$ with the relations \eqref{eq: relation Heisenberg}, \eqref{eq: relation HeisenbergE}, \eqref{bilinearE}, \eqref{SerreE}.
\end{defin}

\begin{defin}
Algebra $\DIm$ is an algebra generated by $P_{-1,b}$, $P_{0,b}$ for $b \in \mathbb{Z}$ and central element $c^{\pm1}$, $\left(c'\right)^{\pm1}$ with the relations \eqref{eq: relation Heisenberg}, \eqref{eq: relation HeisenbergF}, \eqref{bilinearF}, \eqref{SerreF}.
\end{defin}

\begin{prop}
The algebras  $\DIp$ and $\DIm$ are subalgebras of $\DI$.
\end{prop}

\paragraph{Connection with spherical DAHA} 
Denote
\begin{align}
    [i]_v^{\pm} =& \frac{v^{\pm 2i}-1}{v^{\pm 2}-1},  &  [i]_v =& \frac{v^{i}-v^{-i}}{v-v^{-1}},\\
    [k]!_v^{\pm}  =& [1]_v^{\pm} \cdots [k]_v^{\pm},  &  [k]!_v  =& [1]_v \cdots [k]_v.
\end{align}
In this paragraph we will need to consider double affine Hecke algebras for different parameters $q$ and $v$, therefore we will write $\mathcal{H}_N(q,v)$. Let $\bc_+$ and $\bc_-$ be the symmetrizer and the anti-symmetrizer in finite Hecke algebra
\begin{align}
    \bc_+ =& \frac{1}{[N]!_v^+} \sum v^{l(\sigma)} T_{\sigma}, & \bc_- =& \frac{1}{[N]!_v^-} \sum (-v)^{-l(\sigma)} T_{\sigma}.
\end{align}
The basic property of $\bc_{\pm}$ is that for $i= 1, \dots, N-1$
\begin{align} \label{eq: basic property of (anti)symmetrizer}
    T_i \bc_+ =& \bc_+ T_i = v \bc_+,  &    T_i \bc_- =& \bc_- T_i = - v^{-1} \bc_-.
\end{align}
Let $\SHpm(q,v) = \bc_{\pm} \mathcal{H}_N(q,v) \bc_{\pm}$ be the corresponding spherical DAHA.

\begin{prop}
    There is an algebra isomorphism $\SHm(q,v) \cong \SHp(q, -v^{-1})$.
\end{prop}
\begin{proof}
  The relations \eqref{eq: DAHA relation 1}--\eqref{eq: DAHA relation 4} imply that there is an isomorphism $\mu \colon \mathcal{H}_N(q,v) \xrightarrow{\sim} \mathcal{H}_N(q, -v^{-1})$ defined by $\mu(T_i) = T_i$, $\mu(Y_i)= Y_i$, $\mu(\pi)=\pi$. To finish the proof we note that $\mu(\bc_-) = \bc_{+}$.
\end{proof}

\begin{thm}[\cite{SV11}]
	The following formulas determine a surjection of the algebra $\DI$ onto $\SHp(q,v)$ for $q_1 = q$, $q_2 =v^{2}$ 	
	\begin{subequations}
	\begin{align}
		P_{0,k}^{(N)} =& \bc_+(Y_1^k + \dots + Y_N^k) \bc_+, & P_{0,-k}^{(N)} =& q^k  \bc_+(Y_1^{-k} + \dots + Y_N^{-k}) \bc_+,\\
		P_{k,0}^{(N)} =&q^k \bc_+(X_1^k + \dots + X_N^k) \bc_+, & P_{-k,0}^{(N)} =&   \bc_+(X_1^{-k} + \dots + X_N^{-k}) \bc_+,\\
		P_{1,b}^{(N)} =& q [N]^-_v \bc_+ X_1 Y_1^b \bc_+, & P_{-1,b}^{(N)} =& [N]^+_v \bc_+ Y_1^b X_1^{-1} \bc_+,
	\end{align}
	\end{subequations}
	here $k\in \mathbb{Z}_{>0}$, $b \in \mathbb{Z}$, and the image of $P_{a,b}$ is denoted by $P_{a,b}^{(N)}$. 
\end{thm}

\begin{Remark}
	In \cite{SV11} the authors prove that the quotient of $\DI^{\pm}$ by the relations $c=c'=1$ is a projective limit of the corresponding subalgebras of $\SHp$. This deep result is one of the motivations for the limit $N \rightarrow \infty$ to be studied below. Though formally speaking, we will not use the mentioned result of \emph{loc. cit.}
\end{Remark}

%

\begin{corol} \label{corollary: P}
Algebra $\DI$ surjects onto $\SHm(q,v)$ for $q_1=q$, $q_2 =v^{-2}$. Moreover
	\begin{subequations}
	\begin{align}
     \label{eq:P{0,k}N}
     P_{0,k}^{(N)} =& \bc_-(Y_1^k + \dots + Y_N^k) \bc_-, & P_{0,-k}^{(N)} =& q^k  \bc_-(Y_1^{-k} + \dots + Y_N^{-k}) \bc_-,  \\
      \label{eq:P{k,0}N}
     P_{k,0}^{(N)} =&q^k \bc_-(X_1^k + \dots + X_N^k) \bc_-, & P_{-k,0}^{(N)} =&   \bc_-(X_1^{-k} + \dots + X_N^{-k}) \bc_-,\\
     \label{eq:P{1,k}N}
     P_{1,b}^{(N)} =& q [N]^+_v \bc_- X_1 Y_1^b \bc_-, & P_{-1,b}^{(N)} =& [N]^-_v \bc_- Y_1^b X_1^{-1} \bc_-.      
    \end{align}
	\end{subequations}
\end{corol}

\section{Deformed exterior power} \label{sec:wedge}
Spherical DAHA $\SHm(q,v)$ acts on the subspace $\bc_- \Ch_{u_0, \dots, u_{n-1}}^{(n,n')} = \bc_- \left(\mathbb{C}^n[Y^{\pm1}]\right)^{\otimes N} \subset \left(\mathbb{C}^n[Y^{\pm1}]\right)^{\otimes N}$. The space $\bc_- \left(\mathbb{C}^n[Y^{\pm1}]\right)^{\otimes N}$ was considered in \cite{KMS95, LT00}. In \emph{loc. cit.}, the authors considered only affine Hecke algebra action on $ \left(\mathbb{C}^n[Y^{\pm1}]\right)^{\otimes N}$, but not DAHA. In this section, we will recall and extend their results. Spherical DAHA will be considered in the subsequent sections.
\subsection{Finite $v$-wedge} 
The $v$-deformed exterior power can be defined as a subspace $\bc_- \left(\mathbb{C}^n[Y^{\pm1}]\right)^{\otimes N}$. On the other hand it can be identified with the quotient space via tautological projection
\begin{align}
	\bc_- \left(\mathbb{C}^n[Y^{\pm1}]\right)^{\otimes N} &\xrightarrow{\sim} \left(\mathbb{C}^n[Y^{\pm1}]\right)^{\otimes N} \Big{/} \sum_i \Image (T_i+v^{-1})
\end{align}
The inverse map is induced by $\bc_-$. We will use both interpretations as the subspace and as the quotient. Denote by $e_{i_1} \wedge \dots \wedge e_{i_N} =\bc_- \left( e_{i_1} \otimes \dots \otimes e_{i_N} \right)$.
\begin{lemma}[{\cite[eq. (41), (42)]{KMS95}}] \label{lemma: q-wedge relation}
    Let $l = \mb + nk + s$ for $k \geq 0$ and $s= 0, \dots, n-1$. Then
	\begin{subequations}
	    \begin{align}
	        e_l \wedge e_{\mb} =&- e_{\mb} \wedge e_l & &\text{for $s=0$}
	        \\
	        e_l \wedge e_{\mb} =&- v e_{\mb} \wedge e_l & &\text{for $k=0$}
	        \\
	        e_l \wedge e_{\mb} =&- v e_{\mb} \wedge e_l - e_{l-nk} \wedge e_{\mb+nk}- v e_{\mb + nk} \wedge e_{l - nk} & &\text{otherwise}
	    \end{align}
	\end{subequations}
    The above identities can be used for vectors of the form $e_{i_1} \wedge \dots \wedge e_{l} \wedge e_{\mb} \wedge \dots \wedge e_{i_N}$.
\end{lemma}

\begin{prop}[{\cite[Prop. 1.3]{KMS95}}]
	The vectors $e_{i_1} \wedge \dots \wedge e_{i_N}$ for $i_1<i_2<\dots <i_N$ form a basis of $\bc_- ( \mathbb{C}^n[Y^{\pm1}])^{\otimes N}$.
\end{prop}

\begin{lemma}[{\cite[Lemma 7.6]{LT00}}] \label{Pauli exclusion}
Let $k_1, \dots, k_N$ be integers such that $\sum_{i=1}^N (i -\mt -k_i) < N $ and all $k_i<N-\mt$ for certain $\mt \in \mathbb{Z}$. Then $e_{k_1} \wedge \dots \wedge e_{k_N} = 0$.
\end{lemma}

\paragraph{Vertex operators}
Considering a vector \(w\) as an element of the subspace and the quotient space
\begin{align*}
&w=\sum_{t_1, \dots, t_N} y_{t_1, \dots, t_N} e_{t_1} \otimes \dots \otimes e_{t_N} \in \bc_- \left(\mathbb{C}^n[Y^{\pm1}]\right)^{\otimes N}, &
&w = \sum_{i_1 < \dots < i_N} x_{i_1, \dots, i_N} e_{i_1} \wedge \dots \wedge e_{i_N}.
\end{align*}
Let us define (modes of) vertex operators \( \Phi_k, \Psi_k \colon \bc_- ( \mathbb{C}^n[Y^{\pm1}])^{\otimes N} \rightarrow \bc_- ( \mathbb{C}^n[Y^{\pm1}])^{\otimes (N+1)}\) 
by the formula 
\begin{align}
	\Phi_k(w) =\sum_{i_1 < \dots < i_N} x_{i_1, \dots, i_N} e_k\wedge e_{i_1} \wedge \dots \wedge e_{i_N},
	\quad
	\Psi_k(w) =\sum_{i_1 < \dots < i_N} x_{i_1, \dots, i_N} e_{i_1} \wedge \dots \wedge e_{i_N} \wedge e_k,
\end{align}
Note that here \(w\) is considered as an element of the quotient. The vertex operators $\Phi_k, \Psi_k $ can also be defined in terms of the subspace. 
\begin{multline*}
	\Phi_k(w)=	\bc_-^{(N+1)} \sum_{i_1< \dots< i_N} x_{i_1, \dots, i_N} e_k \otimes e_{i_1} \otimes \dots \otimes e_{i_N}\\ = \bc_-^{(N+1)} \bc_-^{(N)} \sum_{i_1< \dots< i_N} x_{i_1, \dots, i_N} e_k \otimes e_{i_1} \otimes \dots \otimes e_{i_N} = \bc_-^{(N+1)} \sum_{t_1, \dots, t_N} y_{t_1, \dots, t_N}\, e_k \otimes e_{t_1} \otimes \dots \otimes e_{t_N},
\end{multline*}
where $\bc_-^{(N)}$ denotes anti-symmetrizer in $\mathcal{H}_N$.

Define (modes of) the dual vertex operators $\Phi_k^*, \Psi_k^* \colon \bc_- ( \mathbb{C}^n[Y^{\pm1}])^{\otimes N} \rightarrow \bc_- ( \mathbb{C}^n[Y^{\pm1}])^{\otimes (N-1)}$ by the formula
\begin{align}
	\Phi_k^*(w) =  \sum_{t_2, \dots, t_N}y_{-k,t_2 \dots, t_N}\; e_{t_2} \otimes \dots \otimes e_{t_N},\quad 
	\Psi_k^*(w) = \sum_{t_1, \dots, t_{N-1}} y_{t_1, \dots, t_{N-1},-k}\; e_{t_1} \otimes \dots \otimes e_{t_{N-1}}.
\end{align}
Note that here \(w\) is considered as an element of the subspace. It is easy to see that
\[ 
	(T_i+v^{-1})\sum_{t_2, \dots, t_N} y_{-k,t_2 \dots, t_N} \;e_{t_2} \otimes \dots \otimes e_{t_N}= 
	(T_i+v^{-1}) \sum_{t_1, \dots, t_{N-1}} y_{t_1, \dots, t_{N-1},-k} \; e_{t_1} \otimes \dots \otimes e_{t_{N-1}} =0
\]
for any $1\leq i \leq N-2$. Hence the image of $\Phi_k^*$ and $\Psi_k^*$ indeed belongs to $\bc_- \left(\mathbb{C}^n[Y^{\pm1}]\right)^{\otimes N-1}$.

Consider operators 
\begin{equation} \label{eq: def b_j}
b_j = Y_1^j + \dots + Y_N^j.
\end{equation}
\begin{lemma}
    The following relations hold for $k \in \mathbb{Z}$ and $j \in \mathbb{Z}_{\neq 0}$
    \begin{align} \label{eq: commutator bj and the vertexes}
        [b_j, \Phi_k] =& \Phi_{k+nj}, & [b_j, \Psi_k] =& \Psi_{k+nj},  &  [b_j, \Phi^*_k] =& -\Phi^*_{k+nj}, & [b_j, \Psi^*_k] =& -\Psi^*_{k+nj}.
    \end{align}
\end{lemma}
\begin{proof}
	Follows directly from the following formulas
	\begin{align*}
	b_j \, \sum_{i_1< \dots< i_N} x_{i_1, \dots , i_N} e_{i_1} \wedge \dots \wedge e_{i_N}=& \sum_{i_1< \dots< i_N} \sum_{r=1}^N x_{i_1, \dots, i_N} e_{i_1}  \wedge  \dots \wedge e_{i_r+ jn} \wedge \dots  \wedge e_{i_N} \\
		b_j \sum_{t_1, \dots, t_N} y_{t_1, \dots, t_N} e_{t_1} \otimes \dots \otimes e_{t_N} =& \sum_{t_1, \dots, t_N} \sum_{r=1}^N y_{t_1, \dots, t_N} e_{t_1} \otimes  \dots \otimes e_{t_r+ jn}\otimes \dots \otimes e_{t_N}.
	\end{align*}
\end{proof}

\paragraph{Involution} In order to prove certain properties of vertex operators we will need bar involution. It is an antilinear map \(\bar{v} = v^{-1}\) and its action on \(\left(\mathbb{C}^n[Y^{\pm1}]\right)^{\otimes N}\) given by the formula \cite[Prop. 5.5]{LT00}
\begin{align}\label{eq:bar on tensor}
	\overline{e_{t_1} \otimes \dots \otimes e_{t_N}} = v^{N(N-1)/2-\pt_{\mathbf{t}}}T_{w_0}e_{t_N} \otimes \dots \otimes e_{t_1}.
\end{align}
Here $\mathbf{t} = \{t_1, \dots, t_N \}$, $\pt_{\mathbf{t}}$ is the number of pairs $\{ t_r, t_s \}$ such that $t_r \not\equiv t_s \bmod{n}$, \(w_0\) is the longest element in Weyl group and \(T_{w_0}\) is the corresponding element of the Hecke algebra. This is an \emph{ad hoc} definition, see \cite{LT00} for more details. 

Bar involution preserves the subspace $\bc_-\!\left(\mathbb{C}^n[Y^{\pm 1}] \right)^{\otimes N}$ and acts by the formula \cite[Prop. 5.9]{LT00}
\begin{align}\label{eq:bar on wedge}
	\overline{e_{i_1} \wedge \dots \wedge e_{i_N}} =& (-1)^{N(N-1)/2}v^{-\pt_{\mathbf{i}}}e_{i_N} \wedge \dots \wedge e_{i_1}
\end{align}

\begin{lemma} \label{lemma: involution and vertex}
    The following holds
    \begin{align}
(-1)^N \overline{v^{-a_{\mathbf{i}}(k)} \Psi_k \overline{e_{i_1} \wedge \dots \wedge e_{i_N}}} =&  \Phi_k e_{i_1} \wedge \dots \wedge e_{i_N}  \label{eq: involution Psi},\\
(-1)^{N-1}\overline{v^{a_{\mathbf{i}}(k)} \Psi^*_k \overline{e_{i_1} \wedge \dots \wedge e_{i_N}}} =&  \Phi_k^* e_{i_1} \wedge \dots \wedge e_{i_N}. \label{eq: involution Psi*}
\end{align}
here $a_{\mathbf{i}}(k)$ is the number of \(r\) such that \(i_r \not\equiv k \bmod{n}\).
\end{lemma}
\begin{proof}
	For the first relation we use \eqref{eq:bar on wedge} and obtain
	\begin{multline*}
		\overline{\Psi_k \overline{e_{i_1} \wedge \dots \wedge e_{i_N}}}=	\overline{\Psi_k (-1)^{N(N-1)/2}v^{-\pt_{\mathbf{i}}}e_{i_N} \wedge \dots \wedge e_{i_1}}= (-1)^{N(N-1)/2}v^{\pt_{\mathbf{i}}} \overline{ e_{i_N} \wedge \dots \wedge e_{i_1}\wedge e_k}\\ = (-1)^{N}v^{-a_{\mathbf{i}}(k)} e_k\wedge e_{i_1}\wedge \dots \wedge e_{i_N} =(-1)^{N}v^{-a_{\mathbf{i}}(k)} \Phi_k e_{i_1}\wedge \dots \wedge e_{i_N}.
	\end{multline*}
	For the second relation we use notation \(e_{i_1} \wedge \dots \wedge e_{i_N}=\sum_{\mathbf{j}} y_{\mathbf{j}}e_{j_1}\otimes \dots \otimes e_{j_N} \). Then
	\begin{equation} \label{eq: bar for wedge sum j}
		\overline{e_{i_1}\wedge \dots \wedge e_{i_N}}=(-v)^{-N(N-1)/2}T_{w_0}^{-1}\overline{e_{i_1}\wedge \dots \wedge e_{i_N}}= (-1)^{N(N-1)/2} v^{-\pt_{\mathbf{i}}}\sum\nolimits_{\mathbf{j}} \bar{y}_{\mathbf{j}} \,e_{j_N}\otimes \dots \otimes e_{j_1}.
	\end{equation}
	Here we used relations \eqref{eq:bar on tensor} and \(\pt_{\mathbf{i}}=\pt_{\mathbf{j}}\). Therefore, using \eqref{eq: bar for wedge sum j} twice, we obtain
	\begin{multline*}
		\overline{\Psi_k^* \overline{e_{i_1} \wedge \dots \wedge e_{i_N}}}=\sum\nolimits_{\mathbf{j}} (-1)^{N(N-1)/2} v^{\pt_{\mathbf{i}}} y_{\mathbf{j}} \overline{\Psi_k^*    e_{j_N}\otimes \dots \otimes e_{j_1} }
		\\=\sum\nolimits_{j_2,\dots,j_N} (-1)^{N(N-1)/2} v^{\pt_{\mathbf{i}}} y_{-k,j_2,\dots,j_N} \overline{e_{j_N}\otimes \dots \otimes e_{j_2} }
		\\
		=(-1)^{(N-1)} v^{a_{\mathbf{i}}(k)} \overline{\overline{\Phi_k^*  e_{i_1} \wedge \dots \wedge e_{i_N}}}	
		=(-1)^{(N-1)} v^{a_{\mathbf{i}}(k)} \Phi_k^*  e_{i_1} \wedge \dots \wedge e_{i_N}.
	\end{multline*}
\end{proof}

\subsection{The limit}  \label{subsection: limit}

Let us consider an inductive system of vector spaces
\begin{equation} \label{eq: the injective system}
  \bc_-\mathbb{C}^n[Y^{\pm 1}] \xrightarrow{\varphi_{2,1}^{(\mt)}}  \bc_-\left(\mathbb{C}^n[Y^{\pm 1}] \right)^{\otimes 2} \xrightarrow{\varphi_{3,2}^{(\mt)}}   \cdots \xrightarrow{\varphi_{N,N-1}^{(\mt)}}  \bc_- \left(\mathbb{C}^n[Y^{\pm 1}] \right)^{\otimes N} \xrightarrow{\varphi_{N+1,N}^{(\mt)}} \cdots
\end{equation}
with the maps $\varphi^{(\mt)}_{N+1,N}(w) = w \wedge e_{N-\mt}$. Denote the injective limit by $\siw{\mt}\!\!\left(\mathbb{C}^n[Y^{\pm 1}] \right)$. Also, let us define maps $\varphi^{(\mt)}_{R,N} (w)= w \wedge e_{N-\mt} \wedge e_{N+1-\mt} \wedge \dots \wedge e_{R-1-\mt}$ for \(R>N\). 

By $\varphi_N^{(\mt)} = \varphi^{(\mt)}_{\infty, N} \colon \bc_- \left(\mathbb{C}^n[Y^{\pm 1}] \right)^{\otimes N} \rightarrow \siw{\mt} \left(\mathbb{C}^n[Y^{\pm 1}] \right)$ we denote the canonical map.

\begin{defin} \label{defin: stabilizes}
A sequence of operators $A^{(N)} \colon \bc_- \left(\mathbb{C}^n[Y^{\pm 1}] \right)^{\otimes N} \rightarrow  \bc_- \left(\mathbb{C}^n[Y^{\pm 1}] \right)^{\otimes (N+\delta)}$ \emph{stabilizes} if for any $w \in \bc_- \left(\mathbb{C}^n[Y^{\pm 1}] \right)^{\otimes k} $ there is $M$ such that for any $N>M$ we have 
\begin{equation}
    \varphi_{N+n+\delta, N+\delta}^{(\mt+\delta)} \circ A^{(N)}  \circ \varphi_{N, k}^{(\mt)} (w)= A^{(N+n)} \circ \varphi_{N+n, k}^{(\mt)} (w).
\end{equation}
\end{defin}
If a sequence $A^{(N)}$ stabilizes, then it induces an operator $\hat{A} \colon \siw{\mt}\!\!\left(\mathbb{C}^n[Y^{\pm 1}] \right) \rightarrow \siw{\mt+\delta}\!\!\left(\mathbb{C}^n[Y^{\pm 1}] \right)$. Actually, the operator $\hat{A}$ depends on the residue of $N$ modulo $n$. We will omit this dependence in our notation.
\begin{prop}
    Let $A_1^{(N)}$ and $A_2^{(N)}$ stabilize. Then the composition $A_1^{(N)} A_2^{(N)}$ stabilizes and the induced operator equals to the composition of induced operators $\hat{A}_1 \hat{A}_2$. 
\end{prop}
\begin{prop}{\cite[Sect. 7.6]{LT00}} \label{prop: stable bar}
    Action of bar involution stabilizes.
\end{prop}
 Let \(\lambda= (\lambda_1 \geq \lambda_2 \geq \dots \geq \lambda_r > 0)\) be a partition with \(l(\lambda) = r \leq N\). Consider a vector
\begin{equation}\label{eq:|lambda>}
    | \lambda \rangle_{N,\mt} = e_{-\lambda_1-\mt} \wedge e_{-\lambda_2+1-\mt} \wedge \dots \wedge e_{-\lambda_r+ r-1-\mt } \wedge e_{ r-\mt }\wedge \dots \wedge e_{N-1-\mt}.
\end{equation}
We will abbreviate \(|\lambda\rangle_{N,\mt} = |\lambda\rangle_{N}=|\lambda\rangle\) if the indices are clear from the context.  We will write \(|\lambda\rangle_{\infty,\mt}=\varphi^{(\mt)}_{\infty, N} |\lambda\rangle_{N,\mt}\). 

\begin{lemma} \label{lemma: phi star stabilization}
For $\mt -k + | \lambda| < N$ we have
\begin{equation}\label{eq: phi star stabilization}
    [N]_v^+ \left( \Phi^*_k | \lambda \rangle_{N,\mt} \right)  \wedge  e_{N-\mt} = [N+1]_v^+ \Phi^*_k  | \lambda \rangle_{N+1,\mt}.
\end{equation}
\end{lemma}
\begin{proof}
	Let us introduce notation \(|\lambda \rangle_{N,\mt} = \sum\nolimits_{j_1, \dots, j_{N}} y_{j_1, \dots, j_{N}}  e_{j_1} \otimes \dots \otimes e_{j_{N}}.\)
	Then we have
	\begin{multline}\label{eq:LHS phi star stabilization}
		\text{LHS of \eqref{eq: phi star stabilization}}=[N]_v^+ \left(\sum\nolimits_{j_2,\dots,j_N}  y_{-k, j_2, \dots, j_{N}} e_{j_2} \otimes \dots \otimes e_{j_{N}} \right) \wedge e_{N-\mt}\\=[N]_v^+\sum\nolimits_{j_2,\dots,j_N}  y_{-k, j_2, \dots, j_{N}} e_{j_2} \wedge \dots \wedge e_{j_{N}} \wedge e_{N-\mt}.
	\end{multline}

	To compute the RHS of \eqref{eq: phi star stabilization}, we will use factorization formula
	\begin{equation} \label{eq: antisym factor}
		\bc^{(N+1)}_-=\frac{1}{[N+1]_v^+}\left(\sum\nolimits_{p=1}^{N+1} (-v)^{p+N-1}  T_p\cdots  T_N \right) \bc^{(N)}_-, 
	\end{equation}
  	here for $p=N+1$ we mean $T_p\cdots  T_N= 1$ . Using \eqref{eq: antisym factor} we obtain that 
  	\begin{equation*}
  		| \lambda \rangle_{N,\mt}  \wedge  e_{N-\mt} =\frac{1}{[N+1]_v^+} \left(\sum\nolimits_{p=1}^{N+1} (-v)^{p+N-1}  T_p\cdots  T_N \right) \left(\sum\nolimits_{\mathbf{j}} y_{\mathbf{j}} \,e_{j_1}{\otimes}\cdots{\otimes}e_{j_{N}}{\otimes}e_{N-\mt} \right).
 	\end{equation*}
 	Here and below we use multi-index notation $\mathbf{j}=(j_1, \dots, j_N)$. For the action of \(T_p \cdots T_N\) we will use \eqref{eq: T1}--\eqref{eq: T4}. Informally, \eqref{eq: T1}--\eqref{eq: T4} say that under the action of $T$, the vectors either remains the same, or permute, or \emph{approach} to each other.
 	
 	In the computation below,  \emph{lower terms} stands for linear combination of terms \(e_{l_1}\otimes \dots \otimes e_{l_{N+1}}\) where \(\forall i, \, l_i <N-\mt\). Then we get
	\begin{multline*}
		[N+1]_v^+ \,| \lambda \rangle_{N,\mt}  \wedge  e_{N{-}\mt} = \Big(   \sum\nolimits_{\mathbf{j}} \Big(\sum\nolimits_{p=1}^{N}(-v)^{p+N-1} T_{p}\cdots  T_{N-1} y_{\mathbf{j}} 
		\\
		 \big( v^{\delta_{j_N \equiv N{-}\mt}} e_{j_1} {\otimes} \cdots {\otimes} e_{j_{N-1}} {\otimes}e_{N{-}\mt}{\otimes} e_{j_N}+(v-v^{-1})e_{j_1} {\otimes} \cdots {\otimes} e_{j_{N}} {\otimes}e_{N{-}\mt}
		   \big)  
		\\   
		+(-v)^{2N} y_{\mathbf{j}}\, e_{j_1} {\otimes} \cdots {\otimes} e_{j_{N}} {\otimes}e_{N{-}\mt} \Big)\Big) +\text{lower terms} 
		\\
		=  \Big(   \sum\nolimits_{\mathbf{j}} \Big(\sum\nolimits_{p=1}^{N}(-v)^{p+N-1} v^{\delta_{j_N \equiv N{-}\mt}}T_{p}\cdots T_{N-1} y_{\mathbf{j}} \, e_{j_1} {\otimes} \cdots {\otimes} e_{j_{N-1}} {\otimes}e_{N{-}\mt}{\otimes} e_{j_N}+
		\\ 
		\big((-v)^{2N}+\sum\nolimits_{p=1}^N (v-v^{-1}) (-v)^{p{+}N{-}1}(-v)^{p{-}N}  \big) y_{\mathbf{j}}\, e_{j_1} {\otimes} \cdots {\otimes} e_{j_{N}} {\otimes}e_{N{-}\mt} \Big) \Big) +\text{lower terms} 
		\\
		= \dots = \sum\nolimits_{q=1}^{N+1} \sum\nolimits_{\mathbf{j}} v^{\sum_{r=q}^N \delta_{j_r \equiv N{-}\mt} } \Big( (-v)^{q{+}N{-}1}+(v{-}v^{-1})\sum\nolimits_{p=1}^{q{-}1} (-v)^{p{+}N{-}1} (-v)^{p{-}q{+}1} \Big)   \\
		y_{\mathbf{j}}\, e_{j_1} {\otimes} \cdots {\otimes} e_{j_{q{-}1}} {\otimes} e_{N{-}\mt} {\otimes} e_{j_q} {\otimes} \cdots {\otimes} e_{j_N} 		  
		+\text{lower terms}
		\\ 
		=\sum\nolimits_{q=1}^{N+1} \sum\nolimits_{\mathbf{j}} v^{\sum_{r=q}^N \delta_{j_r \equiv N{-}\mt} } (-v)^{N{-}q{+}1}   y_{\mathbf{j}}\,	
		e_{j_1} {\otimes} \cdots {\otimes} e_{j_{q{-}1}} {\otimes} e_{N{-}\mt} {\otimes} e_{j_q} {\otimes} \cdots {\otimes} e_{j_N} 
		+\text{lower terms}. 
	\end{multline*}
	In the computation we used \eqref{eq: T1}, \eqref{eq: T3} and relation \(T_p \sum\nolimits_{\mathbf{j}} y_\mathbf{j} e_{j_1} \otimes \cdots \otimes e_{j_{q{-}1}} {\otimes} e_{N{-}\mt} {\otimes} e_{j_q} {\otimes} \cdots {\otimes} e_{j_N} = (-v)^{-1} \sum\nolimits_{\mathbf{j}} y_\mathbf{j} e_{j_1} {\otimes} \cdots {\otimes} e_{j_{q{-}1}} {\otimes} e_{N{-}\mt} {\otimes} e_{j_q} {\otimes} \cdots {\otimes} e_{j_N} \) for \(p<q-1\).
	
	In order to compute the RHS of \eqref{eq: phi star stabilization}, we apply \(\bc_-^{(N)} \Phi_k^*\). By Lemma \ref{Pauli exclusion}, the assumption $\mt -k + | \lambda| < N$  implies that the \emph{lower terms} vanish after the action of \(\bc_-^{(N)}\). Hence we get 
	\begin{multline}\label{eq:RHS phi star stabilization}
 		\text{RHS of \eqref{eq: phi star stabilization}}
 		\\=
		\sum\nolimits_{j_2,\dots,j_{N}} \sum\nolimits_{q=2}^{N+1} v^{\sum_{r=q}^N \delta_{j_r \equiv N{-}\mt} } (-v)^{N{-}q{+}1}    y_{-k,j_2,\dots,j_N} e_{j_2} {\wedge} \cdots {\wedge} e_{j_{q{-}1}} {\wedge} e_{N{-}\mt} {\wedge} e_{j_q} {\wedge} \cdots {\wedge} e_{j_{N}} 
		\\
		=\sum\nolimits_{j_2,\dots,j_{N}} \sum\nolimits_{q=2}^{N+1} (-v)^{2N{+}2{-}2q}  y_{-k,j_2,\dots,j_N} e_{j_2} \wedge \dots \wedge e_{j_{N}} \wedge e_{N-\mt}  \\
		=[N]_v^+ \sum\nolimits_{j_2,\dots,j_{N}}  y_{-k,j_2,\dots,j_N} e_{j_2} \wedge \dots \wedge e_{j_{N}} \wedge e_{N-\mt}
 	\end{multline}
	where we used Lemma \ref{lemma: q-wedge relation} to permute \(e_{N-\mt}\) to the right and Lemma \ref{Pauli exclusion} to cancel out additional lower terms. Comparing the formulas \eqref{eq:LHS phi star stabilization} and \eqref{eq:RHS phi star stabilization} we get the result.
\end{proof}

\begin{prop} \label{prop: stable vertex}
    Action of $\Phi_k$ and $\tilde{\Phi}_k^* =[N]_v^+ \Phi_k^*$ stabilize.
\end{prop}
\begin{proof}
Lemma \ref{lemma: q-wedge relation} implies that action of $\Phi_k$ stabilizes for $N-\mt>k$. Lemma \ref{lemma: phi star stabilization} implies that $\tilde{\Phi}_k^*$ stabilizes.
\end{proof}
\begin{Remark}\label{rem:phi* stabilization}
	We used Lemma \ref{lemma: phi star stabilization} in the proof above. But in Section \ref{sec:Semi-inifinite 1} below, we will need a refinement of this result. Let us introduce the following notation
	\begin{align}
		 \tilde{\Phi}^*_k | \lambda \rangle_{N,\mt} &=\sum_{\mu} x_{\mu}e_{-\mu_1+1-\mt}\wedge 	\dots \wedge e_{-\mu_{N-1}+N-\mt-1},
		\\ 	
		  \tilde \Phi^*_k | \lambda \rangle_{N+1,\mt}  &=\sum_{\mu} x'_{\mu}e_{-\mu_1+1-\mt}\wedge 	\dots \wedge e_{-\mu_{N}+N-\mt}.
	\end{align}
	Then for \(l(\mu)\leq N-1\) we have \(x_\mu=x'_\mu\).  
	
	The proof is based on the same computation as the proof of Lemma \ref{lemma: phi star stabilization}. The lower term no longer vanish after the action of \(\bc_-^{(N)}\) since we have dropped the assumption $\mt -k + | \lambda| < N$. Nevertheless, the lower terms do not contribute to $x_{\mu}'$ for \(l(\mu)\leq N-1\).
\end{Remark} 
\begin{prop} \label{prop: stable dual vertex}
    The following operators stabilize
    \begin{align} \label{eq: def psi tilde via psi and star}
       \tilde{\Psi}_k =&(-1)^{N} v^{\left\lfloor \frac{N-1-\mt-k}{n} \right\rfloor-N} \Psi_k,  &    \tilde{\Psi}_k^*=& (-1)^{N-1} [N]_v^- v^{N-1 - \left\lfloor \frac{N-1-\mt+k}{n} \right\rfloor} \Psi_k^*.
    \end{align}
\end{prop}
\begin{proof}
Follows from the previous proposition and Lemma \ref{lemma: involution and vertex}.
\end{proof}

\begin{Remark} \label{remark: normailization of vertex operators}
We have chosen the factors in the definition of $\tilde{\Phi}_k^*$, $\tilde{\Psi}_k$, and $\tilde{\Psi}_k^*$ such that
\begin{align}
    \Phi_{-\mt-1} | \varnothing \rangle_{N,\mt} =& | \varnothing \rangle_{N+1, \mt+1},  &    \tilde{\Phi}_{\mt}^* | \varnothing \rangle_{N,\mt} =& |\varnothing \rangle_{N-1, \mt-1}, \label{eq: normalization phi}\\
    \tilde{\Psi}_{-\mt-1} | \varnothing \rangle_{N,\mt} =& | \varnothing \rangle_{N+1,\mt+1},  &   \tilde{\Psi}_{\mt}^* | \varnothing \rangle_{N,\mt} =& | \varnothing \rangle_{N-1,\mt-1}. \label{eq: normalization psi}
\end{align}
The statement for $\Phi_{-\mt-1}$ is obvious. For $\tilde{\Phi}_{\mt}^*$ one can argue by induction using Lemma \ref{lemma: phi star stabilization}. Formula \eqref{eq: normalization psi} follows from \eqref{eq: normalization phi} by Lemma \ref{lemma: involution and vertex}.
\end{Remark}

\begin{Remark}
    Also, the operators $\Phi_k$, $\tilde{\Phi}_k^*$, $\tilde{\Psi}_k$, and $\tilde{\Psi}_k^*$ satisfy certain intertwining properties \cite{FR92}. We will discuss the properties for $\hphi^*_k$ in Appendix \ref{subsection : Vertex operators}. The properties for other operators are analogous. The intertwining properties and the normalization conditions \eqref{eq: normalization phi}, \eqref{eq: normalization psi} determine the operators uniquely. In particular, the factors in the definition of $\tilde{\Phi}_k^*$, $\tilde{\Psi}_k$, and $\tilde{\Psi}_k^*$ are determined uniquely. Also, note that we do not define $\tilde{\Phi}_k$ since the operators $\Phi_k$ satisfy all the properties.
\end{Remark}

Let us denote the induced operators as follows
\begin{align}
    \hat{\Phi}_k \colon& \siw{\mt}\!\!\left(\mathbb{C}^n[Y^{\pm 1}] \right)  \rightarrow  \siw{\mt+1}\!\!\left(\mathbb{C}^n[Y^{\pm 1}] \right), & \hat{\Phi}^*_k \colon& \siw{\mt}\!\!\left(\mathbb{C}^n[Y^{\pm 1}] \right)  \rightarrow  \siw{\mt-1}\!\!\left(\mathbb{C}^n[Y^{\pm 1}] \right), \label{eq def: vertex the limit phi} \\
    \hat{\Psi}_k \colon& \siw{\mt}\!\!\left(\mathbb{C}^n[Y^{\pm 1}] \right)  \rightarrow  \siw{\mt+1}\!\!\left(\mathbb{C}^n[Y^{\pm 1}] \right), &
    \hat{\Psi}_k^* \colon& \siw{\mt}\!\!\left(\mathbb{C}^n[Y^{\pm 1}] \right)  \rightarrow  \siw{\mt-1}\!\!\left(\mathbb{C}^n[Y^{\pm 1}] \right). \label{eq def: vertex the limit psi} 
\end{align}

\begin{defin} \label{defin: weakly stabilizes}
A sequence of operators $A^{(N)} \colon \bc_- \left(\mathbb{C}^n[Y^{\pm 1}] \right)^{\otimes N} \rightarrow  \bc_- \left(\mathbb{C}^n[Y^{\pm 1}] \right)^{\otimes (N+\delta)}$ \emph{weakly stabilizes} if for any $w \in \bc_- \left(\mathbb{C}^n[Y^{\pm 1}] \right)^{\otimes k} $ there is $M$ such that for any $N>M$ we have 
\begin{equation}
    \varphi_{N+\delta}^{(\mt+\delta)} \circ A^{(N)} \circ \varphi_{N, k}^{(\mt)} (w)= \varphi_{N+n+\delta}^{(\mt+\delta)} \circ A^{(N+n)} \circ \varphi_{N+n, k}^{(\mt)} (w).
\end{equation}
\end{defin}

\begin{prop}[\cite{KMS95}] \label{prop: Bk stabilization}
    The operators  $b_k = Y_1^k + \dots + Y_N^k$ stabilize for $k <0$ and weakly stabilize for $k>0$. The induced operators $B_k$ satisfy a deformed Heisenberg algebra relation
    \begin{equation} \label{eq: deformed heisenberg}
        [B_k, B_l] = k [n]_{v^k}^+ \delta_{k+l,0}.
    \end{equation}
\end{prop}

\begin{Example}
    The sequence of operators $b_1= Y_1 + \dots + Y_N$ does not stabilize. It can be seen from the formula
    \begin{equation} \label{eq: b1 example}
        b_1 e_0 \wedge e_1 \wedge \dots \wedge e_{N-1} = \sum_{k =0}^{n-1} (-v)^k e_0 \wedge e_1 \wedge \dots \wedge \hat{e}_{N-1-k} \wedge \dots \wedge e_{N-1} \wedge e_{N-1-k+n}.
        \end{equation}
    Definitely, the RHS of \eqref{eq: b1 example} does not belong to the image of $\varphi_{N, N-n}^{(0)}$. Though the RHS of \eqref{eq: b1 example} belongs to the kernel of $\varphi_N^{(0)}$. Hence $B_1 |\varnothing \rangle_{\infty, 0}= B_1 \varphi_N^{(0)} ( e_0 \wedge \dots \wedge e_{N-1}) =0$. Moreover, note that
     \begin{align} 
        b_{-1} b_1 &|\varnothing \rangle_{N, 0}=   [n]_v^+ |\varnothing \rangle_{N, 0}  +  \ker \varphi_N^{(0)}, \\
        B_{-1}B_1&|\varnothing \rangle_{\infty, 0}= 0.
        \end{align}
        This means that composition of induced operators does not have to be equal to the induced operator of the composition (if the second operator just weakly stabilizes). Also, this illustrates the fact that $[B_1, B_{-1}]=[n]_v^+$, though $[b_1, b_{-1}] =0$.
\end{Example}

\begin{prop}
    Let a sequence $A_1^{(N)}$ weakly stabilizes and $A_2^{(N)}$ stabilizes. Then the composition $A_1^{(N)} A_2^{(N)}$ weakly stabilizes and the induced operator equals the composition of induced operators $\hat{A}_1 \hat{A}_2$. 
\end{prop}

\begin{prop} \label{prop: heis and vertex operators}
    The following relations hold for any $k \in \mathbb{Z}$ and $j \in \mathbb{Z}_{>0}$
    \begin{align}
        &[B_{-j}, \hat{\Phi}_k] = \hat{\Phi}_{k-nj}, & &[B_{-j}, \hat{\Psi}_k] = v^{-j} \hat{\Psi}_{k-nj},  &  &[B_{-j}, \hat{\Phi}^*_k] = -\hphi^*_{k-nj}, & &[B_{-j}, \hpsi^*_k] = -v^{-j} \hpsi^*_{k-nj},  \label{eq: commut B-j and the vertexes} \\
        &[B_j, \hat{\Phi}_k] = \hat{\Phi}_{k+nj}, & &[B_j, \hat{\Psi}_k] = v^{j(2n{-}1)} \hat{\Psi}_{k+nj},  &  &[B_j, \hat{\Phi}^*_k] = -v^{2jn}\hphi^*_{k+nj}, & &[B_j, \hpsi^*_k] = -v^{-j}\hpsi^*_{k+nj}.
    \end{align}
\end{prop}
\begin{proof}

Commutation relations \eqref{eq: commut B-j and the vertexes} follows from \eqref{eq: commutator bj and the vertexes} since $b_{-j}$ stabilizes. Also, one can check that
\begin{equation}
[B_j, \hphi_k] \varphi_N^{(\mt)} (w) = \varphi_N^{(\mt)}( [b_j, \Phi_k] w) =   \varphi_N^{(\mt)}(  \Phi_{k+nj} w) = \hphi_{k+nj} \varphi_N^{(\mt)}( w).
\end{equation}
To prove the relation with $\hat{\Psi}_k$ we use Lemma \ref{lemma: involution and vertex} and \cite[Prop. 7.8]{LT00} that $\overline{B_j}=v^{-2j(n-1)}B_j$ for $j>0$. Let $\mathtt{a}_{\mathbf{\lambda}}(k) = a_{\lambda}(k) + \lfloor (N-1-\mt -k)/n \rfloor - N$ for sufficiently large $N$.\footnote{We use notation $a_{\lambda}(k)=a_{\mathbf{i}}(k)$ for $i_j = - \lambda_j + j-1 -\mt$, see Lemma \ref{lemma: involution and vertex}} Then 
\begin{multline}
	[B_j, \hat{\Psi}_k] | \lambda \rangle_{\infty} = \overline{(-1)^{N} {[v^{-2j(n-1)}B_j, v^{-\mathtt{a}_{\lambda}(k)}\hat{\Phi}_k]} \overline{| \lambda \rangle}_{\infty}}\\ = \overline{ (-1)^{N} v^{-2j(n-1)} \times v^{-j-\mathtt{a}_{\lambda}(k+nj)}\hat{\Phi}_{k+nj}  \overline{| \lambda \rangle}_{\infty} } =v^{j(2n-1)} \hat{\Psi}_{k+nj}  | \lambda \rangle_{\infty},
\end{multline}

The relation $[B_j, \hphi^*_k] = -v^{2jn}\hphi^*_{k+nj}$ is proven in Appendix \ref{appendix: relations} (Theorem \ref{thm: Bj Phik star}).  Finally,
\begin{multline}
[B_j, \hpsi^*_k] | \lambda \rangle_{\infty} =  \overline{(-1)^{N-1} [v^{-2j(n-1)}B_j, v^{\mathtt{a}^*_{\lambda}(k)} \hphi_k^*]\overline{| \lambda \rangle}_{\infty}}  \\= \overline{(-1)^{N} v^{-2j(n-1)+2jn} \times v^{-j+\mathtt{a}^*_{\lambda}(k+nj)} \hphi_{k+jn}^* \overline{| \lambda \rangle}_{\infty}} = - v^{-j} \hpsi_{k+jn}^* | \lambda \rangle_{\infty},
\end{multline}
here $\mathtt{a}^*_{\mathbf{\lambda}}(k) = a_{\lambda}(k) + \lfloor (N-1-\mt +k)/n \rfloor - N+1$ for sufficiently large $N$.
\end{proof}
\section{Semi-infinite construction of twisted Fock module I}
\label{sec:Semi-inifinite 1}
The goal of this and the next sections is to provide an explicit construction for the action of $\DI$ on twisted Fock module $\mathcal{F}_u^{\sigma}$ (Theorem \ref{thm: explict construction of twisted Fock}). This is the central result of the whole paper. Our method is semi-infinite construction. Namely, we will use the explicit realization of $\Ch_u^{\sigma} \cong \Ch_{u_0, \dots, u_{n-1}}^{(n,n')}$ (Theorem \ref{thm: finite explicit construction}) to derive an explicit construction of $\mathcal{F}_u^{\sigma}$ as a limit $N \rightarrow \infty$.

In Section \ref{subsection: left and right}, we will study the limit $N \rightarrow \infty$ for the Chevalley generators (after a rescaling) denoted by $\tilde{P}_{1,b}^{(N)}$, $\tilde{P}_{0,b}^{(N)}$, and $\tilde{P}_{-1,b}^{(N)}$. It turns out that the generators $\tilde{P}_{1,b}^{(N)}$ and  $\tilde{P}_{-1,b}^{(N)}$ converge for $|q^{-1} v^2|<1$ and $|q^{-1}v^2|>1$ respectively. Therefore we can not obtain the action of whole $\DI$ by a straightforward limit argument. Though we prove that we do obtain actions of the subalgebras $\DIp$ and $\DIm$.

We get explicit formulas for the limit of Chevalley generators. The formulas allow us to make analytic continuation for general $q$ and $v$ (subsection \ref{subsection: analytic continuation}). Also, we consider the formulas in the case $n=1$, $n'=0$ and show that the obtained operators give (non-twisted) Fock module of $\DI$ (subsection \ref{subsec: example n=1}). We will prove for general $n$ and $n'$ that the obtained operators give twisted Fock module of whole $\DI$ in Section \ref{sec: twisted Fock II}.

\subsection{Finite case} \label{subsec: finite} It follows from the Corollary \ref{corollary: P} that
\begin{align}
    P_{1,b}^{(N)} =& (-1)^{N-1} q [N]_v \bc_- \pi^{-1} Y_1^b \bc_-, & P_{-1,b}^{(N)} =& (-1)^{N-1} [N]_v \bc_- Y_1^b \pi \bc_-.         
\end{align}
Recall the representation $\Ch_{u_0, \dots, u_{n-1}}^{(n,n')}$, see Theorem \ref{thm: finite explicit construction}. By the construction, we have
\begin{subequations}
	\begin{align} 
	P_{0,k}^{(N)} =& b_k,  \quad \quad \quad P_{0,-k}^{(N)} = q^k b_{-k}, 
	\label{eq: P -1 b with uk 1}
	\\ 
	P_{1,b}^{(N)} =& (-1)^{N-1} q  [N]_v \sum_{a=0}^{n-1} \sum_{l \in \mathbb{Z}} u_{a}^{-1} q^{- l} \Psi_{a+nl} \Phi^*_{-a-nl+n'+nb},
	\label{eq: P -1 b with uk 2}
	\\
	P_{-1,b}^{(N)} =& (-1)^{N-1} [N]_v \sum_{a=0}^{n-1} \sum_{l \in \mathbb{Z}} u_a q^{l} \Phi_{a+nl-n'+nb} \Psi^*_{-a-nl}.
	\label{eq: P -1 b with uk 3}
	\end{align} 
\end{subequations}
Note that for each vector $w \in \bc_- \left( \mathbb{C}^n[Y^{\pm1}]\right)^{\otimes N}$ only finitely many terms on the RHS of \eqref{eq: P -1 b with uk 2}--\eqref{eq: P -1 b with uk 3} have non-zero action.

To prepare to take the limit $N \rightarrow \infty$, let us rewrite the formulas above. Define $\tilde{P}_{\pm 1,b}^{(N)} = v^{\pm \frac{N}{n}} P_{\pm 1, b}^{(N)}$. Then
\begin{subequations}
    \begin{align}
     \label{eq: SDAHA via vertex finite case tilde 2}
	\tilde{P}_{1,b}^{(N)} =&  q \sum_{a=0}^{n-1} \sum_{l \in \mathbb{Z}} u_a^{-1} v^{\frac{N}{n}- \left\lfloor \frac{N-1-\mt-a}{n}  \right\rfloor } v^l q^{-l} \tilde{\Psi}_{a+nl} \tilde{\Phi}^*_{-a-nl+n'+nb},
	\\ \label{eq: SDAHA via vertex finite case tilde 3}
	\tilde{P}_{-1,b}^{(N)} =& \sum_{a=0}^{n-1} \sum_{l \in \mathbb{Z}} u_a v^{ \left\lfloor \frac{N-1-\mt-a}{n}  \right\rfloor - \frac{N}{n} } v^{-l} q^{l} \Phi_{a+nl-n'+nb} \tilde{\Psi}^*_{-a-nl}.
	\end{align}
\end{subequations}
\begin{Remark} \label{remark: notation N-1 and m-1}
    Let us write $\tilde{\Psi}_k^{(N, \mt)}$ instead of $\tilde{\Psi}_k$ to emphasis the dependence on $N$ and $\mt$. In the formula \eqref{eq: SDAHA via vertex finite case tilde 2} we should use $\tilde{\Psi}_k^{(N-1, \mt-1)}$. Indeed, the operator $ \tilde{\Phi}^*_{-a-nl+n'+nb}$ maps to $\left( \mathbb{C}^n[Y^{\pm1}] \right)^{\otimes (N-1)}$. Moreover, the inductive system must be taken for $\mt -1$. In the calculation above we use
    \[
    \tilde{\Psi}_{a+nl}^{(N-1, \mt-1)} =(-1)^{N-1} v^{\left\lfloor \frac{N-1-\mt-a}{n} \right\rfloor-N+1} v^{-l}\Psi_{a+nl}.
    \]
\end{Remark}

Recall that $u =  u_0 \cdots u_{n-1} q^{1-n}$, see Theorem \ref{thm:isomorphism with twisted}. Then we have
\begin{equation} \label{eq: product of ua with addandum}
\prod_{a=0}^{n-1} v^{ \left\lfloor \frac{N-1-\mt-a}{n}  \right\rfloor - \frac{N}{n} } u_a = v^{-\mt-n} \times q^{n-1} u.
\end{equation}
We can choose any $u_0, \dots, u_{n-1}$ such that \eqref{eq: product of ua with addandum} holds, see Corollary \ref{corol: depnds on product}. Let us take
\begin{equation}  \label{def Ch n n' u}
    u_a = v^{ \frac{N}{n} - \left\lfloor \frac{N-1-\mt-a}{n}  \right\rfloor } \times u^{\frac{1}{n}} v^{-\frac{\mt+1}{n}} \times v^{-\frac{a}{n}} q^{\frac{a}{n}}.
\end{equation}
As a corollary of the above discussion, we obtain the following proposition.
\begin{prop}
\begin{subequations}
The following formulas determine action of $\SHm$ on $\bc_- \left( \mathbb{C}^n [Y^{\pm1}] \right)^{\otimes N}$
	\begin{align} \label{eq: SDAHA via vertex finite case 1}
	P_{0,k}^{(N)} =& b_k,  \quad \quad \quad P_{0,-k}^{(N)} = q^k b_{-k}, 
	\\ \label{eq: SDAHA via vertex finite case 2}
	\tilde{P}_{1,b}^{(N)} =& u^{-\frac{1}{n}}  v^{\frac{\mt+1}{n}} q  \sum_{k \in \mathbb{Z}}   \qt^{-k/n} \tilde{\Psi}_{k} \tilde{\Phi}^*_{-k+n'+nb},
	\\ \label{eq: SDAHA via vertex finite case 3}
	\tilde{P}_{-1,b}^{(N)} =& u^{\frac{1}{n}} v^{-\frac{\mt+1}{n}} \sum_{k \in \mathbb{Z}}    \qt^{k/n}  \Phi_{k-n'+nb} \tilde{\Psi}^*_{-k},
	\end{align}
\end{subequations}
here $\qt = v^{-1} q$. The obtained representation is isomorphic to $\Ch_u^{\sigma}$.
\end{prop}
Denote the obtained representation by $\Ch_u^{(n, n')}$.
\subsection{The limit for the right and left halves} \label{subsection: left and right}
Below we will construct an action of $\DIp$ on $\siw{\mt} \left( \mathbb{C}^n[Y^{\pm1}]\right)$. Analogous results hold for $\DIm$. 

To simplify our notation, we will consider the case $u^{-\frac{1}{n}}  v^{\frac{1}{n}} q=1$ and $\mt=0$. We will recover the parameter $u$ at the end. Recall $\qt = v^{-1} q$. 

\begin{prop} \label{prop: P1b negative stabel}
    The sequence $\tilde{P}_{1,b}^{(N)}=  v^{\frac{N}{n}}P_{1,b}^{(N)}$ stabilizes for $n'+nb <0$. The induced operator
    \begin{equation} \label{eq: hat P and vertex negative}
         \hat{P}_{1,b} = \sum_{k \in \mathbb{Z}} \qt^{-(k+n'+nb)/n} \hpsi_{k+ n'+nb} \hphi^*_{-k}.
    \end{equation}
    For any vector $\hat{w} \in \siw{0} \left( \mathbb{C}^n[Y^{\pm1}]\right)$, only finitely many terms $\hpsi_{k+n'+nb} \hphi^*_{-k} \hat{w}$ are non-zero.
\end{prop}
To prove the proposition we need certain preparations. Recall the notation \(|\lambda \rangle_{N,0}=|\lambda \rangle_{N} \) introduces in \eqref{eq:|lambda>}.
We will need the following lemma.
\begin{lemma} \label{lemma: finitly many terms P1b}
    $\Psi_{k-\Delta} \Phi^*_{-k} |\lambda \rangle_N = 0$ for $\Delta > 0$ and $k \geq |\lambda| + \Delta$.
\end{lemma}
\begin{proof}
Introduce notation \(|\lambda \rangle_N = \sum_{\mathbf{j}} y_{\mathbf{j}} e_{j_1} \otimes \dots \otimes e_{j_{N}}\).
Then
\begin{equation} \label{eq: after we have applied phi star}
\Psi_{k-\Delta} \Phi_{-k}^* \,  | \lambda \rangle_N = \sum_{j_2,  \dots, j_{N}} y_{k, j_2, \dots, j_{N}} e_{j_2} \wedge \dots \wedge e_{j_{N}} \wedge e_{k - \Delta}.
\end{equation}
Consider the decomposition with respect to the basis \(|\mu\rangle_N\)
\begin{equation} 
    \sum_{j_2,  \dots, j_{N}} y_{k, \dots, j_{N}} e_{j_2} \wedge \dots \wedge e_{j_{N}} \wedge e_{k - \Delta} = \sum_{\mu_1 \geq \dots \geq \mu_N} x_{\mu_1, \dots, \mu_N}e_{-\mu_1}\wedge e_{-\mu_2 +1} \wedge \dots \wedge e_{-\mu_N + N-1}.
\end{equation}
For the coefficients $ x_{\mu_1, \dots, \mu_N}\neq 0$ we have $|\mu| = |\lambda|+ \Delta$. 

There exists a number $t \in \{ k, k+1, \dots, N-1 \}$ such that $t \not\in \{ j_2, \dots, j_{N} \}$ for the terms of the RHS of \eqref{eq: after we have applied phi star}. Then there is $t' \in \{ k, k+1, \dots, N-1 \}$ such that $t' \not\in \{ -\mu_{k+1} +k, -\mu_{k+2} +k+1, \dots, -\mu_{N} +N-1 \}$. Hence there is $\mu_s>0$ for $s \geq k+1$. Therefore $k+1 \leq | \mu | = | \lambda| + \Delta$, which contradicts the assumption of the lemma.
\end{proof}

\begin{proof}[Proof of Proposition \ref{prop: P1b negative stabel}]
Using Lemma \ref{lemma: finitly many terms P1b}, we can specify \eqref{eq: SDAHA via vertex finite case 2} omitting zero terms and obtain
    \begin{equation}
        \tilde{P}_{1,b}^{(N)} | \lambda \rangle_N = \sum_{k = - \lambda_1}^{|\lambda|-1 -n'-nb} \qt^{(-k-n'-nb)/n} \tpsi_{k+n'+nb} \Tphi^*_{-k}  | \lambda \rangle_N.
    \end{equation}
    Propositions \ref{prop: stable vertex} and \ref{prop: stable dual vertex} imply that each term $ \qt^{(-k-n'-nb)/n} \tpsi_{k+n'+nb} \Tphi^*_{-k}$ stabilizes. 
\end{proof}

\paragraph{Convergence} Action of operators $\tilde{P}_{1,b}^{(N)}$  does not weakly stabilize for $n'+nb>0$. Therefore we will need the following notion.

\begin{defin}
Action of operators $A^{(N)} \colon \bc_- \left(\mathbb{C}^n[Y^{\pm 1}] \right)^{\otimes N} \rightarrow  \bc_- \left(\mathbb{C}^n[Y^{\pm 1}] \right)^{\otimes N}$ \emph{converges} if for any $w \in  \bc_- \left(\mathbb{C}^n[Y^{\pm 1}] \right)^{\otimes N} $ the following sequence converges for $R \rightarrow \infty$
\begin{equation} \label{eq: def convergence}
    \varphi_{N+nR}^{(\mt)} \circ A^{(N+nR)} \circ \varphi_{N+nR, N}^{(\mt)}(w).
\end{equation}
\end{defin}

\begin{Remark}
    This is the first place in our article where we use that the base field is $\mathbb{C}$, but not a field of characteristic 0.  Note that $\siw{\mt} \left(\mathbb{C}^n[Y^{\pm 1}] \right)$ is a graded vector space with finite-dimensional graded components.  \emph{Convergence of  \eqref{eq: def convergence}} is understood in sense of sequences in a finite-dimensional vector space over $\mathbb{C}$.
    
    Actually, all the convergences below will follow from the convergence of infinite geometric series. Therefore all matrix elements in the limit will be rational functions.
\end{Remark}
\begin{prop} \label{prop P tilde converges}
    The operators $\tilde{P}_{1,b}^{(N)}= v^{\frac{N}{n}}P_{1,b}^{(N)}$ converge for $|q^{-1} v^2 |<1$. Moreover, the induced operator $ \hat{P}_{1,b}$ equals
    \begin{equation} \label{eq: hat P and vertex}
         \hat{P}_{1,b} = \sum_{k \in \mathbb{Z}} \qt^{(-k-n'-nb)/n} \hpsi_{k+n'+nb} \hphi^*_{-k}.
    \end{equation}
    In particular, the series on the RHS of \eqref{eq: hat P and vertex} converges.
\end{prop}
To prove the proposition we need certain preparations.
\begin{lemma} \label{lemma: geometric series P1b stable}
    $\hpsi_{k+\Delta+n} \hphi^*_{-k-n} |\lambda \rangle_{\infty} = v \hpsi_{k+\Delta} \hphi^*_{-k}  |\lambda \rangle_{\infty}$ for $\Delta \geq 0$ and $k \geq |\lambda|  -\Delta$.
\end{lemma}
\begin{proof}
	Let $\Delta = nl -s$ for $s=1, \dots, n$. Lemma \ref{lemma: finitly many terms P1b} implies that for $k+nl \geq |\lambda|+(nl-\Delta)$ we have \(\hpsi_{k+\Delta} \hphi^*_{-k - nl} |\lambda \rangle_{\infty}=0\), and for $k+nl \geq |\lambda|-nl+(nl-\Delta)$ we have \( \hpsi_{k+\Delta} \hphi^*_{-k - nl}B_l |\lambda \rangle_{\infty}=0\). Hence for $k \geq |\lambda| -\Delta$ we get using the Proposition \ref{prop: heis and vertex operators}
	\begin{equation} \label{eq: two term relation}
    	0 = [B_l,  \hpsi_{k+\Delta} \hphi^*_{-k - nl}]  |\lambda \rangle_{\infty} = v^{l(2n-1)} \left( \hpsi_{k+\Delta+ln} \hphi^*_{-k-ln} - v^{l}\hpsi_{k+\Delta} \hphi^*_{-k} \right) |\lambda \rangle_{\infty}.
	\end{equation}
	We argue by induction on $l$. For \(l=1\) the lemma follows from \eqref{eq: two term relation}. Assuming induction hypothesis for $l-1$, we have for $j+n \geq |\lambda| - \left( \Delta - n \right)$ 
	\begin{multline} \label{eq: three term relation}
    	0 = [B_1,\hpsi_{j+\Delta+n} \hphi^*_{-j-2n}]   |\lambda \rangle_{\infty} - v [B_1, \hpsi_{j+\Delta} \hphi^*_{-j-n}]  |\lambda \rangle_{\infty}\\ = v^{2n-1} \left( \hpsi_{j+\Delta+2n} \hphi^*_{-j-2n} -2v \hpsi_{j+\Delta+n} \hphi^*_{-j-n} + v^2 \hpsi_{j+\Delta} \hphi^*_{-j}\right)   |\lambda \rangle_{\infty}.
	\end{multline}    
	Relation \eqref{eq: three term relation} implies that for $k \geq |\lambda| -\Delta$
	\begin{equation} \label{eq: three term relation with l}
		\left(\hpsi_{k+\Delta+ln} \hphi^*_{-k-ln}  - l v^{l-1} \hpsi_{k+\Delta+n} \hphi^*_{-k-n}+ (l-1) v^{l} \hpsi_{k+\Delta} \hphi^*_{-k} \right)   |\lambda \rangle_{\infty}=0.
	\end{equation}
	Relations \eqref{eq: two term relation} and \eqref{eq: three term relation with l} imply $\big( \hpsi_{k+\Delta+n} \hphi^*_{-k-n}-  v \hpsi_{k+\Delta} \hphi^*_{-k} \big)   |\lambda \rangle_{\infty}=0$, which completes the step of the induction.
\end{proof}

This lemma implies that the series on the RHS of \eqref{eq: hat P and vertex} converges (this boils down to the convergence of a geometric series). The next lemma is a finite analog valid before the limit \(N\rightarrow \infty\).

\begin{lemma} \label{lemma: geometric series P1b}
    $\tpsi_{k+\Delta} \Tphi^*_{-k} |\lambda \rangle_N = v \tpsi_{k-n+\Delta} \Tphi^*_{-k+n} |\lambda \rangle_N$ for $\Delta > 0$, and $N - \Delta  > k \geq |\lambda| -\Delta+n$.
\end{lemma}
    \begin{proof}
	Let
	\(	\tilde{\Phi}_{-k}^* \,  | \lambda \rangle_{N} =\sum_{\mu} \tilde{x}_{\mu}e_{-\mu_1+1}\wedge e_{-\mu_2 +2} \wedge \dots \wedge e_{-\mu_{N-1} + N-1} \). 
	We claim that 
	\begin{equation}
		\Psi_{k+\Delta} \Big(e_{-\mu_1+1}\wedge e_{-\mu_2 +2} \wedge \dots \wedge e_{-\mu_{N-1} + N-1}\Big)=0,\quad \text{ if \(l(\mu) \neq k+\Delta\)}.
	\end{equation}
	Indeed, if $\mu_{k+ \Delta} = 0$, then we get zero from the vanishing of the tale \( e_{k + \Delta} \wedge \dots \wedge e_{N-1} \wedge e_{k+\Delta}\) by Lemma \ref{Pauli exclusion}. 
	
	If $\mu_{k+ \Delta+1} > 0$ then there exists a number $t \in \{ k+\Delta+1, k+\Delta+2, \dots, N-1 \}$ such that $t \not\in \{ -\mu_{k+\Delta + 1} +k + \Delta +1, \dots, -\mu_{N-1} +N-1 \}$. 	Consider the expansion
	\begin{equation}\label{eq: psi k plus dela mu nu}
		\Psi_{k+\Delta} \Big(e_{-\mu_1+1}\wedge \dots \wedge e_{-\mu_{N-1} + N-1}\Big) = \sum_{\nu_1 \geq \dots \geq \nu_N} \bar{x}_{\nu_1, \dots, \nu_N}e_{-\nu_1}\wedge \dots \wedge e_{-\nu_N + N-1}.
	\end{equation}
	Then for each term on the RHS there is $t' \in  \{ k+\Delta+1, k+\Delta+2, \dots, N-1 \}$ such that $t' \not\in \{ -\nu_{k+\Delta + 2} +k + \Delta +1, \dots, -\nu_{N} +N-1 \}$. Hence there is $\nu_p>0$ for $p > k+\Delta+1$. Therefore $k+\Delta+ 1 <  | \nu |= | \lambda| - \Delta$, which contradicts the assumption of the lemma. Hence the expression \eqref{eq: psi k plus dela mu nu} is zero.	
	
	Therefore 
	 \begin{multline*}
		\tpsi_{k+\Delta} \Tphi^*_{-k} |\lambda \rangle_N  = \tpsi_{k+\Delta} \Big(\sum_{\mu,\, l(\mu)=k+\Delta} \tilde{x}_{\mu}e_{-\mu_1}\wedge \dots \wedge e_{-\mu_{k+\Delta}+k+\Delta} \wedge e_{k+\Delta+1}  \wedge \dots \wedge  e_{N-1} \Big)\\
		 =  (-1)^{N-1} v^{\left\lfloor \frac{N-1-k}{n} \right\rfloor-N+1}  \sum_{\mu,\, l(\mu)=k+\Delta} \tilde{x}_{\mu}\, e_{-\mu_1}\wedge \dots \wedge e_{-\mu_{k+\Delta}+k+\Delta} \wedge e_{k+\Delta+1}  \wedge \dots \wedge  e_{N-1} \wedge e_{k+\Delta}\\
		= (-1)^{k+\Delta} \,v^{\lfloor \frac{N-1-k}{n} \rfloor - \lfloor \frac{N-1-k -\Delta}{n} \rfloor -k-\Delta }\!\!\!\!\!\! \sum_{\mu,\, l(\mu)=k+\Delta} \tilde{x}_{\mu}\, e_{-\mu_1}\wedge \dots \wedge e_{-\mu_{k+\Delta}+k+\Delta} \wedge e_{k+\Delta} \wedge e_{k+\Delta+1}  \wedge \dots \wedge  e_{N-1}. 
	\end{multline*}
	It was shown in Remark \ref{rem:phi* stabilization} that the coefficients \(\tilde{x}_{\mu}\) are stable. Hence
	\begin{equation} \label{eq: stabilization psi phi delta}
		\varphi_{N+n,N}^{(0)} \left( \tpsi_{k+\Delta} \Tphi^*_{-k} |\lambda \rangle_N \right) =  \tpsi_{k+\Delta} \Tphi^*_{-k}  |\lambda \rangle_{N+n}.
	\end{equation}
	Thus the lemma follows from Lemma~\ref{lemma: geometric series P1b stable}.
\end{proof}

\begin{proof}[Proof of Proposition \ref{prop P tilde converges}]
	Due to Proposition \ref{prop: P1b negative stabel}, it is enough to consider the case $n'+nb>0$. It follows from \eqref{eq: SDAHA via vertex finite case 2} that
	\begin{equation} \label{eq:tilde P and vertex}
         \tilde{P}_{1,b}^{(N)} = \sum_{k \in \mathbb{Z}} \qt^{(-k-n'-nb)/n} \tpsi_{k+n'+nb} \Tphi^*_{-k}.
    \end{equation}
   	We know that each term on the RHS of \eqref{eq:tilde P and vertex} stabilizes to the corresponding term on the RHS of \eqref{eq: hat P and vertex}. 
	Let us consider the vector
	\begin{equation}
    	w_{k, R} = \varphi_{N+nR}^{(0)} \circ \left( \qt^{(-k-n'-nb)/n} \tpsi_{k+n'+nb} \Tphi^*_{-k} \right) | \lambda \rangle_{N+nR}.
	\end{equation}
	The stabilization mentioned above means that $w_{k,R} $ stabilizes for any fixed $k$ and $R \rightarrow \infty$. Now the proposition follows from the following statements
    \begin{enumerate}[(i)]
        \item $w_{k,R} =0$ for $k < - \lambda_1$, \label{statement 1}
        \item $w_{k, R} = q^{-1} v^{2} w_{k-n,R}$  for $N+nR-n'-nb > k \geq |\lambda| -n'-nb+n$,  \label{statement 3}
        \item $w_{k,R}$ =0 for $k \geq N+nR -n'-nb$.\label{statement 4}
    \end{enumerate} 
    Statement \eqref{statement 1} is obvious. 
    Statement \eqref{statement 3} is equivalent to Lemma \ref{lemma: geometric series P1b}. For the statement \eqref{statement 4} note that terms containing \(e_{k+n'+nb}\) vanish after \(\varphi_{N+nR}^{(0)}\) for such \(k\) . 
\end{proof}

In the rest of this subsection we will assume  $|q^{-1} v^2 |<1$. In order to study relations on \(\hat{P}_{1,b}\), we will need the following proposition.

\begin{prop} \label{prop: convergence of the product}
    The operators $\tilde{P}^{(N)}_{1,b_t}  \cdots \tilde{P}^{(N)}_{1,b_2} \tilde{P}^{(N)}_{1,b_1}$ converge to $\hat{P}_{1,b_t}  \cdots \hat{P}_{1,b_2} \hat{P}_{1,b_1}$.
\end{prop}
\begin{proof}
For partitions $\lambda = (\lambda_1 \geq \dots \geq \lambda_r)$ and $\nu=(\nu_1 \geq \dots \geq \nu_s) $  such that $r+s \leq N$, we denote
\begin{equation*}
   | \mu, \nu \rangle_N = e_{-\mu_1} \wedge e_{-\mu_2 +1}\wedge  \dots \wedge e_{\mu_r + r-1} \wedge e_{r} \wedge \dots \wedge e_{N-s-1} \wedge e_{N-s + \nu_s} \dots \wedge e_{N-2 + \nu_2}\wedge e_{N-1 + \nu_1}.
\end{equation*}
In particular, $| \mu, \varnothing \rangle_N = | \mu \rangle_N$. Let
\begin{equation}
    \tilde{P}_{1, b}^{(N)} | \lambda \rangle_N = \sum_{\mu} x_{\mu}^{(N)} | \mu \rangle_N + \sum_{\tilde{\mu}, \tilde{\nu}}  y_{\tilde{\mu}, \tilde{\nu}}^{(N)} | \tilde{\mu}, \tilde{\nu} \rangle_N.
\end{equation}
In the first sum we have $| \mu| =|\lambda| + n + n' b $. In the second sum we have $ | \tilde{\mu}| - | \tilde{\nu} |=|\lambda| + n' + n b $ and $0< |\tilde{\nu}| \leq n' +nb$. Note that only finitely many diagrams $\mu$, $\tilde{\mu}$ and $\tilde{\nu}$ satisfy this conditions.
\begin{lemma} \label{lemma: y tends to 0}
    The coefficients $y_{\tilde{\mu}, \tilde{\nu}}^{(N)}$ tend to $0$ for $N \rightarrow \infty$.
\end{lemma}
\begin{proof}Note that 
\begin{equation} \label{eq: phi psi result sum mu tilde nu tilde}
    \sum_{\tilde{\mu}, \tilde{\nu}}  y_{\tilde{\mu}, \tilde{\nu}}^{(N)} | \tilde{\mu}, \tilde{\nu} \rangle_N = \sum_{k=N-n'-nb}^{N-1} \qt^{(-k-n'-nb)/n} \tpsi_{k+n'+nb} \Tphi^*_{-k} | \lambda \rangle_N.
\end{equation}
Also, we can see that
\begin{equation} \label{eq: some technical eq tilde psi k expantion}
    \Tphi^*_{-k} | \lambda \rangle_N  = \sum_{\Delta = 0}^{N-1-k}  \mathtt{y}_{\tilde{\mu}, \Delta}^{(k, N)} \, e_{-\tilde{\mu}_1} \wedge e_{-\tilde{\mu}_2+1} \wedge \dots \wedge \hat{e}_{k+\Delta} \wedge \dots \wedge e_{N-1}.
\end{equation}
Hence, all $\tilde{\nu}$ appearing on the RHS of \eqref{eq: phi psi result sum mu tilde nu tilde} are hook diagrams $\tilde{\nu} = (k+n'+nb-N+1, 1^{N-1-k - \Delta})$. Recall Remark \ref{remark: notation N-1 and m-1}. Denote $ \mathtt{c}_{k, N}= (-1)^{N-1} v^{\left\lfloor (N-1-k)/n \right\rfloor-N+1} $. We have
\begin{align} \label{eq: connection of tilde nu with delta}
 y_{\tilde{\mu}, \tilde{\nu}}^{(N)} = \qt^{(-k-n'-nb)/n} \mathtt{c}_{k+n'+nb, N}\, \mathtt{y}_{\tilde{\mu}, \Delta}^{(k, N)} \quad \quad \text{for $\tilde{\nu} = (k+n'+nb-N+1, 1^{N-1-k - \Delta})$.}
\end{align}

To study the coefficients $\mathtt{y}_{\tilde{\mu}, \Delta}^{(k, N)}$ we will use the following trick. Let us act by $\tilde{\Psi}_{k+\Delta}$ on \eqref{eq: some technical eq tilde psi k expantion}. Using Lemma \ref{lemma: q-wedge relation}, we 
obtain
\begin{equation} \label{eq: some technical eq act by psi and phi delta coefficient does not depend}
   \tpsi_{k+\Delta} \Tphi^*_{-k} | \lambda \rangle_N  =  (-1)^{N-1-k-\Delta} \, v^{N-1-k-\Delta-\lfloor \frac{N-1-k -\Delta}{n} \rfloor}  \mathtt{c}_{k+\Delta,N} \mathtt{y}^{(k, N)}_{\tilde{\mu}, \Delta} \, | \tilde{\mu} \rangle_N.
\end{equation}
Denote  $\tilde{\mathtt{y}}^{(k, N)}_{\tilde{\mu}, \Delta} = (-1)^{N-1-k-\Delta} \, v^{N-1-k-\Delta-\lfloor \frac{N-1-k -\Delta}{n} \rfloor} \mathtt{c}_{k+\Delta,N} \mathtt{y}^{(k, N)}_{\tilde{\mu}, \Delta}$. Note that Lemma \ref{lemma: geometric series P1b} and stabilization relation \eqref{eq: stabilization psi phi delta}  imply 
\begin{equation}
\tilde{\mathtt{y}}_{\tilde{\mu}, \Delta}^{(k+n, N+n)} = v  \tilde{\mathtt{y}}_{\tilde{\mu}, \Delta}^{(k,N+n)} =v \tilde{\mathtt{y}}^{(k, N)}_{\tilde{\mu}, \Delta}.
\end{equation}
Then $(-1)^n v^{-n} \mathtt{y}_{\tilde{\mu}, \Delta}^{(k+n, N+n)} =  v \mathtt{y}_{\tilde{\mu}, \Delta}^{(k,N)}$.
Thus,  relation \eqref{eq: connection of tilde nu with delta} implies $y_{\tilde{\mu}, \tilde{\nu}}^{(N+n)} = \qt^{-1}v y_{\tilde{\mu}, \tilde{\nu}}^{(N)}= q^{-1}v^2 y_{\tilde{\mu}, \tilde{\nu}}^{(N)}$. This finishes the proof since we have assumed $|q^{-1}v^{2}|<1$.
\end{proof}

Let
\begin{equation} \label{eq P acts alpha beta}
    \tilde{P}_{1, b}^{(N)} | \alpha , \beta \rangle_N = \sum_{\gamma} \tilde{x}_{\gamma}^{(N)} | \gamma \rangle_N + \sum_{\epsilon, \delta}  \tilde{y}_{\epsilon, \delta}^{(N)} | \epsilon, \delta \rangle_N.
\end{equation}
\begin{lemma}  \label{lemma: converges}
   The coefficients $\tilde{x}_{\gamma}^{(N)}$ and $\tilde{y}_{\epsilon, \delta}^{(N)}$ are bounded.
\end{lemma}
\begin{proof}
In case of empty $\beta$, the lemma follows from Proposition \ref{prop P tilde converges} and Lemma \ref{lemma: y tends to 0}. Below we will deduce the general case from the case of empty $\beta$.

Recall the operators $b_j$, see \eqref{eq: def b_j}. Note that     
    \begin{align}
    | \alpha , \beta \rangle_N =  \sum_{j_1, \dots, j_l, \rho} z_{j_1, \dots, j_l, \rho} b_{j_1} \cdots b_{j_l} | \rho \rangle_N.
    \end{align}
    Moreover, the coefficients $z_{j_1, \dots, j_l, \rho}$ does not depend on $N$. Then lemma follows from the commutation relation $[b_j, \tilde{P}_{1,b}^{(N)}] = (q^j -1) \tilde{P}_{1,b+j}^{(N)}$.
\end{proof}
Let
\begin{equation}
    \tilde{P}^{(N)}_{1,b_t} \cdots  \tilde{P}^{(N)}_{1,b_{2}} \tilde{P}^{(N)}_{1,b_1} | \lambda \rangle_N = \sum_{\zeta} x_{\zeta, t}^{(N)} | \zeta \rangle_N + \sum_{\eta, \theta}  y_{\eta, \theta, t}^{(N)} | \eta, \theta \rangle_N.
\end{equation}
\begin{lemma} \label{lemma: y tends to 0 for t}
The coefficients $y_{\eta, \theta,t}^{(N)}$ tend to $0$ for $N \rightarrow \infty$.
\end{lemma}
\begin{proof}
Follows from Lemmas \ref{lemma: y tends to 0} and \ref{lemma: converges} by induction on $t$.
\end{proof}
Let us prove Proposition \ref{prop: convergence of the product} by induction on $t$. Let 
\begin{align}
      \tilde{P}^{(N)}_{1,b_{t-1}} \cdots \tilde{P}^{(N)}_{1,b_1} | \lambda \rangle_N &= \sum_{\zeta} x_{\zeta, t-1}^{(N)} | \zeta \rangle_N + \sum_{\eta, \theta}  y_{\eta, \theta, t-1}^{(N)} | \eta, \theta \rangle_N,\\
       \hat{P}_{1,b_{t-1}} \cdots \hat{P}_{1,b_1} | \lambda \rangle_{\infty} &= \sum_{\zeta} x_{\zeta, t-1}   | \zeta \rangle_{\infty}.
\end{align}
Note, that $ \varphi_{N}^{(0)} | \eta, \theta \rangle_N=0$. Hence the assumption of the induction says that $x_{\zeta, t-1}^{(N)}$ tends to $x_{\zeta, t-1}$. Then Lemmas \ref{lemma: converges} and \ref{lemma: y tends to 0 for t} imply that 
\begin{equation*}
  \lim_{N \rightarrow \infty} \Big( \varphi_{N}^{(0)} \tilde{P}^{(N)}_{1,b_t} \cdots \tilde{P}^{(N)}_{1,b_1} | \lambda \rangle_N - \hat{P}_{1,b_t} \cdots \hat{P}_{1,b_1} | \lambda \rangle_{\infty} \Big) =  \lim_{N \rightarrow \infty}  \sum_{\eta, \theta}  y_{\eta, \theta, t-1}^{(N)} \varphi_{N}^{(0)}  \tilde{P}^{(N)}_{1,b_t} | \eta, \theta \rangle_N =0
\end{equation*}
\end{proof}
Now let us drop the assumption $u^{-\frac{1}{n}} v^{\frac{1}{n}} q= 1$. Also, recall $\qt = v^{-1} q$.
\begin{thm} \label{thm: positive half}
	For $|q^{-1} v^2 | <1$ the following formulas determine an action of $\DIp$ on the space $\siw{0} \left( \mathbb{C}^n[Y^{\pm1}]\right)$
	\begin{subequations}
		\begin{align}
			&c \mapsto v^{-n}, \quad \quad
			c'\mapsto v^{-n'}, \quad \quad
			P_{0, -j} \mapsto q^j B_{-j}, \quad \quad
			P_{0,j} \mapsto  \frac{q^j-1}{v^{-2j} q^j-1} v^{-jn} B_j, \label{eq: Heisenber representatio}
			 \\
			 &P_{1,b} \mapsto \hat{P}_{1,b}= u^{-\frac{1}{n}} v^{\frac{1}{n}}q \sum_{k \in \mathbb{Z}} \qt^{(-k-n'-nb)/n} \hpsi_{k+n'+nb} \hphi^*_{-k}, \label{eq: positive half E}
		\end{align}
	\end{subequations}
	for $q_1=q$, $q_2 =v^{-2}$.
\end{thm}
\begin{proof}
In this proof we will denote by $P_{0,j}$ and $P_{1,b}$ the images, prescribed by \eqref{eq: Heisenber representatio} and \eqref{eq: positive half E}. The proof is a verification of the relations  \eqref{eq: relation Heisenberg}, \eqref{eq: relation HeisenbergE}, \eqref{bilinearE}, and \eqref{SerreE}.

Relation \eqref{eq: relation Heisenberg} follows from \eqref{eq: deformed heisenberg}
\begin{equation}
    [P_{0,j}, P_{0,-j}] =  \frac{q^j-1}{v^{-2j} q^j-1} v^{-jn}  q^j [B_j, B_{-j}] = j \frac{(1-q^j)(v^{-nj}-v^{nj})}{(1 -q^{-j} v^{2j})(1-v^{-2j})}. 
\end{equation}

Equality in \eqref{eq: positive half E} was proven in Proposition \ref{prop P tilde converges}. For relation \eqref{eq: relation HeisenbergE} we will use the formula with vertex operators and Proposition \ref{prop: heis and vertex operators}. For example, for $j>0$
\begin{multline}
    [P_{0,j}, P_{1,b} ] =  u^{-\frac{1}{n}} v^{\frac{1}{n}}  q \sum_{k \in \mathbb{Z}} \qt^{(-k-n'-nb)/n} \Big( [P_{0,j}, \hpsi_{k+n'+nb}] \hphi^*_{-k}  +  \hpsi_{k+n'+nb} [P_{0,j}, \hphi^*_{-k}] \Big) \\
    = \frac{q^j-1}{v^{-2j} q^j-1} v^{-jn} \left( \qt^{j} v^{j(2n-1)} - v^{2jn} \right) P_{1, b+ j} =  ({q^j-1}) v^{jn} P_{1, b+ j} 
\end{multline}

Relations \eqref{bilinearE} and \eqref{SerreE} hold for $\tilde{P}_{1,b}^{(N)}$ by Corollary \ref{corollary: P}. Proposition \ref{prop: convergence of the product} implies that the relations hold for $P_{1,b} = \hat{P}_{1,b}$.  Here we use the definition of $\hat{P}_{1,b}$ as the limit of $\tilde{P}_{1,b}^{(N)}$.
\end{proof}
\begin{Remark}
    Recall that $c'$ does not appear in the relations of $\DIp$. Hence we do not have to specify $c'$ to formulate Theorem \ref{thm: positive half}. The requirement $c' \mapsto v^{-n'}$ will be essential in Section \ref{sec: twisted Fock II}. For the same reason, we will specify $c'$ below.
\end{Remark}
 
\begin{thm} \label{thm: negative half}
For $|q v^{-2} |<1$ the following formulas determine an action of $\DIm$ on the space $\siw{0} \left( \mathbb{C}^n[Y^{\pm1}]\right)$
\begin{subequations}
	\begin{align}
		&c \mapsto v^{-n}, \quad \quad c' \mapsto v^{-n'}, \quad \quad	P_{0, -j} \mapsto q^j B_{-j}, \quad \quad P_{0,j} \mapsto  \frac{q^j-1}{v^{-2j} q^j-1} v^{-jn} B_j, \\
		&P_{-1,b} \mapsto v^{-nb} \hat{P}_{-1,b}= u^{\frac{1}{n}} v^{-\frac{1}{n}} v^{-nb}  \sum_{k \in \mathbb{Z}} \qt^{k/n} \hphi_{k-n'+nb} \hpsi^*_{-k}.   \label{eq:negative half F}
	\end{align}
\end{subequations}
for $q_1=q$, $q_2 =v^{-2}$.
\end{thm}
\begin{proof}[Sketch of a proof.]
The results of this subsection have a counterpart for $P_{-1,b}$, the proofs are analogous.  Relations \eqref{bilinearF}, \eqref{SerreF} hold for both $\hat{P}_{-1,b}$ and $v^{-nb} \hat{P}_{-1,b}$. Proposition \ref{prop: heis and vertex operators} implies relation \eqref{eq: relation HeisenbergF} for the operators $P_{-1,b} = v^{-nb} \hat{P}_{-1,b}$. 
\end{proof}
 \paragraph{Currents}
 For $\alpha=0,1, \dots, n-1$ let us define the following currents (i.e. operator-valued formal power series)
 \begin{align}
     \hphi_{(\alpha)}(z) =& \sum_{k \in \mathbb{Z}} \hphi_{\alpha+nk} z^{-k}, & \hpsi_{(\alpha)}(z) =& \sum_{k \in \mathbb{Z}} \hpsi_{\alpha+nk} z^{-k},\\
     \hphi_{(\alpha)}^*(z) =& \sum_{k \in \mathbb{Z}} \hphi^*_{-\alpha+nk} z^{-k}, & \hpsi_{(\alpha)}^*(z) =& \sum_{k \in \mathbb{Z}} \hpsi_{-\alpha+nk}^* z^{-k}. 
 \end{align}
 Then \eqref{eq: positive half E} and \eqref{eq:negative half F} can be reformulated as follows
 \begin{align}
     E(z) =&  u^{-\frac{1}{n}} v^{\frac{1}{n}} q \sum_{\alpha - \beta \equiv n'} \qt^{\,-\frac{\alpha}{n}} z^{\frac{\beta-\alpha+ n'}{n}}\hpsi_{(\alpha)}(\qt z) \hphi_{(\beta)}^*(z) \label{eq: current E},\\
     F(z) =& u^{\frac{1}{n}}  v^{-\frac{1}{n}}  \sum_{\alpha - \beta \equiv -n'} \qt^{\frac{\beta}{n}} v^{\beta -\alpha -n'} z^{\frac{\beta-\alpha- n'}{n}}\hphi_{(\alpha)}(v^{n} z) \hpsi_{(\beta)}^*(\qt v^{n} z), \label{eq: current F}
 \end{align}
where $\equiv$ stands for $\equiv \pmod{n}$.
 \subsection{Analytic continuation} \label{subsection: analytic continuation}
Consider $\hat{P}_{1,b} |\lambda\rangle_{\infty}$ for any $\lambda$ and $n' + nb >0$. It follows from Lemma \ref{lemma: geometric series P1b stable} that the series \eqref{eq: positive half E} applied to \(|\lambda\rangle_{\infty}\) gives infinite geometric series, which equals to a rational function. This gives an analytic continuation from the region \(|q^{-1}v^2|<1\) to arbitrary $q$ and $v$.

Another way to say this is just rewrite
\begin{equation} \label{eq: rewritten, analytic continuation}
\sum_{k \in \mathbb{Z}} \qt^{\frac{-k-n'-nb}{n}} \hpsi_{k+n'+nb} \hphi^*_{-k}  = \frac{1}{1-q^{-1} v^{2}} \sum_{k \in \mathbb{Z}} \qt^{\frac{-k-n'-nb}{n}} \left( \hpsi_{k+n'+nb} \hphi^*_{-k} - v \hpsi_{k+n'+nb-n} \hphi^*_{-k+n} \right).
\end{equation}
Lemma \ref{lemma: geometric series P1b stable} implies that for any vector $|\lambda \rangle_\infty$, only finitely many terms of the RHS of \eqref{eq: rewritten, analytic continuation} do not annihilate $|\lambda \rangle_{\infty}$. Hence the sum is well-defined without the assumption $|q^{-1} v^{2} |<1$.
\begin{prop} \label{prop: analytic cont DIp}
The following formulas determine an action of $\DIp$ on $\siw{0} \left( \mathbb{C}^n[Y^{\pm1}]\right)$
\begin{subequations}
	\begin{align}
		&c \mapsto v^{-n}, \quad \quad c' \mapsto v^{-n'}, \quad \quad P_{0, -j} \mapsto q^j B_{-j}, \quad \quad \quad P_{0,j} \mapsto  \frac{q^j-1}{v^{-2j} q^j-1} v^{-jn} B_j,\\
		&P_{1,b} | \lambda \rangle_{\infty} = \frac{ u^{-\frac{1}{n}}v^{\frac{1}{n}} q}{1-q^{-1} v^{2}} \sum_{k \in \mathbb{Z}} \qt^{\frac{-k-n'-nb}{n}} \left( \hpsi_{k+n'+nb} \hphi^*_{-k} - v \hpsi_{k+n'+nb-n} \hphi^*_{-k+n} \right) | \lambda \rangle_{\infty}. \label{eq: analytic continuation E}
	\end{align} 
\end{subequations}
for $q_1=q$, $q_2 =v^{-2}$.
\end{prop}
\begin{proof}
Note that \eqref{eq: analytic continuation E} is an analytic continuation of \eqref{eq: positive half E}. The relations hold after an analytic continuation.
\end{proof}  
The above construction can be reformulated in the language of currents. Namely, the following current is well-defined
\begin{equation} \label{eq: regular product psi phi start 1}
\hpsi_{(\alpha)}(\qt z) \hphi_{(\beta)}^*(z) \stackrel{\text{def}}{=} \left. \frac{1}{1-q^{-1} v^{2}} (1-vz/w)\hpsi_{(\alpha)}(w) \hphi_{(\beta)}^*(z) \right|_{w= \qt z}
\end{equation}
This notation allows us to use \eqref{eq: current E} without the assumption $|q^{-1} v^{2} |<1$.
Analogously, the following current is well-defined
\begin{equation}  \label{eq: regular product psi phi start 2}
\hphi_{(\alpha)}(v^{n} z) \hpsi_{(\beta)}^*(\qt v^{n} z) \stackrel{\text{def}}{=} \left. \frac{1}{1-qv^{-2}} (1-v^{-1} w/z)\hphi_{(\alpha)}(v^{n} z) \hpsi_{(\beta)}^*( v^{n} w) \right|_{w = \qt z}
\end{equation}
We prefer to write formulas via currents. The next proposition is a counterpart of Proposition \ref{prop: analytic cont DIp} for $\DIm$. We omit a version without the currents.
\begin{prop} \label{prop: analytic cont DIm}
The following formulas determine an action of $\DIm$ on $\siw{0} \left( \mathbb{C}^n[Y^{\pm1}]\right)$
\begin{subequations}
	\begin{align}
	&c \mapsto v^{-n}, \quad \quad c' \mapsto v^{-n'}, \quad \quad P_{0, -j} \mapsto q^j B_{-j}, \quad \quad \quad P_{0,j} \mapsto  \frac{q^j-1}{v^{-2j} q^j-1} v^{-jn} B_j,
	\\
	&F(z) \mapsto u^{\frac{1}{n}}v^{-\frac{1}{n}} \sum_{\alpha - \beta \equiv -n' } \qt^{\frac{\beta}{n}} v^{\beta -\alpha -n'} z^{\frac{\beta-\alpha- n'}{n}}\hphi_{(\alpha)}(v^{n} z) \hpsi_{(\beta)}^*(\qt v^{n} z).
	\end{align} 
\end{subequations}
for $q_1=q$, $q_2 =v^{-2}$.
\end{prop}
\subsection{Example $n=1$} \label{subsec: example n=1}
Consider an example $n=1$ and $n'=0$. In this case $[B_k, B_l] = k \delta_{k+l,0}$. The space $\siw{\mt} \left( \mathbb{C}[Y^{\pm1}]\right)$ is Fock space for the Heisenberg algebra. Namely, it has a cyclic vector $| \varnothing \rangle_{\infty,\mt} = e_{-\mt} \wedge e_{-\mt+1} \wedge \cdots$ such that $B_k | \varnothing \rangle_{\infty,\mt} = 0$ for $k>0$. Let us consider an operator $e^{\pm Q} \colon \siw{\mt} \left( \mathbb{C}[Y^{\pm1}]\right) \rightarrow \siw{\mt \pm 1} \left( \mathbb{C}[Y^{\pm 1}]\right)$ determined by $e^{\pm Q} | \varnothing \rangle_{\infty,\mt} = | \varnothing \rangle_{\infty,\mt \pm 1}  $ and $[B_k, e^{\pm Q}]=0$. 
\begin{prop}
The operators $\hphi_j,\hpsi_j \colon \siw{\mt} \left( \mathbb{C}[Y^{\pm1}]\right) \rightarrow \siw{\mt+1} \left( \mathbb{C}[Y^{\pm1}]\right)$ are determined by
\begin{align}
    \hphi(z) =&  \exp \left(\sum_{k=1}^{\infty} \frac{z^k}{k} B_{-k}\right) \exp \left(-\sum_{k=1}^{\infty} \frac{z^{-k}}{k} B_{k}\right) e^Q z^{\mt+1},\\
    \hpsi(z) =& \exp \left(\sum_{k=1}^{\infty} \frac{v^k z^k}{k} B_{-k}\right) \exp \left(-\sum_{k=1}^{\infty} \frac{v^{-k}z^{-k}}{k} B_{k}\right) e^Q z^{\mt+1}.
\end{align}
The operators $\hphi^*_j,\hpsi^*_j \colon \siw{\mt} \left( \mathbb{C}[Y^{\pm1}]\right) \rightarrow \siw{\mt-1} \left( \mathbb{C}[Y^{\pm1}]\right)$ are determined by
\begin{align}
    \hphi^*(z) =& \exp \left(-\sum_{k=1}^{\infty} \frac{v^{2k} z^k}{k} B_{-k}\right) \exp \left(\sum_{k=1}^{\infty} \frac{z^{-k}}{k} B_{k}\right) e^{-Q} z^{-\mt},\\
     \hpsi^*(z) =& \exp \left(-\sum_{k=1}^{\infty} \frac{v^{-k}z^k}{k} B_{-k}\right) \exp \left(\sum_{k=1}^{\infty} \frac{v^{-k} z^{-k}}{k} B_{k}\right) e^{-Q} z^{-\mt}.
\end{align}
\end{prop}
\begin{proof} Follows from Proposition \ref{prop: heis and vertex operators} and Remark \ref{remark: normailization of vertex operators}.
\end{proof}
Then
\begin{align}
    \hpsi(w) \hphi^*(z)  =& \frac{w^{\mt} z^{-\mt}}{1-v z/w}\exp \left(\sum_{k=1}^{\infty} \frac{v^k w^k-v^{2k} z^k}{k} B_{-k}\right) \exp \left(\sum_{k=1}^{\infty} \frac{z^{-k}- v^{-k} w^{-k}}{k} B_{k}\right), \\
    \hphi(z) \hpsi^*(w) =& \frac{z^{\mt} w^{-\mt}}{1-v^{-1} w/z}\exp \left(\sum_{k=1}^{\infty} \frac{z^k-v^{-k} w^k}{k} B_{-k}\right) \exp \left(\sum_{k=1}^{\infty} \frac{v^{-k} w^{-k} -z^{-k}}{k} B_{k}\right).
\end{align}
Substituting this to \eqref{eq: current E} and \eqref{eq: current F}, we obtain
\begin{align}
E(z) =& \frac{ u^{-1} v^{1-\mt} q^{1+\mt}}{1-q^{-1}v^{2}}\exp \left(  \sum_{k=1}^{\infty} \frac{q^k-v^{2k}}{k} B_{-k}z^k \right)  \exp \left( \sum_{k=1}^{\infty} \frac{1-q^{-k}}{k} B_k z^{-k} \right),  \label{eq: example E} \\
F(z) =& \frac{u  v^{\mt -1} q^{-\mt}}{1- q v^{-2}}\exp \left( \sum_{k=1}^{\infty} \frac{v^{k}-q^k v^{-k}}{k} B_{-k}z^k \right)  \exp \left( \sum_{k=1}^{\infty} \frac{v^{-k}(q^{-k}-1)}{k} B_k z^{-k} \right). \label{eq: example F}
\end{align}
The following proposition is \cite[Prop. A.6]{FHHSY}. 
 \begin{prop}
 Formulas \eqref{eq: example E}, \eqref{eq: example F}, and
 \begin{align}
 c &\mapsto v^{-1}, &  c'&\mapsto 1, &
 P_{0, -j} &\mapsto q^j B_{-j}, & P_{0,j} \mapsto&  \frac{q^j-1}{v^{-2j} q^j-1} v^{-jn} B_j
 \end{align}
 determine an action of $\DI$ for $q_1=q$, $q_2 =v^{-2}$.
 \end{prop}
 
  For $\mt=0$ we will denote the representation by $\mathcal{F}_u$ and call it \emph{Fock module}.

\begin{Remark} 
 Note that Propositions \ref{prop: analytic cont DIp} and \ref{prop: analytic cont DIm}
 guarantee only existence of actions of $\DIp$ and $\DIm$ separately. Remarkably, the actions of $\DI^{\pm}$ are restrictions of the action of whole $\DI$. Below we will prove the same result for general $n$.
 
 The obtained representation is celebrated \emph{Fock module} of $\DI$.  It was constructed in \cite{FHHSY} via formulas \eqref{eq: example E}, \eqref{eq: example F} (up to a different notation). Though, to the best of our knowledge, the interpretation via the operators $\hphi(z)$, $\hpsi(z)$, $\hphi^*(z)$, and $\hpsi^*(z)$ is new.
 \end{Remark}
\section{Semi-infinite construction of twisted Fock module II} \label{sec: twisted Fock II} 
In the previous section, we have obtained actions of $\DIp$ and $\DIm$ on $\siw{0} \left( \mathbb{C}^n[Y^{\pm1}]\right)$. In this section we prove that these actions (after a simple rescaling) give action of the whole algebra $\DI$. We do not check directly the relations \eqref{eq: commutator E and F 1}, \eqref{eq: relation central extention} due to technical difficulties.

We prove that the defined below operators $\tilde{P}_{a,b}^{(N)}$ stabilize for $a n' + bn  \leq 0$ and general $q$ and $v$. Therefore we obtain a representation of the corresponding subalgebra $\DI^{\swarrow}\!\subset \DI$ (subsection \ref{subsec: bottom half}). We extend the action to the whole $\DI$, the obtained representation is isomorphic to twisted Fock module $\mathcal{F}_u^{\sigma}$ by construction. Then we compare the obtained action $\DI \curvearrowright \siw{0} \left( \mathbb{C}^n[Y^{\pm1}]\right)$  with the actions of the subalgebras $\DI^{\pm}$. 

Recall that the actions of $\DI^{\pm}$ are determined by explicit formulas for Chevalley generators (Proposition \ref{prop: analytic cont DIp} and \ref{prop: analytic cont DIm}). Hence we get explicit formulas for the action of Chevalley generators of $\DI$ on $\siw{0} \left( \mathbb{C}^n[Y^{\pm1}]\right)$, the obtained representation is isomorphic to twisted Fock module $\mathcal{F}_u^{\sigma}$ (Theorem \ref{thm: explict construction of twisted Fock}). This is the central result of the whole paper.

\subsection{The limit for the bottom half} \label{subsec: bottom half}
\paragraph{Existence of the limit} Below we will use the results of Sections \ref{ssec: Triangularity of Macdonald operators} and \ref{ssec: monomial basis}. Recall that $m,m'$ are integers such that $nm'-n'm=1$ and $0\leq m <n$, $0 \leq m' < n'$. Also, recall the automorphism \(\sigma \in \widetilde{SL}(2,\mathbb{Z})\) defined in Section \ref{ssec: Triangularity of Macdonald operators}.

\begin{lemma}\label{lemma:limit A B} 
	a) The sequences of operators $v^{\frac{km N}n} P_{km,-km'}^{(N)}$ stabilize for $k \in \mathbb{Z}_{>0}$.	
	
	b) The sequences of operators $v^{kN} P_{kn,-kn'}^{(N)} - u^{-k}v^{-k}q^{2k} \sum_{i=1}^{N}  (v^2q^{-1})^{ik}$  stabilize for $k \in \mathbb{Z}_{>0}$.	
	
	c) The sequences of operators $v^{-kN} P_{-kn,kn'}- u^{k}v^{k}q^{-k} \sum_{i=1}^{N}  (v^{-2}q)^{ik}$ stabilize for $k \in \mathbb{Z}_{>0}$.	
\end{lemma}
\begin{proof} a)
	The formula \eqref{eq:P{0,k}N} and $SL(2,\mathbb{Z})$-transformation properties of the $P_{a,b}^{(N)}$ generators imply
	\begin{equation}
	P_{km,-km'}^{(N)}= \sigma\left(P_{0,-k}^{(N)} \right) =  q^k\bc_- \sum_{i=1}^N \sigma\left(Y_i \right)^{-k} \bc_- =q^k\bc_- \sum_{i=1}^N \AB_i^{-k} \bc_-.
	\end{equation}
	Note that $\sum_{i=1}^N \AB_i^{-k}$ commutes with finite Hecke algebra and, in particular, with $\bc_-$. Hence
	\begin{equation}\label{eq:P{km,km'} action}
		P_{km,-km'}^{(N)} |\lambda\rangle_N =q^k\sum_{i=1}^N \bc_- \AB_i^{-k}  \big( e_{-\lambda_1} \otimes e_{-\lambda_2+1} \otimes \dots \otimes e_{N-1}\big).
	\end{equation}
	 We decompose the proof into two steps. 
	
		\emph{Step 1.}	First we show that $i$-th term in \eqref{eq:P{km,km'} action} vanishes for $i>|\lambda|+k$. Denote
		\begin{equation} \label{eq: application of Ak to lambda}
		    \AB_i^{-k}  \big( e_{-\lambda_1} \otimes e_{-\lambda_2+1} \otimes \dots \otimes e_{N-1}\big) = \sum_{j_1, \dots, j_N} c_{j_1, \dots j_N}  e_{j_1} \otimes \dots \otimes e_{j_N}.
		\end{equation}
		For a sequence $\{j_1, \dots , j_N \}$, we will say that a number $r$ is a \emph{hole} if $0 \leq r \leq N-1$ and $r \not\in \{j_1, \dots , j_N \}$.
		Let us prove that for each summand in \eqref{eq: application of Ak to lambda} there is a hole $r \geq i-1$.
		
		We will use the formula \eqref{eq:Ai in G}. Note that the operators $\kappa_i^{-\ct}G_{i,j}^{\pm 1}\kappa_i^{\ct}$, $\kappa_i^{-\ct}G_{j,i}^{\pm 1}\kappa_i^{\ct}$, and $\kappa_i$ preserve the existence of a hole with position $\geq i-1$. If there is no such holes, operator $\kappa_i$ must create one. Hence the operator $\AB_i^{-k}$ must create a hole $r \geq i-1$.
		
		Let us also denote
		\begin{equation}
		   \bc_- \AB_i^{-k}  \big( e_{-\lambda_1} \otimes e_{-\lambda_2+1} \otimes \dots \otimes e_{N-1}\big) = \sum_{\mu_1 \geq \mu_2 \geq \cdots} \tilde{c}_{\mu}  e_{-\mu_1} \wedge e_{-\mu_2 +1} \wedge \dots \wedge e_{N-1}.
		\end{equation}		
		
		Note that $|\mu| = | \lambda| + k $. Existence of a hole $r$ implies $\mu_{r+1} >0$. Hence $|\mu| \geq r+1 \geq i > |\lambda|+k$. Hence the sum runs over the empty set, i.e. $ \bc_- \AB_i^{-k}  \big( e_{-\lambda_1} \otimes e_{-\lambda_2+1} \otimes \dots \otimes e_{N-1}\big) = 0$.

	\emph{Step 2.}	 It remains to study the terms in \eqref{eq:P{km,km'} action} for small $i$. If we replace $N\mapsto N+n$ and use again formula \eqref{eq:Ai in G}, we get $mk$ additional factors of the form $\kappa_i^{-\ct}G_{i,N+1}\cdots G_{i,N+n}\kappa_i^{\ct}$. Each of these factors acts diagonally with addition of the terms with holes $r \geq N$. As before, such additional terms vanish after the action of $\bc_-$. The diagonal part acts by $v^{-{mk}}$ by the formulas \eqref{eq: G2} and \eqref{eq: G4}.
	
	b) The proof is similar to the previous one. Using formula \eqref{eq:P{k,0}N} and  $SL(2,\mathbb{Z})$-transformation property, we obtain $P_{kn,-kn'}^{(N)}= q^k\bc_- \sum_{i=1}^N \BB_i^{-k} \bc_-$.  Then
	\begin{equation}\label{eq:P{kn,-kn'} action}
		v^{kN}P_{kn,-kn'}^{(N)} |\lambda\rangle_N =q^kv^{kN}\sum_{i=1}^N \bc_- \BB_i^{-k}  \big( e_{-\lambda_1} \otimes e_{-\lambda_2+1} \otimes \dots \otimes e_{N-1}\big).
	\end{equation}
	We have two steps as in the proof above.
		
	\emph{Step 1.} Let $i>|\lambda|$. Hence $\lambda_i=0$. In order to compute $i$-th term in the sum \eqref{eq:P{kn,-kn'} action} we use the formula \eqref{eq: BB via G}. Each triangular operator of the form $\kappa_i^{-\ct}G_{i,j}^{\pm 1}\kappa_i^{\ct}$, $\kappa_i^{-\ct}G_{j,i}^{\pm 1}\kappa_i^{\ct}$ acts on $e_{-\lambda_1} \otimes e_{-\lambda_2+1} \otimes \dots \otimes e_{N-1}$ diagonally with addition of the terms with holes $r \geq i-1$.  The terms with the holes vanish after the action of $\bc_-$. The diagonal contribution was computed in the proof of the Corollary \ref{corol: eigenvectors twisted case}. Hence the $i$-th term in the sum \eqref{eq:P{kn,-kn'} action} is equal to 
	\begin{multline*}
		 q^kv^{kN}(u_0 \cdots u_{n-1} q^{1-n} v^{\ltt_i } q^{-\lambda_i+i-1})^{-k}\big(e_{-\lambda_1} \otimes e_{-\lambda_2+1} \otimes \dots \otimes e_{N-1}\big) 
		 \\
		 =	 u^{-k}v^{-k}q^{2k} (v^2q^{-1})^{ik}
	 \big(e_{-\lambda_1} \otimes e_{-\lambda_2+1} \otimes \dots \otimes e_{N-1}\big),
	\end{multline*}
	where we used $\ltt_i=N+1-2i$ and convention \eqref{def Ch n n' u}.

	\emph{Step 2.} Analogous to the above. 
	
	c)  Using \eqref{eq:P{k,0}N} and $SL(2,\mathbb{Z})$-transformation, we have $P_{-kn,kn'}^{(N)}= \bc_- \sum_{i=1}^N \BB_i^{k} \bc_-$.  The remaining part of the proof is similar to the proof of b).	
\end{proof}

\begin{prop}\label{prop:limit bottom}
The operators $v^{\frac{aN}n}P_{a,b}^{(N)}$ stabilize for $a n' + bn  < 0$.
\end{prop}
\begin{proof}
	It follows from the commutation relations that any $P_{a,b}\in \DIp$ with $a>0$ can be represented as algebraic combination of $P_{1,0}$, its commutators with $P_{0,b}$ for $b \in \mathbb{Z}$, and also $c,c'$. Using $\widetilde{SL}(2,\mathbb{Z})$ symmetry we see that any element  $P_{a,b}\in \DI$ with $a n' + bn  < 0$ is algebraic combination of $P_{m,-m'}$, its commutators with $P_{bn,-bn'}$ for $b \in \mathbb{Z}$, and $c$, $c'$. To finish the proof we use  Lemma \ref{lemma:limit A B}. 
\end{proof}

\begin{Remark}
	a) This proposition gives another proof of Proposition \ref{prop: P1b negative stabel}.
	
	b) The additional series $\sum_{i=1}^{\infty} (v^2 q^{-1})^{k i} $ which appears in Lemma \ref{lemma:limit A B}b) converges if $|v^2 q^{-1}|<1$. This is in agreement with Theorem \ref{thm: positive half}. 
\end{Remark}

For $an' + bn < 0$, let $\tilde{P}_{a,b}^{(N)} = v^{\frac{aN}n}P_{a,b}^{(N)}$. Denote the stable limit of $\tilde{P}_{a,b}^{(N)}$ by  $\hat{P}_{a,b}$. Similarly, consider operators
\begin{align}
    \tilde{P}_{kn,-kn'}^{(N)}  =& v^{kN} P_{kn,-kn'}^{(N)} - u^{-k}v^{-k}q^{2k} \sum_{i=1}^{N}  (v^2q^{-1})^{ik}  + \frac{u^{-k}v^{k}q^{k}}{1- \left( v^2 q^{-1} \right)^k} \label{eq: Macdonald additive constant 1} \\
    \tilde{P}_{-kn,kn'}^{(N)}  =& v^{-kN} P_{-kn,kn'}^{(N)}- u^{k}v^{k}q^{-k} \sum_{i=1}^{N}  (v^{-2}q)^{ik} + \frac{u^{k}v^{-k}}{1- \left(v^{-2} q \right)^k} \label{eq: Macdonald additive constant 2} 
\end{align}

Let $\hat{P}_{kn,-kn'}$ and $\hat{P}_{-kn,kn'}$ denote the stable limits of  $\tilde{P}_{kn,-kn'}^{(N)}$ and $\tilde{P}_{-kn,kn'}^{(N)}$. Let $\DI^{\swarrow}$  be a subalgebra of $\DI$ generated by $P_{a,b}$ for $an' + bn \leq 0$. 
\begin{prop} \label{prop: Pab central charge apperar}
There is an action of $\DI^{\swarrow}$ on $\siw{0} \left( \mathbb{C}^n[Y^{\pm1}]\right)$ given by 
\begin{subequations}
	\begin{align}
	&c \mapsto v^{-n}, \quad \quad c' \mapsto v^{-n'}, & & \label{eq:Pab central charge apperar} 
	\\
	&P_{a,b} \mapsto \hat{P}_{a,b} & &\text{for $a \geq 0$,} \label{eq:Pab central charge apperar plus} 
	\\
	&P_{a,b} \mapsto v^{-n'a-nb}\hat{P}_{a,b} & &\text{for $a <0$.}  \label{eq:Pab central charge apperar minus} 
	\end{align}
\end{subequations}
\end{prop}
\begin{proof}
	From the limit arguments we see that operators $\hat{P}_{a,b}$ satisfy the relations of $\DI^{\swarrow}$ for $c=c'=1$. It remains to show that formulas \eqref{eq:Pab central charge apperar plus}--\eqref{eq:Pab central charge apperar minus} defines isomorphism of \(\DI^{\swarrow}\big|_{c=1,c'=1}\) and \(\DI^{\swarrow}\big|_{c=v^{-n},c'=v^{-n'}}\). 

	Let \(\tilde{\sigma}\) be an element of \(\widetilde{SL}(2, \mathbb{Z})\) such that the corresponding matrix in \(SL(2,\mathbb{Z})\) is \(\begin{pmatrix} -n' & -n \\ m' & m \end{pmatrix}\) and \(n_{\tilde{\sigma}}(0,-1)=0\). Using the formula \eqref{eq:sigma central charge}, we see that action of \(\tilde{\sigma}\) induces an isomorphism
	\begin{equation}
	 \DI^{\swarrow}\big|_{c=v^{-n},c'=v^{-n'}} \xlongrightarrow{\sim} \DIp\big|_{c=1,c'=v}.
	\end{equation}
	In the region $an' + bn \leq 0$ the winding number \(n_{\tilde{\sigma}}(a,b)=0\) for \(a\geq 0\) and \(n_{\tilde{\sigma}}(a,b)=1\) for \(a< 0\), hence the formula \eqref{eq: sigma action on Pab} for the action of \(\tilde{\sigma}\) on \(P_{a,b}\) generators gives
	\begin{equation}
		P_{a,b} \mapsto \begin{cases}
			P_{-n'a-nb,m'a+mb} & \text{for $a \geq 0$} \\
			v^{-n'a-nb} P_{-n'a-nb,m'a+mb} & \text{for $a < 0$}
		\end{cases}
	\end{equation}
	Similarly \(\tilde{\sigma}\) induces an isomorphism between \(\DI^{\swarrow}\big|_{c=1,c'=1}\) and  \(\DIp\big|_{c=1,c'=1}\)
		\begin{equation}
		P_{a,b} \mapsto P_{-n'a-nb,m'a+mb}.
	\end{equation}	
	Since the relations of \(\DIp\) does not include \(c'\), there is an isomorphism of \(\DIp\big|_{c=1,c'=v}\) and \(\DIp\big|_{c=1,c'=1}\) which acts \(P_{a,b}\mapsto P_{a,b}\) for any \(a\geq 0\) and \(b\). 
	
	To sum up the above, we have obtained a chain of isomorphisms
	\begin{equation*}
	    \DI^{\swarrow}\big|_{c=v^{-n},c'=v^{-n'}} \xrightarrow{\sim} \DIp\big|_{c=1,c'=v} \xrightarrow{\sim} \DIp\big|_{c=1,c'=1} \xrightarrow{\sim}  \DI^{\swarrow}\big|_{c=1,c'=1}.
	\end{equation*}
To finish the proof we notice that the composition indeed is given by \eqref{eq:Pab central charge apperar plus}--\eqref{eq:Pab central charge apperar minus}.
\end{proof}
Denote the obtained representation of $\DI^{\swarrow}$ by $\mathcal{F}_u^{\swarrow}$. 

\paragraph{Connection with twisted representation} Below we will give an alternative interpretation of $\mathcal{F}_u^{\swarrow}$ as a version of twisted representation. To do this we need to introduce the following notion.

Let $\DI^{\downarrow}$ be a subalgebra of $\DI$ generated by $P_{a,b}$ for $b\leq 0$. In Section \ref{subsec: example n=1} we have constructed action of $\DI$ on $\siw{0} \left( \mathbb{C}[Y^{\pm1}]\right)$. Denote its restriction to $\DI^{\downarrow}$ by $\rho_1$. On the other hand, let us consider particular case $n=1$ and $n'=0$ of Proposition \ref{prop: Pab central charge apperar}. It gives \emph{a priori} another action of $\DI^{\downarrow}$ on $\siw{0} \left( \mathbb{C}[Y^{\pm1}]\right)$. Denote it by $\rho_2$.

\begin{lemma} \label{lemma: rho 1 and rho 2}
The actions $\rho_1$ and $\rho_2$ coincide.
\end{lemma}
\begin{proof}
Let us prove $\rho_1(P_{a,b})=\rho_2(P_{a,b})$ for $a \geq 0$. The operator $v^{\frac{aN}{n}} P_{a,b}^{(N)}$ converges to $\rho_1(P_{a,b})$ for $|v^2 q^{-1}|<1$ by Proposition \ref{prop: convergence of the product}. On the other hand, operator $\tilde{P}_{a,b}^{(N)}$ stabilizes and the induced operator is $\rho_{2}(P_{a,b})$. Notice that for $|v^2 q^{-1}|<1$, the limits of $v^{\frac{aN}{n}} P_{a,b}^{(N)}$ and $\tilde{P}_{a,b}^{(N)}$ coincide (even for $an'+bn=0$, see \eqref{eq: Macdonald additive constant 1}, \eqref{eq: Macdonald additive constant 2}). Hence $\rho_1(P_{a,b}) = \rho_2(P_{a,b})$ for $|v^2 q^{-1}|<1$. Since matrix coefficients of $\rho_1(P_{a,b})$ and $\rho_2(P_{a,b})$ are analytic (even rational) functions of $q$ and $v$, we have $\rho_1(P_{a,b}) = \rho_2(P_{a,b})$ for any values of $q$ and $v$.

The case $a < 0$ is analogous. Note that prefactors in formulas \eqref{eq:Pab central charge apperar minus} and \eqref{eq:negative half F} coincide for \(n=1, n'=0\).
\end{proof}

Let $\mathcal{F}_{u}^{\downarrow}$ be the restriction of Fock module $\mathcal{F}_{u}$ to $\DI^{\downarrow}$. Recall that $\sigma$ is an element in $\widetilde{SL}(2, \mathbb{Z})$ such that 
$\DI^{\swarrow} = \sigma \left(  \DI^{\downarrow}  \right)$. Recall the formulas \eqref{eq:sigma central charge}, \eqref{eq: sigma action on Pab} for the action $\widetilde{SL}(2, \mathbb{Z}) \curvearrowright \DI$.

\begin{lemma} \label{lemma: tilde SL_2 z action}
For $b \leq 0$, the element $\sigma(P_{a,b})$ acts on  $\mathcal{F}_u^{\swarrow}$ as
\begin{align} 
v^{- n_{ \sigma}(1,0) b } &P_{na-mb, -n'a + m'b} &  &\text{for $a \geq 0$ or $\frac{m}{n} b > a$}\\
v^{-  \left( n_{ \sigma}(1,0) + 1 \right) b } &P_{na-mb, -n'a + m'b}       & &otherwise
\end{align}
\end{lemma}
\begin{proof}
In this proof we assume $b \leq 0$. One can check that the winding number is given by
\begin{equation}
    n_{\sigma}(a,b)= \begin{cases}
     n_{\sigma}(1,0) & \text{for $a \geq 0$ or $\frac{m}{n} b > a$}\\
     n_{\sigma}(1,0)+1 & \text{otherwise}
     \end{cases}
\end{equation}
To finish the proof, we notice that the element $ \left( c'\right)^{na-mb} c^{-n'a + m' b}$ acts on $\mathcal{F}_u^{\swarrow}$ as $v^{-b}$.
\end{proof}
\begin{prop} \label{prop: twisted rep bottom}
There is an isomorphism  of vector spaces $\hat{\psi} \colon \mathcal{F}_{u}^{\downarrow} \rightarrow \mathcal{F}_u^{\swarrow}$ such that $\hat{\psi}$ intertwines the actions. More precisely, $\sigma(X) \hat{\psi} w = \hat{\psi} \left( X w \right)$ for any $w \in  \mathcal{F}_{u}^{\downarrow}$ and $X \in \DI^{\downarrow}$.
\end{prop}
\begin{proof}
Note that $\mathcal{F}_{u}^{\downarrow}$ is a free cyclic module over the algebra generated by $P_{0,-k}$ for $k \in \mathbb{Z}_{>0}$ with cyclic vector $| \varnothing\rangle_{\infty} = e_{0} \wedge e_{1} \wedge \cdots$. Analogously, $\mathcal{F}_u^{\swarrow}$ is a free cyclic module over algebra generated by $\sigma (P_{0,-k})$ with the corresponding\footnote{Here and below we use the same notation for vectors in $\mathcal{F}_{u}^{\downarrow}$ and $\mathcal{F}_u^{\swarrow}$. Hopefully, this will not lead to a confusion.} cyclic vector $|\varnothing \rangle_{\infty}$. Hence there is an isomorphism of vector spaces  $\hat{\psi} \colon \mathcal{F}_{u}^{\downarrow} \rightarrow \mathcal{F}_u^{\swarrow}$ such that $\hat{\psi}(| \varnothing \rangle_{\infty})=| \varnothing \rangle_{\infty}$ and $\hat{\psi}$ satisfies the intertwining property for $X= P_{0,-k}$. 

It remains to prove that the intertwining property holds for any $X \in \DI^{\downarrow}$. To do this we will need the following notation. Let $ \mathcal{F}_{u}^{\downarrow}[\leq k]$ and $\mathcal{F}_u^{\swarrow}[\leq k]$ be subspaces, spanned by $| \lambda \rangle_\infty$ for $|\lambda| \leq k$. Analogously, let $\bc_- \left(\mathbb{C}[Y^{\pm 1}] \right)^{\otimes N}[\leq k]$ and $\bc_- \left(\mathbb{C}^n[Y^{\pm 1}] \right)^{\otimes N}[\leq k]$ be subspaces, spanned by $| \lambda \rangle_N = e_{-\lambda_1} \wedge e_{-\lambda_2+1} \wedge \dots \wedge e_{N-1}$ for $|\lambda| \leq k$. Let us consider the following diagram

\begin{equation} \label{cd: kth level iso}
	\begin{tikzcd}
	 \bc_- \left(\mathbb{C}[Y^{\pm 1}] \right)^{\otimes N}[\leq k] \arrow[d, "\tilde{\psi}"] \arrow[r, "\varphi_{N}^{[1]}"] &
	 \mathcal{F}_{u}^{\downarrow}[\leq k] \arrow[d, "\hat{\psi}"] 
	\\
	 \bc_- \left(\mathbb{C}^n[Y^{\pm 1}] \right)^{\otimes N}[\leq k]
	\arrow[r, "\varphi_{N}^{[n]}"] & \mathcal{F}_u^{\swarrow}[\leq k]
	\end{tikzcd}
\end{equation}
Recall that we used notation $\varphi_N^{(\mt)}$ for the canonical map from the definition inductive limit (see Section \ref{subsection: limit}). In this proof we omit superscript $(\mt)$ since we consider only the case $\mt=0$. We denote the corresponding maps by $\varphi_{N}^{[1]}$  and $\varphi_{N}^{[n]}$ to distinguish the first and the second rows of \eqref{cd: kth level iso}. Note that here we have used interpretation of $\mathcal{F}_{u}^{\downarrow}$ via $\rho_2$, see Lemma \ref{lemma: rho 1 and rho 2}. 

Recall the isomorphism of irreducible representations $\psi \colon \Ch_{u}^{\sigma} \rightarrow \Ch^{(n,n')}_{u}$, see Theorem \ref{thm:isomorphism with twisted}. Hence there exist unique map $\tilde{\psi}$ satisfying
\begin{align}
\tilde{\psi} | \varnothing \rangle_N & = | \varnothing  \rangle_N   &    v^{- \left(n_{ \sigma}(1,0)+ \frac{mN}{n}\right) b}  \sigma \left( P_{a,b}^{(N)} \right) \tilde{\psi}  =&  \tilde{\psi} P_{a,b}^{(N)}
\end{align}
Multiplying by $v^{aN}$, we can reformulate the intertwining property for $\tilde{\psi}$ as follows
\begin{align}  \label{eq psi tilde Pab}
   v^{- n_{ \sigma}(1,0) b}   \tilde{P}_{na-mb, -n'a + m'b}^{(N)}  \tilde{\psi} =   \tilde{\psi} \tilde{P}_{a,b}^{(N)}
\end{align}

Also, note that we abuse notation using the same symbols for maps with and without restriction to corresponding $[\leq k]$ subspace.
\begin{lemma}
Diagram \eqref{cd: kth level iso} is commutative for any (fixed) $k$ and sufficiently large $N$.
\end{lemma}
\begin{proof}
One can verify that the operators $\varphi_{N}^{[n]} \, \tilde{\psi}$ and $\hat{\psi} \varphi_{N}^{[1]}$  satisfy the following properties for sufficiently large $N$
\begin{align}
\varphi_{N}^{[n]} \, \tilde{\psi} | \varnothing \rangle_N =& | \varnothing \rangle_{\infty}     &     \sigma  \left( P_{0, -l} \right)  \varphi_{N}^{[n]} \, \tilde{\psi}  |\lambda \rangle_N =&  \varphi_{N}^{[n]} \, \tilde{\psi} \tilde{P}^{(N)}_{0, -l}  |\lambda \rangle_N  \label{eq: property one} \\
 \hat{\psi} \varphi_{N}^{[1]}| \varnothing \rangle_N =& | \varnothing \rangle_{\infty}     &    \sigma  \left( P_{0, -l} \right)  \hat{\psi} \varphi_{N}^{[1]}  |\lambda \rangle_N =&   \hat{\psi} \varphi_{N}^{[1]} \tilde{P}_{0, -l}^{(N)}  |\lambda \rangle_N \label{eq: property two}
 \end{align}
Actually, \eqref{eq: property one} and \eqref{eq: property two} are the same properties for the operators $\varphi_{N}^{[n]} \, \tilde{\psi}$ and $\hat{\psi} \varphi_{N}^{[1]}$ respectively. 
 To finish the proof we note that the maps $\varphi_{N}^{[n]} \, \tilde{\psi}$ and $\hat{\psi} \varphi_{N}^{[1]}$ are determined by the properties.
\end{proof}

Let us prove that $\sigma(P_{a,b}) \hat{\psi} |\lambda \rangle_{\infty} = \hat{\psi} \left( P_{a,b} |\lambda \rangle_{\infty} \right)$ for $b \leq 0$. Let us take $k$ large enough such that $P_{a,b}|\lambda \rangle_{\infty} \in \mathcal{F}_{u}^{\downarrow}[\leq k]$. Then we take sufficiently large $N$ such that diagram \eqref{cd: kth level iso} is commutative and
\begin{equation} \label{eq: Pab and phi 1}
     P_{a,b} \varphi_{N}^{[1]}   | \lambda \rangle_N  = \begin{cases}
      \varphi_{N}^{[1]} \tilde{P}_{a,b}^{(N)}  | \lambda \rangle_N  & \text{for $a \geq 0$}\\
     v^{-  b} \varphi_{N}^{[1]} \tilde{P}_{a,b}^{(N)}  | \lambda \rangle_N & \text{otherwise}
     \end{cases}
\end{equation}
Also we take $N$ large enough such that
\begin{equation}
      P_{a,b}  \varphi_{N}^{[n]}   | \lambda \rangle_N  = \begin{cases}
      \varphi_{N}^{[n]}  \tilde{P}_{a,b}^{(N)} | \lambda \rangle_N  & \text{for $a \geq 0$}\\
     v^{- n' a - n b} \varphi_{N}^{[n]}  \tilde{P}_{a,b}^{(N)}  | \lambda \rangle_N & \text{otherwise}
     \end{cases}
\end{equation}
Using Lemma \ref{lemma: tilde SL_2 z action}, we obtain
\begin{equation} \label{eq: sigma Pab and phi n}
      \sigma \left( P_{a,b} \right)  \varphi_{N}^{[n]}   | \lambda \rangle_N  = \begin{cases}
     v^{- n_{ \sigma}(1,0) b }  \varphi_{N}^{[n]}  \tilde{P}_{na-mb, -n'a + m'b}^{(N)}  | \lambda \rangle_N  & \text{for $a \geq 0$}\\
     v^{- \left( n_{ \sigma}(1,0) +1 \right) b } \varphi_{N}^{[n]} \tilde{P}_{na-mb, -n'a + m'b}^{(N)}  | \lambda \rangle_N & \text{otherwise}
     \end{cases}
\end{equation}
It follows from the above
\begin{equation*}
    \sigma(P_{a,b}) \hat{\psi} |\lambda \rangle_\infty = \sigma(P_{a,b}) \hat{\psi} \, \varphi_{N}^{[1]} |\lambda \rangle_N = \sigma(P_{a,b})  \varphi_{N}^{[n]} \tilde{\psi} |\lambda \rangle_N =   \dots =\hat{\psi} P_{a,b} \varphi_{N}^{[1]} |\lambda \rangle_N =\hat{\psi} P_{a,b} |\lambda \rangle_\infty, 
\end{equation*}
here the dots stand for the omitted steps involving the cases (straightforward to write down using \eqref{eq: sigma Pab and phi n}, \eqref{eq psi tilde Pab}, and \eqref{eq: Pab and phi 1}).
\end{proof}
\subsection{Action of the whole algebra}
Recall that we write $\equiv$ for $\equiv \pmod{n}$.  Also recall that the currents $\hpsi_{(\alpha)}(v^{-1} q z) \hphi_{(\beta)}^*(z)$ and $\hphi_{(\alpha)}(v^{n} z) \hpsi_{(\beta)}^*(v^{n-1} q z)$ are defined for general $q$ and $v$ by \eqref{eq: regular product psi phi start 1} and  \eqref{eq: regular product psi phi start 2}.
\begin{thm} \label{thm: explict construction of twisted Fock}
The following formulas determine an action of $\DI$ on $\siw{0} \left( \mathbb{C}^n[Y^{\pm1}]\right)$ 
\begin{subequations}
	\begin{align}
		&c \mapsto v^{-n},  \quad \quad  c' \mapsto v^{-n'},\\
		&P_{0, -j} \mapsto q^j B_{-j}, \quad \quad \quad P_{0,j} \mapsto  \frac{q^j-1}{v^{-2j} q^j-1} v^{-jn} B_j, \label{eq: main thm: heisenberg} 
		\\
		&E(z) \mapsto  u^{-\frac{1}{n}}  v^{\frac{1}{n}} q \sum_{\alpha - \beta\equiv n'} q^{-\frac{\alpha}{n}} v^{\frac{\alpha}{n}} z^{\frac{\beta-\alpha+ n'}{n}}\hpsi_{(\alpha)}(v^{-1} q z) \hphi_{(\beta)}^*(z), \label{eq: main th E}  
		\\
		&F(z) \mapsto  u^{\frac{1}{n}}  v^{-\frac{1}{n}} \sum_{\alpha - \beta \equiv  -n'} q^{\frac{\beta}{n}} v^{\frac{n-1}{n}  \beta - \alpha} z^{\frac{\beta-\alpha- n'}{n}}\hphi_{(\alpha)}(v^{n} z) \hpsi_{(\beta)}^*(v^{n-1} q z). \label{eq: main th F}
	\end{align} 
\end{subequations}
for $q_1=q$, $q_2 =v^{-2}$. The obtained representation is isomorphic to twisted Fock module $\mathcal{F}_{u}^{\sigma}$.
\end{thm}
\begin{proof}
There is an action of $\DI$ on $\siw{0} \left( \mathbb{C}^n[Y^{\pm1}]\right)$ determined as follows. Recall that we have defined a map $\hat{\psi}$ from $\mathcal{F}_{u}^{\downarrow}$ to $\siw{0} \left( \mathbb{C}^n[Y^{\pm1}]\right)=\mathcal{F}_u^{\swarrow}$, see Proposition \ref{prop: twisted rep bottom}. But there is an action of the whole $\DI$ on $\mathcal{F}_{u}$, which coincides with $\mathcal{F}_{u}^{\downarrow}$ as a vector space. Hence for any $X \in \DI$ and $w \in \siw{0} \left( \mathbb{C}^n[Y^{\pm1}]\right)$ we can define $\rho_{tw}(X) w :=  \hat{\psi} \circ \rho_{\mathcal{F}_u}\left(\sigma^{-1}(X) \right) \circ \hat{\psi}^{-1} w $. The representation obtained is isomorphic to twisted Fock module $\mathcal{F}_{u}^{\sigma}$. In particular, formula \eqref{eq:sigma central charge} for $\tau = \sigma^{-1}$ implies $\rho_{tw}(c)= v^{-n}$ and $\rho_{tw}(c')=v^{-n'} $. It remains to prove that the action $\rho_{tw}$ is given by \eqref{eq: main thm: heisenberg}--\eqref{eq: main th F}.

Note, that now we have two actions of $\DIp$ on $\siw{0} \left( \mathbb{C}^n[Y^{\pm1}]\right)$. The first one comes from Proposition \ref{prop: analytic cont DIp}, let us denote it by $\rho_+$. The second one comes from the restriction of $\rho_{tw}$ to $\DIp$.
\begin{lemma}
 $\rho_+(P_{a,b}) = \rho_{tw}(P_{a,b})$ for $n'a+nb \leq 0$ and $a \geq 0$.
\end{lemma}
\begin{proof}
Analogous to the proof of Lemma \ref{lemma: rho 1 and rho 2}.
\end{proof}
\begin{lemma} \label{lemma: action coincide plus}
The actions $\rho_+$ and $ \rho_{tw} \big|_{\DIp}$ coincide.
\end{lemma}
\begin{proof}
We will simply write $P_{a,b}=\rho_+(P_{a,b}) = \rho_{tw}(P_{a,b})$ for $n'a+nb \leq 0$ and $b \geq 0$. Any vector of $\siw{0} \left( \mathbb{C}^n[Y^{\pm1}]\right)$ is a linear combination of vectors $P_{a_1,b_1} \cdots P_{a_t,b_t} | \varnothing \rangle_{\infty}$ for $n' a_i + n b_i <0$ and $b_i \geq 0$. The following proposition is \cite[Lemma 5.6]{Burban:2012}.
\begin{prop} \label{prop: PBW for DI plus}
Algebra $\DI^+$ has a basis $P_{k_1, l_1} \cdots P_{k_t, l_t}$ for $\frac{l_1}{k_1} \leq \frac{l_2}{k_2} \leq\dots\leq \frac{l_t}{k_t}$ over $\mathbb{C}[c^{\pm1}, \left(c'\right)^{\pm 1}]$.
\end{prop}
Hence the action of $P_{k,l} \in \DIp$ for $n' k+ nl >0$ is determined by commutation relations in $\DIp$ and the following conditions 
\begin{align}
&P_{\tilde{k},\tilde{l}} | \varnothing \rangle_{\infty} =0 \quad  \text{for $n' \tilde{k}+ n \tilde{l} >0$,} &
&P_{\kt n,- \kt n'} | \varnothing \rangle_{\infty} =  \frac{u^{-\kt }v^{\kt }q^{\kt}}{1- \left( v^2 q^{-1} \right)^{\kt}} | \varnothing \rangle_{\infty}  \quad \text{for $\kt>0$.}
\end{align}
\end{proof}
Let $\rho_-$ denotes the action of $\DIm$, coming from Proposition \ref{prop: analytic cont DIm}. The following lemma is analogous to Lemma \ref{lemma: action coincide plus}.
\begin{lemma}
It holds $ \rho_{tw}(P_{a,b}) =  v^{-an'} \rho_-(P_{a,b})$ for $a \leq 0$.
\end{lemma}
Proposition \ref{prop: analytic cont DIp} implies that the current $\rho_{tw} \left(E(z)\right)=\rho_{+} \left(E(z)\right)$ is given by \eqref{eq: main th E}. Analogously, Proposition \ref{prop: analytic cont DIm} implies that the current $\rho_{tw} \left(F(z)\right)= v^{n'}\rho_{-} \left(F(z)\right)$ is given by \eqref{eq: main th F}. To find the action of $\rho_{tw}(P_{0,j})$, one can use either $\rho_+$ or $\rho_-$.
\end{proof}

\section{Standard basis}\label{sec:standard basis}
As was already mentioned in the introduction, one of the motivations of this paper is Gorsky-Negu\c{t} conjecture on stable envelope bases in the equivariant K-theory of Hilbert schemes of points in \(\mathbb{C}^2\) \cite{GN15}. Recall that there is the standard identification of the K-theory of Hilbert schemes and the space of symmetric functions $\Lambda$. Moreover, stable envelopes have a purely combinatorial characterization as symmetric functions \cite[Sect. 4.1]{Negut:2016}. This gives a reformulation of the conjecture to be given below.

We use notation \(n,n',m,m'\) as before, see Section \ref{subsec: bottom half}. The Fock module $\mathcal{F}_{u}$ can be identified as a vector space with the space of symmetric functions \(\Lambda\) using the correspondence \(p_k \leftrightarrow P_{0,-k}\) for \(k>0\). Hence the twisted Fock module $\mathcal{F}_{u}^\sigma$ can be identified with \(\Lambda\) using the correspondence \(p_k \leftrightarrow P_{ k m  ,-k m'}\) for \(k>0\).

There are several classical bases in the space \(\Lambda\), e.g. Schur basis $\{s_{\mu}\}$. Another important basis consists of \emph{Macdonald symmetric functions} \(P_\lambda\) and their renormalization \(M_\lambda\), see \cite[Sect. 2.4]{Negut:2016} and the references therein. 
\begin{Remark}
    The basis $\{ P_{\lambda} \}$ can be defined as the eigenbasis for the action of the operators $P_{kn,-kn'} \in \DI \curvearrowright \mathcal{F}_u^{\sigma} \cong \Lambda$. Let $P_{\lambda}(q,t) \in \Lambda$ be Macdonald symmetric functions, here we use Macdonald notation $q$,$t$, see \cite{Macdonald:1995}. Then $P_{\lambda}=P_{\lambda}(q, q v^{-2})$.
    
    There is an analogous (but different) construction to obtain Macdonald symmetric functions. One can consider the action $\SHp(q,v) \curvearrowright \bc_+ \mathbb{C}[\AB_1, \dots, \AB_N]$. The corresponding eigenvectors of $\sum_i \BB_i$ are Macdonald symmetric polynomials with the parameters $q$, $t=v^{-2}$ \cite{C92}. In the limit $N \rightarrow \infty$ we obtain Macdonald symmetric functions $P_{\lambda}(q, v^{-2})$. On the contrary, Macdonald symmetric functions with the parameters \(q\), \(t=qv^{-2}\) are limits of the eigenvectors in $\bc_- \mathbb{C}[\AB_1, \dots, \AB_N]$.    
    This difference between the Macdonald parameters in $\bc_+ \mathbb{C}[\AB_1, \dots, \AB_N]$ and $\bc_- \mathbb{C}[\AB_1, \dots, \AB_N]$ was used in the proof Macdonald conjectures, see \cite[Sect. 7]{K97}. 
\end{Remark}
It was shown in Lemma \ref{lemma:limit A B} that the action of \(P_{ k m  ,-k m'} \) stabilizes in the basis \(\{ |\lambda\rangle_N \}\). Hence \(	|\lambda \rangle_\infty\) form a basis of \(\Lambda\). Denote
\begin{align}
    	|\lambda \rangle_\infty =& \sum_{\mu} c_\lambda^\mu (q,v) M_\mu,  &  |\lambda \rangle_\infty =& \sum_{\mu} r_\lambda^\mu (q,v) s_\mu.
\end{align}

\begin{conj} \label{conj: CN comb}
The vectors \(|\lambda \rangle_\infty\) (considered as elements of $\Lambda$) satisfy the following conditions 
\begin{enumerate}[(i)]
	\item  \label{item stable 1}
		\(|\lambda \rangle_\infty\) are integral, i.e. the coefficients \(r_\lambda^\mu (q,v) \in \mathbb{Z}[q^{\pm 1}, v^{\pm 1}]\).
	\item  \label{item stable 2}
		The transition matrix $c_\lambda^\mu (q,v)$ is triangular with respect to the dominance order “\(<\)”.
	\item  \label{item stable 3}
		The coefficients \(c_\lambda^\mu (q,v)\)  satisfy the “window” condition \cite[eq. (4.4)--(4.5)]{Negut:2016} for a certain slope.
\end{enumerate}
Moreover, \(\overline{|\lambda \rangle}_\infty\) satisfy the conditions \eqref{item stable 1}--\eqref{item stable 3} for another slope.
\end{conj}
We do not write the precise form of the window condition since we do not use it here. 

Actually, the conjecture in \cite{GN15} was formulated in a weaker form than Conjecture \ref{conj: CN comb}. Namely, it was stated that there exists an identification of  \(\siw{0}(\mathbb{C}^n[Y^{\pm 1}])\) with the space of symmetric functions \(\Lambda\) such that image of the basis \(\{ |\lambda\rangle_{\infty} \}\) satisfies the properties \eqref{item stable 1}--\eqref{item stable 3} and similarly for \(\{ \overline{|\lambda\rangle}_{\infty} \}\). In this form, the conjecture was proven in \cite{KS:2020} using 3d-mirror symmetry. Above we conjectured additionally that this identification comes from our Theorem \ref{thm: explict construction of twisted Fock}. 

Below we prove the properties \eqref{item stable 1} and \eqref{item stable 2} using analogous properties of the (non-symmetric, finite) basis \(e_\lambda = e_{\lambda_1} \otimes \dots \otimes e_{\lambda_N}\) shown in Section \ref{sec:Representation}. The proof consists of two steps: first we apply \(\bc_-\) and then take the limit \(N\rightarrow \infty\). We believe that the property \eqref{item stable 3} can be also deduced from an analog of the window condition for the basis~\(e_\lambda\).
\begin{thm}\label{thm:stable triangularity}
	The vectors \(|\lambda \rangle_\infty\) (considered as elements of $\Lambda$) satisfy the following conditions 
\begin{enumerate}[(i)]
	\item  \label{item stable th 1}
		\(|\lambda \rangle_\infty\) are integral, i.e. expands in terms of Schur functions with coefficients in $\mathbb{Z}[q^{\pm1}, v^{\pm1}]$.
	\item  \label{item stable th  2}
		The transition matrix from the basis $\{ M_{\mu} \}$  to the basis \(\{ |\lambda \rangle_\infty \}\) is triangular with respect to the dominance order “\(<\)”.
\end{enumerate}
\end{thm}
The theorem precisely states the properties \eqref{item stable 1} and \eqref{item stable 2} from Conjecture \ref{conj: CN comb}. To prove the theorem we need certain preparations. Similarly to the notations of Section \ref{sec:Representation}, let us denote
\begin{equation}
	e_{-\lambda-\rho,N}=e_{-\lambda_1}\otimes e_{-\lambda_2+1} \otimes \dots \otimes e_{N-1}.
\end{equation}
Clearly, we have \(|\lambda\rangle_N= \bc_- e_{-\lambda-\rho,N}\). Similarly, we use notation \(A_{\lambda,N}\) for monomial basis in the space \((\mathbb{C}^n[Y^{\pm 1}])^{\otimes N}\), see formula \eqref{eq:Abasis}. Let $s_{\mu,N}$ and $P_{\mu,N}$ be, respectively, Schur and Macdonald polynomial in $N$ variables. Also, we will need $v$-Vandermonde polynomial
\begin{equation}
    \mathcal{A} = \frac{1}{[N]_v!} \prod_{j>i} \left( v^{-1} \AB_j - v \AB_i \right).
\end{equation}
For a polynomial $f(\AB_1, \dots, \AB_N)$ we denote $f(\AB^{-1})= f(\AB_1^{-1}, \dots, \AB_N^{-1})$.
\begin{lemma} \label{lemma: A is Schur}
a) The vector \(\bc_- A_{\eta,N}\) is skew-symmetric in $\eta_i$. \\
b) If $\eta_i = -\mu_i+i-1 $ for $\mu_1 \geq \mu_2 \geq \dots \geq \mu_N$ then we have \(\bc_- A_{\eta,N} = s_{\mu,N}(\AB^{-1}) \mathcal{A} \).
\end{lemma}
\begin{proof}
a) Introduce usual (not \(v\)-deformed) antisymmetrizer
	\(
		\bc_-^{v=1}=\frac{1}{N!}\sum_{w} (-1)^{l(w)} w,
	\)
	here the summation goes over all the permutations of the generators \(\AB_i\).  It was shown in \cite[Th. 8.2]{K97} that \(\operatorname{Ker}\bc_-^{v=1}= \operatorname{Ker}\bc_-\). Hence the vector \(\bc_- A_{\eta,N} \) is skew-symmetric in  \(\eta_i\).	
	
b) Any polynomial in the image of \(\bc_-\) is a product of a symmetric polynomial and $\mathcal{A}$, see \cite[proof of Th. 7.1]{K97}. Hence \(\bc_- A_{\eta,N} = f_{\mu}(\AB^{-1}) \mathcal{A}\) for some symmetric $f_{\mu}$.   We claim that \(f_{\mu}(\AB^{-1})=s_{\mu}(\AB^{-1})\). This is obvious for \(v=1\). For generic \(v\) it is easy to check that \(f_\varnothing=1\), and that multiplication of \(f_\mu\) by \(\sum_i \AB_i^{-k}\) is given by a formula which does not depend on \(v\) (Murnaghan-Nakayama rule). Evidently, $f_{\mu}$ is determined by this.
\end{proof}

\begin{proof}[Proof of Theorem \ref{thm:stable triangularity}]
	Corollary \ref{color: e via A} states that there is an expansion
	\begin{equation} \label{eq: e lambda rho via A}
		e_{-\lambda -\rho,N}= \sum_{\eta \preceq -\lambda-\rho }\tilde{\alpha}_{\lambda,\eta} A_{\eta,N},
	\end{equation}
	where \(\tilde{\alpha}_{\lambda,\eta}\in  \mathbb{Z}[q^{\pm 1},v^{\pm 1}]\). We apply \(\bc_-\) to both sides. Using Lemma~\ref{lemma: A is Schur}a), the sum can be rewritten as a sum over \(\eta = (\eta_1<\eta_2<\dots<\eta_N)\). Introduce $\mu = - \eta - \rho$ as in Lemma~\ref{lemma: A is Schur}b). The condition \(-\mu - \rho = \eta \prec -\lambda-\rho\) implies \(\mu_i\geq 0\) and \(\mu \leq \lambda\). Due to Lemma \ref{lemma: A is Schur}, we obtain
	    	\begin{equation} \label{eq: transition from standard to schur basis}
		| \lambda \rangle_N = \sum_{\mu \leq \lambda} r_{\lambda,\mu} \, s_{\mu,N}(\AB^{-1}) \mathcal{A}.
	\end{equation}
	Using \(\tilde{\alpha}_{\lambda,\eta}\in  \mathbb{Z}[q^{\pm 1},v^{\pm 1}]\) and the definition of Macdonald polynomials we obtain the following.
	\begin{enumerate}[(i)]
	    \item The coefficients $r_{\lambda,\mu}$ are integral. 
	    \item The transition matrix from the Macdonald basis $\{ P_{\mu,N} \}$ to the basis $\{ | \lambda \rangle_N \}$ is upper triangular. 
	\end{enumerate}
To finish the proof we take the limit $N \rightarrow \infty$.
\end{proof}

\begin{Remark}
	One can show that the costandard basis \(\{ \overline{|\lambda\rangle}_\infty \}\) satisfies the properties \eqref{item stable 1} and \eqref{item stable 2}. Indeed, the transition matrix from \(\{ |\lambda\rangle_\infty \}\) to \(\{ \overline{|\lambda\rangle}_\infty \}\) is  upper triangular with integral coefficients, see formula \eqref{eq:bar on wedge} and Lemma \ref{lemma: q-wedge relation}. One can deduce the properties of the costandard basis from the corresponding properties of the standard basis. This is also consistent with the conjecture of \cite{GN15}.
\end{Remark}

\appendix

\section{Quantum affine algebra and its vertex operators} \label{app:A}
In this paper, we have used the space $\siw{\mt}\!\!\left(\mathbb{C}^n[Y^{\pm 1}] \right)$ and the vertex operators $\Phi(z)$, $\Psi(z)$, $\Phi^*(z)$, and $\Psi^*(z)$. The space $\siw{\mt}\!\!\left(\mathbb{C}^n[Y^{\pm 1}] \right)$ is known as an integrable level-one representation of quantum affine $\qgl$ \cite{KMS95, LT00}. The vertex operators can be defined by intertwining properties.

Integrable level-one representations and the vertex operators have their counterpart for $\vsl$. Below we will study the connection between the $\mathfrak{gl}_n$ and $\mathfrak{sl}_n$ versions. 

We will consider only the vertex operator $\Phi^*(z)$. The  situation for the other operators is analogous. We will need the results concerning $\Phi^*(z)$ for the proof of Proposition \ref{prop: heis and vertex operators}.

\subsection{Action of quantum affine algebra}
Let $\alpha_0$, $\dots$, $\alpha_{n-1}$ be simple positive roots of $\widehat{\mathfrak{sl}}_n$. We will use standard scale product $\left( \alpha_i, \alpha_j \right)$, and Cartan matrix $\langle \alpha_i, \alpha_j \rangle = 2 \left( \alpha_i, \alpha_j \right) / \left( \alpha_j, \alpha_j \right)$. Note that for $\widehat{\mathfrak{sl}}_n$ we have $\langle \alpha_i, \alpha_j \rangle =  \left( \alpha_i, \alpha_j \right)$. Algebra $\vsl$ is generated by $\qE_i$, $\qK_i$ and $\qF_i$ for $i=0, 1, \dots, n-1$. The relations are
\begin{gather}
    \qK_i \qK_j = \qK_j \qK_i, \quad \quad   \qK_i \qE_j \qK_i^{-1} = \fq^{(\alpha_i, \alpha_j)} \qE_j,   \quad \quad    \qK_i \qF_j \qK_i^{-1} = \fq^{-(\alpha_i, \alpha_j)} \qF_j,\\
      [\qE_i , \qF_j] = \delta_{i,j} \frac{\qK_i - \qK_i^{-1}}{\fq - \fq^{-1}},   \\
    \sum_{k=0}^{b_{ij}} (-1)^k \qbin{b_{ij}}{k}  \qE_i^k \qE_j \qE_i^{b_{ij}-k} = 0,    \quad \quad     \sum_{k=0}^{b_{ij}} (-1)^k \qbin{b_{ij}}{k} \qE_i^k \qE_j \qE_i^{b_{ij}-k} = 0.
\end{gather}
here $b_{ij}= 1 - \langle \alpha_i, \alpha_j \rangle $ and $\qbin{b_{ij}}{k} =[b_{ij}]_{\fq}! \big/ \big( [k]_{\fq}![b_{ij}-k]_{\fq}! \big) $.
There is an action of $\vsl$ on $\mathbb{C}^n[Y^{\pm 1}]$ determined as follows
\begin{align}
    \qE_i e_j =& \delta_{i \equiv j} e_{j+1},   &    \qF_i e_j =& \delta_{i \equiv j-1} e_{j-1},  &    \qK_i e_j =& \fq^{\delta_{i \equiv j-1} - \delta_{i \equiv j}} e_{j}.
\end{align}
Using the comultiplication
\begin{align}
\Delta(\qK_i) =& \qK_i \otimes \qK_i, &
\Delta(\qE_i) =& \qE_i \otimes 1 + \qK_i^{-1} \otimes \qE_i,  &
\Delta(\qF_i) =&  \qF_i \otimes \qK_i + 1 \otimes \qF_i,
\end{align}
we can define action of $\vsl$ on $\left( \mathbb{C}^n[Y^{\pm 1}] \right)^{\otimes N}$. Recall the action of affine Hecke algebra $H^Y$ on the space $\left( \mathbb{C}^n[Y^{\pm 1}] \right)^{\otimes N}$ determined in Section \ref{subsection: explicit construction}.
\begin{prop}
The actions of (affine) Hecke algebra and $\vsl$ commute. 
\end{prop}
Hence we have obtained an action $\vsl \curvearrowright \bc_- \left( \mathbb{C}^n[Y^{\pm 1}] \right)^{\otimes N}$.

\paragraph{The limit} Let us use the inductive system \eqref{eq: the injective system}. Let $\qE_i^{(N)}, \qK_i^{(N)}, F_{i}^{(N)} \curvearrowright \left(\mathbb{C}^n[Y^{\pm 1}] \right)^{\otimes N}$ be the operators, coming from the defined above action $\vsl \curvearrowright \left(\mathbb{C}^n[Y^{\pm 1}] \right)^{\otimes N}$. Consider a number $\rt=0, 1,\dots, n-1$ such that $\rt \equiv N-1-\mt \pmod{n}$ .  Let us define the following operators
\begin{align}
    \tilde{\qE}_i^{(N)} =& {\qE}_i^{(N)}, &  \tilde{\qK}_i^{(N)} =& v^{\delta_{i,\rt}} \qK_i^{(N)}, &  \tilde{\qF}_i^{(N)} =& v^{\delta_{i, \rt}} \qF_i^{(N)}.
\end{align}
Below we will need the following versions of Definitions \ref{defin: stabilizes} and \ref{defin: weakly stabilizes}.
\begin{defin}
A sequence of operators $A^{(N)} \colon \bc_- \left(\mathbb{C}^n[Y^{\pm 1}] \right)^{\otimes N} \rightarrow  \bc_- \left(\mathbb{C}^n[Y^{\pm 1}] \right)^{\otimes(N+\delta)}$ \emph{1-stabilizes} if for any $w \in \bc_- \left(\mathbb{C}^n[Y^{\pm 1}] \right)^{\otimes k} $ there is $M$ such that for any $N>M$ we have 
\begin{equation}
    \varphi_{N+\delta+1, N+\delta}^{(\mt+\delta)} \circ A^{(N)} \circ \varphi_{N, k}^{(\mt)} (w)= A^{(N+1)} \circ \varphi_{N+1, k}^{(\mt)} (w).
\end{equation}
\end{defin}

\begin{defin}
A sequence of operators $A^{(N)} \colon \bc_- \left(\mathbb{C}^n[Y^{\pm 1}] \right)^{\otimes N} \rightarrow  \bc_- \left(\mathbb{C}^n[Y^{\pm 1}] \right)^{\otimes (N+\delta)}$ \emph{weakly 1-stabilizes} if for any $w \in \bc_- \left(\mathbb{C}^n[Y^{\pm 1}] \right)^{\otimes k} $ there is $M$ such that for any $N>M$ we have 
\begin{equation}
    \varphi_{N+\delta}^{(\mt+\delta)} \circ A^{(N)} \circ \varphi_{N, k}^{(\mt)} (w)= \varphi_{N+\delta+1}^{(\mt+\delta)} \circ A^{(N+1)} \circ \varphi_{N+1, k}^{(\mt)} (w).
\end{equation}
\end{defin}

\begin{prop}
The operators $\Tphi_k^*$ 1-stabilize.
\end{prop}
\begin{prop}
The operators $\tilde{\qF}_i^{(N)}$ and $\tilde{\qK}_i^{(N)}$ 1-stabilize. The operators $\tilde{\qE}_i^{(N)}$ weakly 1-stabilize. Moreover, for $i \neq \rt$ the operators $\qE_i^{(N)}$ stabilize. More explicitly, for any $w \in \bc_- \left(\mathbb{C}^n[Y^{\pm 1}] \right)^{\otimes k} $ and sufficiently large $N \not\equiv \mt+i+1 \pmod{n}$ we have 
\begin{equation}
    \varphi_{N+n, N}^{(\mt)} \circ \qE_i^{(N)} \circ \varphi_{N, k}^{(\mt)} (w)= \qE_i^{(N+n)} \circ \varphi_{N+n, k}^{(\mt)} (w)
\end{equation}
\end{prop}

Denote the induced operators by $\hat{\qE}_i$, $\hat{\qK}_i$, and $\hat{\qF}_i$.
\begin{prop}[\cite{KMS95, LT00}] \label{prop: action of quatum affine}
The  formulas $\qE_i \mapsto \hat{\qE}_i$,  $\qK_i \mapsto \hat{\qK}_i$,  $\qF_i \mapsto \hat{\qF}_i$ determine an action of $\vsl$ on $\siw{\mt}\!\!\left(\mathbb{C}^n[Y^{\pm 1}] \right)$.
\end{prop}

\paragraph{From $\mathfrak{sl}_n$ to $\mathfrak{gl}_n$}
Let $\qheis$ be the algebra generated by $B_k$ for $k \in \mathbb{Z}$ and central $c^{\pm 1}$ with the relation
\begin{equation} 
        [B_k, B_l] = k  \frac{c^{-2k}-1}{\fq^{2k} -1} \delta_{k+l,0}.
\end{equation}

\begin{Remark}
We abuse notation since $B_k$ was defined as an operator on $\siw{\mt}\!\!\left(\mathbb{C}^n[Y^{\pm 1}] \right)$ and $c$ is an element of $\DI$. 
\end{Remark}

Let $\qgl = \vsl \otimes \qheis$. The algebra $\qgl$ acts on $\siw{\mt}\!\!\left(\mathbb{C}^n[Y^{\pm 1}] \right)$ as follows. The action of $\vsl$ comes from Proposition \ref{prop: action of quatum affine}. The generators $B_k$ act as the operators with the same name defined in Proposition \ref{prop: Bk stabilization}. The central element $c$ acts as multiplication by $v^{-n}$. Denote the obtained representation by $F_{\mt}$.
\begin{prop}
The obtained representation $F_{\mt}$ is irreducible integrable level-one representation of $\qgl$. There are $n$ non-isomorphic classes of such representations $F_0, \dots, F_{n-1}$. The representations $F_{\mt}$ and $F_{\mt+nk}$ are isomorphic for any $k \in \mathbb{Z}$.
\end{prop}
Let $\Fsl_0, \Fsl_1, \dots, \Fsl_{n-1}$ be the irreducible integrable level-one representation of $\vsl$. Let $\Fheis$ be Fock module of $\qheis$ for $c=v^{-n}$.
\begin{prop}
The representation $F_{\mt}$ is isomorphic to $\Fsl_{\mt} \otimes \Fheis$ as representations of $\qgl = \vsl \otimes \qheis$ for $\mt = 0$, $1$, $\dots$, $n-1$.
\end{prop}
\subsection{Vertex operators} \label{subsection : Vertex operators}
Below we will study intertwining properties of $\hphi_k^*$. Its analog defines vertex operators for $\mathfrak{sl}_n$ \cite{FR92}. In this subsection we will recall basic properties of the vertex operators for  $\mathfrak{sl}_n$ and start to study a connection between vertex operators of  $\mathfrak{sl}_n$ and  $\hphi_k^*$. The connection will be made more precise in the next subsection.
\paragraph{Intertwining property} Let us define an operator by
\begin{align} 
    &\hphi^* \colon \siw{\mt}\!\!\left(\mathbb{C}^n[Y^{\pm 1}] \right) \rightarrow \mathbb{C}^n[Y^{\pm 1}] \hat{\otimes} \siw{\mt-1}\!\!\left(\mathbb{C}^n[Y^{\pm 1}] \right)\\
    &   \hphi^* w= \sum_{k \in \mathbb{Z}} e_{-k} \otimes \hphi_k^* w.  \label{def: eq phi hat without an index}
\end{align}
\begin{prop} \label{prop: appendix intertwiner}
$\hphi^*$ is an $\vsl$-intertwiner.
\end{prop}
Evidently, Proposition \ref{prop: appendix intertwiner} is equivalent to the following proposition.
\begin{prop} 
The following relations hold
\begin{align}
    \hphi_k^* \hat{\qK}_i =& v^{ \delta_{i+k \equiv -1}- \delta_{i+k \equiv 0}  } \hat{\qK}_i \hphi_k^* \label{eq: intertwiner K} \\
    \hphi_k^* \hat{\qE}_i =& \delta_{i+k \equiv -1} \hphi_{k+1}^* + v^{\delta_{i+k \equiv 0}  - \delta_{i+k \equiv -1}} \hat{\qE}_i \hphi_k^* \label{eq: intertwiner E}  \\
    \hphi_k^* \hat{\qF}_i =& \delta_{i+k \equiv 0} \hat{\qK}_i\hphi_{k-1}^* +   \hat{\qF}_i \hphi_k^* \label{eq: intertwiner F} 
\end{align}
\end{prop}

\begin{proof}
Analogously to \eqref{def: eq phi hat without an index}, consider operator
\begin{align} 
    \Phi^* \colon \bc_- \left(\mathbb{C}^n[Y^{\pm 1}] \right)^{\otimes N} \rightarrow& \left(\mathbb{C}^n[Y^{\pm 1}] \right) \otimes \bc_- \left(\mathbb{C}^n[Y^{\pm 1}] \right)^{\otimes N-1}\\
    \Phi^* w=& \sum_{k \in \mathbb{Z}} e_{-k} \otimes \Phi_k^* w 
\end{align}
Evidently, the operator is an intertwiner. Equivalently, $\Phi_k^*$ satisfy the counterparts of \eqref{eq: intertwiner K}--\eqref{eq: intertwiner F}
\begin{align}
    \Phi_k^* \qK^{(N)}_i =& v^{\delta_{i+k \equiv -1} - \delta_{i+k \equiv 0} } \qK_i^{(N-1)} \Phi_k^* \label{eq: intertwiner K finite} \\
    \Phi_k^* \qE_i^{(N)} =& \delta_{i+k \equiv -1} \Phi_{k+1}^* + v^{\delta_{i+k \equiv 0} - \delta_{i+k \equiv -1} } \qE_i^{(N-1)} \Phi_k^* \label{eq: intertwiner E finite}  \\
    \Phi_k^* \qF_i^{(N)} =& \delta_{i+k \equiv 0} \qK_i^{(N-1)} \Phi_{k-1}^* +   \qF_i^{(N-1)} \Phi_k^* \label{eq: intertwiner F finite} 
\end{align}
Note that in the relations above we can replace  $\qK_i^{(N)}, \qE_i^{(N)}, \qF_i^{(N)}$ by  $\tilde{\qK}_i^{(N)}, \tilde{\qE}_i^{(N)}, \tilde{\qF}_i^{(N)}$ respectively. Since the operators $\tilde{\qK}_i^{(N)}$, $\tilde{\qF}_i^{(N)}$, and $\tilde{\qE}_j^{(N)}$ stabilize  for $j \neq \rt$, the corresponding relation hold for $\hphi_k^*$. It remains to check
\begin{equation} \label{eq: appendix: tricky intertw}
   \hphi_k^* \hat{\qE}_{\rt} = \delta_{\rt+k \equiv -1} \hphi_{k+1}^* + v^{\delta_{\rt+k \equiv 0} - \delta_{\rt+k \equiv -1}} \hat{\qE}_{\rt} \hphi_k^*  
\end{equation}
Notice that $\qE_{\rt}^{(N+1)}$ stabilizes. Hence \eqref{eq: appendix: tricky intertw} follows from
\begin{equation}
    \Phi_k^* \qE_{\rt}^{(N+1)} = \delta_{\rt+k \equiv -1} \Phi_{k+1}^* + v^{\delta_{\rt+k \equiv 0}  - \delta_{\rt+k \equiv -1}} \qE_{\rt}^{(N)} \Phi_k^* 
\end{equation}
\end{proof}

\paragraph{Vertex operators for $\vsl$}
Denote the highest vector of $\Fsl_{\mt}$ by $| \mt \rangle$. Let us also define $\Fsl_{\mt}$ for any $\mt \in \mathbb{Z}$ by $F_{\mt} = \Fsl_{\mt} \otimes \Fheis$. 

Vertex operators for $\vsl$ are defined as the intertwiners
\begin{align}
&\hphi^{*, sl} \colon \Fsl_{\mt}  \rightarrow \left(\mathbb{C}^n[Y^{\pm 1}] \right)  \hat{\otimes} \Fsl_{\mt-1} \label{eq: sl vetex map}
\end{align}
satisfying the following normalization condition
\begin{align}
&\hphi^{*, sl} | \mt \rangle = e_{-\mt} \otimes | \mt-1 \rangle + \sum_{k< \mt} e_{-k} \otimes w_k  \label{eq: sl vetex map norm}
\end{align}
for certain vectors $w_k \in F_{\mt-1}$.

\begin{prop}[ \cite{FR92}] \label{prop: vertex is unique}
There exists unique operator $\hphi^{*, sl}$ determined by \eqref{eq: sl vetex map}--\eqref{eq: sl vetex map norm}.
\end{prop}
Also, we define operators $\hphi^{*, sl}_k$ and currents $ \hphi_{(\alpha)}^{*, sl}(z)$ by
\begin{align}
\hphi^{*, sl}w = &\sum_{k \in \mathbb{Z}} e_{-k} \otimes \hphi^{*, sl}_k w & \hphi_{(\alpha)}^{*, sl}(z) =& \sum_{k \in \mathbb{Z}} \Phi^{*, sl}_{-\alpha+nk} z^{-k}
\end{align}
Let us consider \emph{principal grading} on $F_{\mt}^{sl}$ given by $\deg |\mt \rangle = - \frac{\mt(\mt+1)}{2}$, $\deg \qE_i = 1$, $\deg \qF_i = -1$. Note that $\deg \hphi^{*, sl}_k = k$. 
\begin{lemma} \label{lemma: vertex operators grading}
For any intertwiner $\phi^{*, sl} = \sum e_{-k} \otimes \phi_k^{*, sl}  \colon \Fsl_{\mt}  \rightarrow \left(\mathbb{C}^n[Y^{\pm 1}] \right)  \hat{\otimes} \Fsl_{\mt-1}$ such that $\deg \phi_k^{*, sl} = k + \Delta$, we have $\phi_k^{*, sl} = \gamma \hat{\Phi}_{k+\Delta}^{*, sl}$ for certain $\gamma \in \mathbb{C}$
\end{lemma}
\begin{proof}
Note that $\deg \phi_{\mt-\Delta}^{*, sl} | \mt \rangle = \deg | \mt -1 \rangle$. Since the subspace of degree $ \deg | \mt -1 \rangle$ is one-dimensional, we have $\phi_{\mt-\Delta}^{*, sl} | \mt \rangle = \gamma | \mt -1 \rangle$ for certain $\gamma \in \mathbb{C}$. Proposition \ref{prop: vertex is unique} implies $\sum e_{-k} \otimes \phi_{k-\Delta}^{*, sl} = \gamma \hphi^{*, sl}$.
\end{proof}
\begin{prop} \label{prop: factorization of vertex operator}
There exist operators $\Xi_d \curvearrowright \Fheis$ such that for $\Xi(z) = \sum_{d \in \mathbb{Z}} \Xi_{d} z^{-d}$ we have
\begin{equation} \label{eq: factorization of vertex operator}
    \hphi_{(\alpha)}^*(z) =   \hphi_{(\alpha)}^{*, \mathfrak{sl}_n}(z) \otimes \Xi (z) 
\end{equation}
\end{prop}
\begin{proof}
We can extend the grading from $F_{\mt}^{sl}$ to $F_{\mt}$ by $\deg B_k = kn$. Note that $\deg \hphi_k = k$. We can present $\hphi_k = \sum_{d, \nu} \phi^{*, sl}_{k, d, \nu}\otimes \Xi_{d, \nu}$ for linear independent operators $\Xi_{d, \nu}$ with $\deg \Xi_{d, \nu} = nd$ (e.g. take $\Xi_{d, \nu}$ to be matrix units for a homogeneous basis of $\Fheis$). Proposition \ref{prop: appendix intertwiner} and Lemma \ref{lemma: vertex operators grading} imply that $\phi^{*, sl}_{k, d, \nu} = \gamma_{d, \nu} \hphi_{k-nd}^{*, sl}$ for certain $\gamma_{d, \nu} \in \mathbb{C}$. Hence $\hphi_k = \sum_{d}  \hphi^{*, sl}_{k-nd}\otimes \left( \sum_{\nu}\gamma_{d, \nu}  \Xi_{d, \nu} \right)$.
\end{proof}

\paragraph{Bosonization} The operator $\hphi_{(n-1)}^{*, sl}(z)$ was calculated in \cite[Thm. 3.4]{K94}. Note that the parameter $q$ used in the $\emph{loc. cit.}$ corresponds to our parameter $v^{-1}$. To write the answer, we recall notation of $\emph{loc. cit.}$

Let $\bar{Q}$ and $\bar{P}$  be root and weight lattices for $\mathfrak{sl}_n$ respectively. We use notation $e^{\beta}$ for an element of group algebra $\mathbb{C}[\bar{P}]$, corresponding to $\beta \in \bar{P}$. Denote the fundamental weights of $\mathfrak{sl}_n$ by $\bar{\Lambda}_1, \dots, \bar{\Lambda}_{n-1}$. There is Heisenberg algebra in $\vsl$ generated by $a_j(k)$ for $k \in \mathbb{Z}$ and $j= 1, \dots, n-1$. Let $F^a$ be Fock module for the Heisenberg algebra. Then $F_{\mt}$ can be naturally identified with $F^a \otimes \mathbb{C}[\bar{Q}] e^{\bar{\Lambda}_{\mt}}$, and the action of $\vsl$ can be constructed explicitly, see \cite{FJ88} or \cite[Sect. 2.4]{K94}.

For $\alpha \in \bar{P}$, let us introduce operator $\partial_{\alpha} \left( w \otimes \beta \right)= (\alpha, \beta) w \otimes \beta$. There exists subalgebra $a_1^*(k)$ in the algebra generated by  $a_j(k)$. In the representation $F_{\mt}$, the operators satisfy

\begin{equation}
    [a_1^*(k), a_1^* (-k)] =  -\frac{[(n-1)k]_v}{k[nk]_v}
\end{equation}
\begin{prop}[\cite{K94}]
Vertex operator $\hphi_{(n-1)}^{*, sl}(z) \colon F_{\mt} \rightarrow F_{\mt-1}$ is given by the following formula
\begin{multline}
    \hphi_{(n-1)}^{*, sl}(z) = \exp \left( - \sum_{k=1}^{\infty} a_1^*(-k) v^{-\frac12 k} z^k \right) \exp \left( - \sum_{k=1}^{\infty}  a_1^*(k) v^{\frac32 k} z^{-k} \right)\\
    \times e^{- \bar{\Lambda}_1} \left( (-1)^{n-1} v^{-1} z \right)^{ - \partial_{\bar{\Lambda}_1} + \frac{n-\mt-1}{n}} v^{\mt} (-1)^{\mt n + \frac12 \mt (\mt+1)}
\end{multline}
\end{prop}
Normal ordered product $:\!\hphi_{(n-1)}^{*, sl}(z_1) \hphi_{(n-1)}^{*, sl}(z_2)\!:$ is an operator from $F_{\mt}$ to $F_{\mt-2}$ defined by the following formula
\begin{multline}
:\!\hphi_{(n-1)}^{*, sl}(z_1) \hphi_{(n-1)}^{*, sl}(z_2)\!: =  \exp \left( - \sum_{k=1}^{\infty} a_1^*(-k) v^{-\frac12 k} (z_1^k+z_2^k) \right) \exp \left( - \sum_{k=1}^{\infty}  a_1^*(k) v^{\frac32 k} (z_1^{-k}+z_2^{-k}) \right)\\
    \times e^{- 2\bar{\Lambda}_1}  \prod_{j=1,2} \left( (-1)^{n-1} v^{-1} z_j \right)^{ - \partial_{\bar{\Lambda}_1} + (n-\mt_j -1)/n} v^{\mt_j} (-1)^{\mt_j n + \frac12 \mt_j (\mt_j+1)}, 
\end{multline}
here $\mt_j = \mt + j-2$. Note that the normal ordering is not symmetric, namely
\begin{equation}
z_1^{-\frac1n}:\!\hphi_{(n-1)}^{*, sl}(z_1) \hphi_{(n-1)}^{*, sl}(z_2)\!:\, = z_2^{-\frac1n}:\!\hphi_{(n-1)}^{*, sl}(z_2) \hphi_{(n-1)}^{*, sl}(z_1)\!:. \label{interchanging phi star}
\end{equation}
The following relations can be checked directly
\begin{equation}
    \hphi_{(n-1)}^{*, sl}(z_1) \hphi_{(n-1)}^{*, sl}(z_2) = \left((-1)^{n-1} v^{-1} z_1 \right)^{\frac{n-1}{n}} \frac{(v^{2}z_2/z_1, v^{2n})_{\infty}}{(v^{2n}z_2/z_1, v^{2n})_{\infty}} :\!\hphi_{(n-1)}^{*, sl}(z_1) \hphi_{(n-1)}^{*, sl}(z_2)\!:. \label{phi star normal ordering}
\end{equation}

\subsection{Factorization of the vertex operator}  \label{appendix: relations}
We continue to study connection between vertex operators for $\mathfrak{sl}_n$ and $\hphi_{(\alpha)}^*(z)$. This subsection is devoted to a proof of Theorem \ref{thm: Bj Phik star}. The theorem is used for the proof of Proposition \ref{prop: heis and vertex operators}.

\begin{thm} \label{thm: Bj Phik star}
The following holds
\begin{equation}
    \hphi_{(\alpha)}^*(z) =   \hphi_{(\alpha)}^{*, \mathfrak{sl}_n}(z) \otimes  \exp \left( -\sum_{j>0} \frac{v^{2jn}}{j [n]_{v^j}^+} B_{-j} z^{j}\right)   \exp \left( \sum_{j>0} \frac{1}{j [n]_{v^j}^+} B_j z^{-j}\right).
\end{equation}
\end{thm}

To prove the theorem we need certain preparations. Let us define
\begin{equation}
\Phi_{(\alpha)}^*(z) = \sum_{k \in \mathbb{Z}} \Phi^*_{-\alpha+nk} z^{-k}
\end{equation}

\begin{lemma} \label{lemma: bilinear relation without hat}
The following relation holds
\begin{equation}
  (v^2 z_1 -  z_2)  \Phi_{(\alpha)}^* (z_1) \Phi_{(\alpha)}^* (z_2)     = (v^2 z_2 - z_1) \Phi_{(\alpha)}^* (z_2) \Phi_{(\alpha)}^* (z_1) 
\end{equation}
\end{lemma}
\begin{proof}
Consider operator 
\begin{align}
\left( 1 \otimes  \Phi^* \right) \circ \Phi^* \colon \bc_- \left(\mathbb{C}^n[Y^{\pm 1}] \right)^{\otimes N} \rightarrow& \left(\mathbb{C}^n[Y^{\pm 1}] \right) \otimes \left(\mathbb{C}^n[Y^{\pm 1}] \right) \otimes \bc_- \left(\mathbb{C}^n[Y^{\pm 1}] \right)^{\otimes N-2} \\
  \left( 1 \otimes  \Phi^* \right) \circ \Phi^*  w=& \sum_{k, l \in \mathbb{Z}} e_{l} \otimes e_{k} \otimes \Phi_{-k}^* \Phi_{-l}^* w
\end{align} 
Also, consider operator 
\begin{equation}
T_{12} \curvearrowright \left(\mathbb{C}^n[Y^{\pm 1}] \right) \otimes \left(\mathbb{C}^n[Y^{\pm 1}] \right) \otimes \bc_- \left(\mathbb{C}^n[Y^{\pm 1}] \right)^{\otimes N-2}
\end{equation}
induced from action of $T$ on first two tensor multiples. Recall that $T$ is given by \eqref{eq: T1}--\eqref{eq: T4}. The basic property of anti-symmetrizer \eqref{eq: basic property of (anti)symmetrizer} implies 
\begin{equation} \label{eq: T otimes 1 T anti symmetrizer}
T_{12} \circ \left( 1 \otimes  \Phi^* \right) \circ \Phi^*  = -v^{-1} \left( 1 \otimes  \Phi^* \right) \circ \Phi^*.
\end{equation}
Let $l>k$ and $l \equiv k \bmod{n}$. Consider in \eqref{eq: T otimes 1 T anti symmetrizer} the coefficients in front of $e_l \otimes e_k$ and $e_{l+n}\otimes e_{k-n}$
\begin{equation}
  v \Phi^*_{-l} \Phi^*_{-k} - (v - v^{-1}) \sum_{j=1}^{\infty} \Phi^*_{-k +nj} \Phi^*_{-l - nj} + (v-v^{-1}) \sum_{j=1}^{\infty} \Phi^*_{-l-nj} \Phi^*_{-k+ nj}  = - v^{-1} \Phi^*_{-k} \Phi^*_{-l}
\end{equation}
\begin{equation}
  v \Phi^*_{-l-n} \Phi^*_{-k+n} - (v - v^{-1}) \sum_{j=2}^{\infty} \Phi^*_{-k +nj} \Phi^*_{-l - nj} + (v-v^{-1}) \sum_{j=2}^{\infty} \Phi^*_{-l-nj} \Phi^*_{-k+ nj}  = - v^{-1} \Phi^*_{-k+n} \Phi^*_{-l-n}
\end{equation}
Hence
\begin{multline}
  v \left( \Phi^*_{-l} \Phi^*_{-k} - \Phi^*_{-l-n} \Phi^*_{-k+n} \right) - (v - v^{-1})  \Phi^*_{-k +n} \Phi^*_{-l - n} + (v-v^{-1})  \Phi^*_{-l-n} \Phi^*_{-k+ n}   \\ = - v^{-1} \left( \Phi^*_{-k} \Phi^*_{-l} - \Phi^*_{-k+n} \Phi^*_{-l-n} \right)
\end{multline}
Equivalently
\begin{equation}
  v  \Phi^*_{-l} \Phi^*_{-k} - v^{-1} \Phi^*_{-l-n} \Phi^*_{-k+n}     = - v^{-1}  \Phi^*_{-k} \Phi^*_{-l} +  v \Phi^*_{-k+n} \Phi^*_{-l-n}
\end{equation}
Substituting $l-n$ instead of $l$ and multiplying by $v$, we obtain
\begin{equation} \label{lemma phi star a symmetric}
  v^2  \Phi^*_{-l+n} \Phi^*_{-k} - \Phi^*_{-l} \Phi^*_{-k+n}     =  v^2 \Phi^*_{-k+n} \Phi^*_{-l} -  \Phi^*_{-k} \Phi^*_{-l+n} 
\end{equation}
To finish the proof we notice that \eqref{lemma phi star a symmetric} is symmetric on $l$ and $k$.
\end{proof}

\begin{prop} \label{prop: interchanging Phi alpha hat}
The following holds
\begin{equation} \label{eq: interchanging Phi alpha hat}
  (v^2 z_1 -  z_2)  \hphi_{(\alpha)}^* (z_1) \hphi_{(\alpha)}^* (z_2)     = (v^2 z_2 - z_1) \hphi_{(\alpha)}^* (z_2) \hphi_{(\alpha)}^* (z_1) 
\end{equation}
\end{prop}
\begin{proof}
Follows from Lemma \ref{lemma: bilinear relation without hat} since the operators $\Phi_k^*$ stabilize. 
\end{proof}

\begin{Remark}
    Proposition \ref{prop: interchanging Phi alpha hat} can be generalized. Namely, one can write interchanging relation for $\Phi_{\alpha}^*(z_1)$ and $\Phi_{\beta}^*(z_2)$. The result is \emph{$R$-matrix relation} \cite[eq.(2.17)]{DO94} \cite[eq.(6.31)]{JM95}. We will neither formulate nor use the relation. 
\end{Remark}

\begin{proof}[Proof of Theorem \ref{thm: Bj Phik star}]
Let us substitute \eqref{eq: factorization of vertex operator} to \eqref{eq: interchanging Phi alpha hat} for $\alpha = n-1$. Using relation \eqref{phi star normal ordering}, we obtain
\begin{multline} 
    (v^2 z_1- z_2) z_1^{\frac{n-1}{n}} \frac{(v^{2} z_2/z_1, v^{2n})_{\infty}}{(v^{2n}z_2/z_1, v^{2n})_{\infty}} :\!\hphi_{(n-1)}^{*, sl}(z_1) \hphi_{(n-1)}^{*, sl}(z_2)\!: \otimes \, \Xi(z_1) \Xi(z_2)  \\=  (v^2 z_2- z_1) z_2^{\frac{n-1}{n}}  \frac{(v^{2} z_1/z_2, v^{2n})_{\infty}}{(v^{2n}z_1/z_2, v^{2n})_{\infty}}  :\!\hphi_{(n-1)}^{*, sl}(z_2) \hphi_{(n-1)}^{*, sl}(z_1)\!: \otimes \, \Xi(z_2) \Xi(z_1) 
\end{multline}
Using \eqref{interchanging phi star}, one can see that
\begin{equation} \label{eq: bilinear relation on Xi}
    (v^2 z_1- z_2) z_1 \frac{(v^{2} z_2/z_1, v^{2n})_{\infty}}{(v^{2n}z_2/z_1, v^{2n})_{\infty}}  \Xi(z_1) \Xi(z_2)  =  (v^2 z_2- z_1) z_2  \frac{(v^{2} z_1/z_2, v^{2n})_{\infty}}{(v^{2n}z_1/z_2, v^{2n})_{\infty}}  \Xi(z_2) \Xi(z_1) 
\end{equation}
The relation $[B_{-j}, \hat{\Phi}^*_k] = -\hphi^*_{k-nj}$ implies
\begin{equation} \label{eq: Xi as Xi minus and exponent}
    \Xi(z) =    \Xi_- (z) \exp \left( \sum_{j>0} \frac{1}{j [n]_{v^j}^+} B_j z_1^{-j}\right),
\end{equation}
here $\Xi_-(z)$ is a formal power series, the coefficients are operators on $\Fheis$, and $[B_{-j}, \Xi_-(z)]=0$. Equivalently, $\Xi_-(z) = \sum_{\mu} \alpha_{\mu_1, \dots, \mu_j} B_{-\mu_1} \cdots B_{- \mu_j} z^{| \mu|}$. Denote  
\begin{equation}
\Xi_- \left[B_{-k}+ z_1^{-k} \right](z_2) = \sum_{\mu} \alpha_{\mu_1, \dots, \mu_j} \left(B_{-\mu_1} + z_1^{-\mu_1}\right) \cdots \left(B_{-\mu_j} + z_1^{-\mu_j} \right) z_2^{| \mu|}.
\end{equation}
Substituting  \eqref{eq: Xi as Xi minus and exponent} to \eqref{eq: bilinear relation on Xi}, we obtain
\begin{multline} 
  (v^2 z_1- z_2) z_1  \frac{(v^{2} z_2/z_1, v^{2n})_{\infty}}{(v^{2n}z_2/z_1, v^{2n})_{\infty}}  \Xi_-(z_1) \Xi_-\left[B_{-k} + z_1^{-k} \right](z_2)  \\=  (v^2 z_2- z_1) z_2   \frac{(v^{2} z_1/z_2, v^{2n})_{\infty}}{(v^{2n}z_1/z_2, v^{2n})_{\infty}}  \Xi_-(z_2) \Xi_- \left[B_{-k} + z_2^{-k} \right](z_1) 
\end{multline}
Consider expansion of the LHS in $z_2$. Note that only non-negative degrees in $z_2$ appear. Analogously, expansion of the RHS in $z_1$ has only non-negative degrees. Hence we can divide by $(v^2 z_1 - z_2)(v^2 z_2 -z_1)$ and obtain
\begin{equation} \label{eq: equation Xi minus with two elliptic function}
   \frac{(v^{2n+2} z_2/z_1, v^{2n})_{\infty}}{(v^{2n}z_2/z_1, v^{2n})_{\infty}}  \Xi_-(z_1) \Xi_-\left[B_{-k} + z_1^{-k} \right](z_2)  =   \frac{(v^{2n+2} z_1/z_2, v^{2n})_{\infty}}{(v^{2n}z_1/z_2, v^{2n})_{\infty}}  \Xi_-(z_2) \Xi_- \left[B_{-k} + z_2^{-k} \right](z_1) 
\end{equation}
Let us define
\begin{equation} \label{eq: notation xi tilde}
    \tilde{\Xi}_-(z) = \Xi_-(z) \times \exp \left( \sum_{j>0} \frac{v^{2jn}}{j [n]_{v^j}^+} B_{-j} z^{j}\right)
\end{equation}
We can substitute \eqref{eq: notation xi tilde} to \eqref{eq: equation Xi minus with two elliptic function}. Note that the exponent from \eqref{eq: notation xi tilde} is an invertible series. Since \eqref{eq: equation Xi minus with two elliptic function} has only positive degree in both $z_1$ and $z_2$, we can multiply both sides by the inverse to the exponents. We obtain
\begin{equation} \label{eq: tilde Xi bilinear relation}
    \tilde{\Xi}_-(z_1) \tilde{\Xi}_-\left[B_{-k} + z_1^{-k} \right](z_2)  =   \tilde{\Xi}_-(z_2) \tilde{\Xi}_- \left[B_{-k} + z_2^{-k} \right](z_1) 
\end{equation}
It is legitimate to divide by $\tilde{\Xi}_-(z_1) \tilde{\Xi}_-(z_2) $. \emph{A priori}, the result is a series in $z_1$ and $z_2$ with coefficients in rational function in $B_{-j}$. We obtain
\begin{equation}
   \frac{\tilde{\Xi}_-\left[B_{-k} + z_1^{-k} \right](z_2)}{ \tilde{\Xi}_-(z_2) }  =  \frac{\tilde{\Xi}_- \left[B_{-k} + z_2^{-k} \right](z_1)}{ \tilde{\Xi}_-(z_1) }
\end{equation}
On the RHS we have only positive powers of $z_1$ and negative powers of $z_2$, and vise versa for the LHS. Hence, the expression is a constant. Therefore $\tilde{\Xi}(z)$ is a constant.  Normalization condition \eqref{eq: sl vetex map norm} implies $\tilde{\Xi}(z)=1$.
\end{proof}
\bibliographystyle{alpha}
	\bibliography{bibtex}

		\noindent \textsc{Landau Institute for Theoretical Physics, Chernogolovka, Russia,\\
	Center for Advanced Studies, Skoltech, Moscow, Russia,\\
	National Research University Higher School of Economics, Moscow, Russia}

\emph{E-mail}:\,\,\textbf{mbersht@gmail.com}\\

\noindent\textsc{Center for Advanced Studies, Skoltech, Moscow, Russia,\\
	National Research University Higher School of Economics, Moscow, Russia}

\emph{E-mail}:\,\,\textbf{roma-gonin@yandex.ru}

\end{document}